\documentclass[oneside]{amsart}

\usepackage[dvipsnames]{xcolor}
\usepackage[utf8]{inputenc}
\usepackage[UKenglish]{babel}

\usepackage{amsmath}
\usepackage{mathtools}
\usepackage{amssymb}
\usepackage{amsthm}
\usepackage{thmtools}
\usepackage{enumitem}
\usepackage{hyperref}
\usepackage{cleveref}

\usepackage{subcaption}
\usepackage{pgfplotstable}
\usepackage{bbm}
\usepackage{tikz}
\usepackage{tikz-cd}
\usetikzlibrary{angles,quotes}

\crefname{equation}{}{}
\Crefname{equation}{Equation}{Equations}

\theoremstyle{plain}
\newtheorem{theorem}{Theorem}[section]
\newtheorem{proposition}[theorem]{Proposition}
\newtheorem{lemma}[theorem]{Lemma}
\newtheorem{corollary}[theorem]{Corollary}
\theoremstyle{remark}
\newtheorem{remark}[theorem]{Remark}
\theoremstyle{definition}
\newtheorem{definition}[theorem]{Definition}

\numberwithin{equation}{section}

\DeclarePairedDelimiter{\abs}{\lvert}{\rvert}
\DeclarePairedDelimiter{\norm}{\lVert}{\rVert}
 
\title{Existence and convergence of the area-constrained elastic flow}

\author[F.~Hauser]{Florian Hauser}
\email{fhauser@uni-bonn.de}

\author[C.~Scharrer]{Christian Scharrer}
\email{scharrer@iam.uni-bonn.de}

\author[A.~West]{Alexander West}
\email{west@iam.uni-bonn.de}

\address{Institute for Applied Mathematics, University of Bonn, Endenicher Allee 60, 53115 Bonn, Germany}

\begin{document}

\maketitle

\section*{Abstract}
We study the evolution of plane closed curves with fixed area moving by the negative $L^2$-gradient of their elastic energy. For smooth initial data, we
establish local and global existence of the flow. By imposing a simplicity assumption and an initial energy bound, we show that the length of the
evolving curve remains uniformly bounded. This yields subconvergence to a critical point, which is then improved to full convergence by utilizing a
\L{}ojasiewicz--Simon inequality. Conversely, an analysis of the energy profile curve, which maps a given length to the minimal energy among all curves
with that length and fixed area, reveals that the length diverges to infinity for initial data satisfying specific length and energy criteria. We visualize
our findings through numerical simulations.

\tableofcontents
\section{Introduction}
The Bernoulli model of an elastic rod, represented by a smooth immersed curve $\gamma: \mathbb{S}^1 \to \mathbb{R}^2$, aims to describe the bending behaviour through its stored elastic energy. This has led to an extensive study of minimizers and gradient flows, often subject to a length constraint, of this energy. In this article, we analyze the gradient flow subject to prescribed area and demonstrate local and global existence, as well as convergence under conditions on the initial data.

We denote the arc-length element $ds \coloneqq \abs{\partial_x \gamma} dx$ and the arc-length derivative $\partial_s \coloneqq \frac{\partial_x}{\abs{\partial_x \gamma}}$. This allows us to define the unit tangent vector $\tau \coloneqq \partial_s \gamma$. Throughout this work, we assume that closed embedded curves are positively oriented (traced counter-clockwise). We define the unit normal vector $\nu$ as the inner normal, obtained by a rotation $\nu \vcentcolon= (-\tau_2,\tau_1)$ of the tangent vector. The \emph{Frenet equations} describe the relations
\begin{equation*}
    \kappa \nu = \partial_s \tau,\qquad -\kappa \tau = \partial_s \nu,
\end{equation*}
where $\kappa$ denotes the curvature. Hence, the curvature is nonnegative for convex curves.
Let  
\begin{equation*}
	\mathcal{E}(\gamma) \coloneqq \int_{\mathbb{S}^1} \kappa^2 ds
\end{equation*}
be twice the elastic energy. This functional is not invariant under scaling: $\mathcal{E}(r \gamma) = r^{-1}\mathcal{E}(\gamma)$ for $r>0$. Consequently, the functional does not admit a global minimizer on the space of all closed curves. Thus, we study the variational problem under constraints. For instance, \textsc{Langer--Singer}~\cite{LANGER198575} proved that among curves of fixed length and winding number $1$, the circle is the unique minimizer.
However, the problem is also interesting from a dynamical perspective. Starting from an initial curve, we aim to deform it to reduce the elastic energy while keeping the length fixed. This evolution is modeled by the constrained $L^2(ds)$-gradient flow of the elastic energy. For a family of curves $\gamma:[0,T) \times \mathbb{S}^1 \to \mathbb{R}^2$ evolving according to the normal velocity $\partial_t \gamma = \xi \nu$ for $\xi \in C^\infty(\mathbb{S}^1, \mathbb{R})$ (tangential components represent reparametrizations and do not affect the energy), the time derivative of the energy is given by
\begin{equation*}
    \partial_t \mathcal{E}(\gamma) = \int_{\mathbb{S}^1} (2\partial_s^2 \kappa + \kappa^3)\xi ds.
\end{equation*}
Consequently, the unconstrained $L^2(ds)$-gradient reads
\begin{equation*}
    \nabla \mathcal{E}(\gamma) = \nabla_{L^2(ds)} \mathcal{E}(\gamma) = (2\partial_s^2 \kappa + \kappa^3)\nu.
\end{equation*}
Similarly, for the length functional  
\begin{equation*}
	\mathcal{L}(\gamma) = \int_{\mathbb{S}^1}ds,
\end{equation*}
there holds 
\begin{equation*}
	\nabla \mathcal{L}(\gamma) = -\kappa \nu.
\end{equation*}
To enforce length preservation, we define the flow as a linear combination of the gradients
\begin{equation*}
\begin{cases}
\begin{aligned}
    \partial_t \gamma &= -\nabla \mathcal{E}(\gamma) + \mu(\gamma) \nabla \mathcal{L}(\gamma) = -(2\partial_s^2 \kappa + \kappa^3 + \mu(\gamma) \kappa)\nu && \text{on } (0,T) \times \mathbb{S}^1,\\
    \gamma(0) &= \gamma_0 && \text{on } \mathbb{S}^1,
\end{aligned}
\end{cases}
\end{equation*}
for some $T \in (0,\infty]$.
Here, the time-dependent Lagrange multiplier $\mu$ is determined by the condition $\frac{d}{dt} \mathcal{L}(\gamma) = 0$. A short calculation yields
\begin{equation*}
    \mu(\gamma) = - \frac{\int_{\mathbb{S}^1}(2\partial_s^2 \kappa+\kappa^3)\kappa ds}{\mathcal{E}(\gamma)}.
\end{equation*}
\textsc{Dziuk--Kuwert--Schätzle} \cite{Dziuk2002} proved that the gradient flow exists globally in time and subconverges to a length-constrained elastica. Such an elastica is a critical point of the energy under the length constraint, satisfying the Euler-Lagrange equation
\begin{equation*}
    2 \partial_s^2 \kappa + \kappa^3 + \mu_\infty \kappa = 0
\end{equation*}
for some constant $\mu_\infty \in \mathbb{R}$. More recently, \textsc{Dall'Acqua--Pozzi--Spener}~\cite{Dall_Acqua_2016} and \textsc{Rupp}~\cite{RUPP2020108708} established full convergence by utilizing a \L{}ojasiewicz--Simon inequality.

One may also introduce further constraints, such as on the enclosed area. For a simple closed planar positively oriented curve $\gamma$, the enclosed area is given by
\begin{equation*}
    \mathcal{A}(\gamma)=-\frac{1}{2}\int_{\mathbb{S}^1}\nu \cdot \gamma ds.
\end{equation*}
For general immersed curves, this remains well defined and is often referred to as \emph{algebraic area}. The $L^2(ds)$-gradient is given by  
\begin{equation*}
	\nabla \mathcal{A}(\gamma) = -\nu.
\end{equation*}
Solutions to the Euler-Lagrange equation 
\begin{equation*}
    0 = -\nabla \mathcal{E}(\gamma) + \lambda_1 \nabla \mathcal{A}(\gamma) + \lambda_2 \nabla \mathcal{L}(\gamma)
\end{equation*}
subject to fixed length and area were analyzed by \textsc{Murai--Matsumoto--Yotsutani} \cite{MinoruMurai2013ConferencePublications}.
The corresponding dynamical problem was studied by \textsc{Okabe} \cite{okabe2007}. %In this work, also global existence and convergence were established.

A natural question arises: What happens if we omit the length constraint? \textsc{Bucur--Henrot} \cite{bucur2014newisoperimetricinequalityelasticae}, as well as \textsc{Ferone--Kawohl--Nitsch} \cite{ferone2015elasticaproblemareaconstraint}, proved the existence of a minimizer, namely the circle, among all simple curves with fixed area. The simplicity assumption is crucial, as there exist non-simple curves with fixed area but length diverging to infinity whose elastic energy tends to zero, see \Cref{fig:triple_eight_energy_to_zero}.

In this article, we address the associated dynamical problem. We study solutions to the equation
\begin{equation}
\label{eq:evolution_gradient_flow}
\begin{cases}
\begin{aligned}
    \partial_t \gamma &= -\nabla \mathcal{E}(\gamma) + \lambda(\gamma) \nabla \mathcal{A}(\gamma) = -(2\partial_s^2 \kappa + \kappa^3 + \lambda(\gamma)) \nu && \text{on } (0,T) \times \mathbb{S}^1,\\
    \gamma(0) &= \gamma_0 && \text{on } \mathbb{S}^1,
\end{aligned}
\end{cases}
\end{equation}
for some $T \in (0,\infty]$.
To ensure the area constraint $\frac{d}{dt} \mathcal{A}(\gamma) = 0$, the Lagrange multiplier is chosen as
\begin{equation}
\label{eq:a_lagrange_multiplier}
    \lambda(\gamma) = -\frac{\int_{\mathbb{S}^1} \kappa^3 ds}{\mathcal{L}(\gamma)}.
\end{equation}
The area constraint $\int_{\mathbb{S}^1}\nabla \mathcal{A}(\gamma) \cdot \partial_t \gamma ds = 0$ yields the energy identity
\begin{equation}
\label{eq:energy_identity}
    \partial_t \mathcal{E}(\gamma) = \int_{\mathbb{S}^1} \nabla \mathcal{E}(\gamma) \cdot \partial_t \gamma ds = -\int_{\mathbb{S}^1} |\partial_t \gamma|^2 ds,
\end{equation}
which implies that the elastic energy is non-increasing. Local and global existence of solutions can be shown by methods similar to \cite{dallacqua2017elasticflowcurveshyperbolic} and \cite{Dziuk2002}.
\begin{theorem}
\label{thm:global_existence}
    There exists a global (i.e. $T=\infty$) smooth solution $\gamma$ to \cref{eq:evolution_gradient_flow} for smooth immersed initial data $\gamma_0:\mathbb S^1\to\mathbb R^2$.
\end{theorem}
However, the asymptotic analysis poses challenges, primarily because there is no uniform bound on the length of the curve. As in the static case, curves may develop increasing length while the energy vanishes and the area is preserved. This behavior arises due to negative contributions to the area, for example of the middle 'belly' of \Cref{fig:triple_eight_energy_to_zero}.
\begin{figure}
    \centering
    \begin{tikzpicture}[scale=0.8]
        \draw[thick] (-3,1) arc (45:315:{sqrt(2)});
        \draw[thick] (-1,1) arc (135:45:{sqrt(2)});
        \draw[thick] (1,-1) arc (315:225:{sqrt(2)});
        \draw[thick] (3,-1) arc (-135:135:{sqrt(2)});
        \draw[thick] (-3,-1) -- (-1,1);
        \draw[thick] (-3,1) -- (-1,-1);
        \draw[thick] (3,-1) -- (1,1);
        \draw[thick] (3,1) -- (1,-1);
    \end{tikzpicture}
\caption{An example of a curve exhibiting three 'bellies'. If the radii of the circular segments are increased appropriately, the elastic energy of this configuration approaches zero while the total enclosed area remains constant.}
\label{fig:triple_eight_energy_to_zero}
\end{figure}
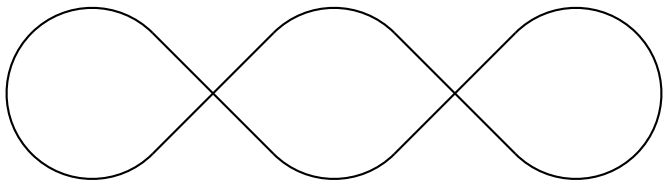
These negative contributions allow the curve's length to increase without bound, as the negative area of the middle 'belly' compensates for the gain in area of the outer 'bellies'.
To prevent this, we require restrictions on the initial data, relying on the 'optimal drop' result from \cite{bucur2014newisoperimetricinequalityelasticae}.

To formulate our main result, it is convenient to describe the curve geometry of a continuously differentiable immersion $\gamma :\mathbb S^1\to\mathbb R^2$ using the tangent angle $\theta$, representing the angle between the tangent vector and the horizontal axis. To be more precise, writing $\mathbb S^1$ as the interval $[0,2\pi]$ with identified endpoints, we choose $\theta$ according to \cite[Theorem 2.24]{MR1882174} satisfying
\begin{equation*}
	\dot \gamma = |\dot\gamma|(\cos \theta,\sin\theta) 
\end{equation*}
such that $\theta$ is continuous on $[0,2\pi)$ and $\theta(0) \in [0, 2\pi)$. Following \textsc{Bucur--Henrot} \cite{bucur2014newisoperimetricinequalityelasticae}, we define a drop as a simple, immersed curve $\gamma \in W^{2,2}([0,\pi], \mathbb{R}^2)$ with $\gamma(0)=\gamma(\pi)$ and $\theta(\pi)=\theta(0)+\pi$. The space $W^{2,2}$ is the natural choice, since the curvature is well defined and the elastic energy is finite.

We denote by $\gamma^*$ the optimal drop \cite{bucur2014newisoperimetricinequalityelasticae}, which is a minimizer of the functional $\mathcal{A} + \mathcal{E}$ in the class of drops. Let $\gamma^*_a$ be the rescaled version of $\gamma^*$ such that $\mathcal{A}(\gamma^*_a)=a$. The drop $\gamma^*_a$ is the minimizer of the elastic energy among all drops with area $a$.
Indeed, if there were a drop $\gamma$ with area $a$ and lower energy, the rescaling of $\gamma$ would contradict the optimality of $\gamma^*$.

The crucial idea for the convergence is an energy barrier argument. By restricting the initial energy, we ensure that the flow preserves the simplicity of the curve. Specifically, we compare the initial energy to the energy of a configuration consisting of two disjoint optimal drops. The minimal energy for such a 'split' configuration with total area $a$ is attained by two equal drops with area $\frac{a}{2}$. If the initial energy is strictly below this threshold, the curve is energetically prevented from forming self-intersections. Consequently, the curve remains simple along the flow. Although the length may theoretically diverge, we show that in this regime of preserved simplicity, a divergence of length would imply that the curve approaches a configuration with energy close to $2\mathcal{E}(\gamma_{\frac{a}{2}}^*)$, leading to a contradiction. The energy identity implies that the flow converges to an area-constrained elastica, that is, a solution to
\begin{equation*}
    2 \partial_s^2 \kappa + \kappa^3 + \lambda_\infty = 0
\end{equation*}
for some constant $\lambda_\infty \in \mathbb{R}$.
\begin{theorem}
\label{thm:convergence}
    Let $\gamma_0 \in C^\infty(\mathbb{S}^1, \mathbb{R}^2)$ be simple and immersed with
    \begin{equation*}
        \mathcal{A}(\gamma_0) =a,\quad
        \mathcal{E}(\gamma_0) < 2 \mathcal{E}(\gamma^*_{\frac{a}{2}}).
    \end{equation*}
    A translated (time-dependent) and reparametrized version $\tilde{\gamma}$ of the global, smooth and simple solution $\gamma$ to \cref{eq:evolution_gradient_flow} from \Cref{thm:global_existence} with initial data $\gamma_0$ converges smoothly to an area-constrained elastica $\gamma_\infty$.
\end{theorem}

It remains unclear whether the conditions specified in \Cref{thm:convergence} are the least restrictive necessary to ensure convergence. For example, one possible improvement would be to omit the simplicity assumption. In this case, we need to specify the winding number  
\begin{equation*}
	n_\gamma = \frac{1}{2\pi}\int_{\mathbb S^1}\kappa ds
\end{equation*}
of the curve $\gamma$, which is $1$ if the curve is simple. The analysis of the energy profile curve $E: [2\pi, \infty) \to [0, \infty)$
\begin{equation}
\label{eq:energy_profile_curve}
    E(\ell) \coloneqq \inf \bigl\{\mathcal{E}(\gamma) \mid \gamma \in W^{2,2}_{\text{Imm}}(\mathbb{S}^1,\, \mathbb{R}^2),\, n_\gamma=1,\, \mathcal{A}(\gamma) = \pi,\, \mathcal{L}(\gamma)=\ell\bigr\}
\end{equation}
reveals another strategy to obtain a length bound. The subscript 'Imm' restricts the space to immersed curves. The domain of the function starts at $2\pi$, a consequence of the isoperimetric inequality $\mathcal{A}(\gamma) \leq \frac{\mathcal{L}^2(\gamma)}{4\pi}$ (see \cite{osserman1978isoperimetric}), where equality is attained exclusively for the circle. We will demonstrate that this function is continuous (see \Cref{thm:energy_continuity}) and, as illustrated by the example in \Cref{fig:triple_eight_energy_to_zero}, $\mathcal{E}(\ell) \to 0$ as $\ell \to \infty$. Consequently, the energy profile curve attains its maximum at some $\ell^* \in [2\pi, \infty)$. As elaborated in \Cref{thm:energy_maximum}, the maximum is not achieved by the circle, proving $\ell^* \neq 2 \pi$. Employing an energy barrier argument once more, we establish a bound on the length, which yields convergence.

\begin{corollary}
\label{cor:energy_profile_convergence}
    Let $\gamma_0 \in C^\infty(\mathbb{S}^1, \mathbb{R}^2)$ be an immersed curve satisfying $n_\gamma=1$, $\mathcal{A}(\gamma_0) = \pi$, $\mathcal{L}(\gamma_0) < \ell^*$ and $\mathcal{E}(\gamma_0) < E(\ell^*)$. Then, the solution $\gamma$ to \cref{eq:evolution_gradient_flow} has length strictly bounded from above by $\ell^*$, and a reparametrized and translated (time-dependent) version $\tilde{\gamma}$ of the flow converges smoothly to an area-constrained elastica.
\end{corollary}

Moreover, this approach allows us to specify explicit conditions under which the length of the curve is guaranteed to diverge. If the length remains bounded, the flow converges to an area-constrained elastica. However, if the initial data $\gamma_0$ lies energetically below these area-constrained elasticae, convergence cannot occur, leading to a divergence in length. A detailed analysis, similar to \cite{bucur2014newisoperimetricinequalityelasticae}, reveals that each area-constrained elastica exhibits periodicity, with exactly one such elastica existing for each number of periods. It is found that, apart from the circle, the only area-constrained elasticae that could potentially lie on the energy profile curve are those with $2$ and $3$ periods. By employing a similar energy barrier argument as previously discussed, we can exclude convergence to the circle.

\begin{corollary}
\label{cor:length_divergence}
    Let $\gamma_2, \gamma_3$ be the area-constrained elasticae with two and three periods, respectively, and $\mathcal A(\gamma_2)=\mathcal A(\gamma_3)=\pi$.
    Let $\gamma_0 \in C^\infty(\mathbb{S}^1, \mathbb{R}^2)$ be immersed and such that $n_{\gamma_0}=1$, $\mathcal{A}(\gamma_0) = \pi$, and
    \begin{equation*}
        \mathcal{L}(\gamma_0) > \ell^*,\quad \mathcal{E}(\gamma_0) < \min\{E(\ell^*), \mathcal{E}(\gamma_2), \mathcal{E}(\gamma_3)\}.
    \end{equation*}
    Then, $\mathcal{L}(\gamma_t) \to \infty$ as $t \to \infty$ for the solution $\gamma$ to the area-constrained elastic flow \cref{eq:evolution_gradient_flow}.
\end{corollary}

Despite these results, several open questions remain. We believe that the circle and $\gamma_2$ are the only area-constrained elasticae that lie on the energy profile curve. Furthermore, we expect that the energy profile curve is differentiable at its maximum, and that this maximum is attained by $\gamma_2$. Should this be the case, it would provide a direct method to compute the explicit energy and length barriers required in \Cref{cor:length_divergence}.

Another natural extension of this work concerns the stability analysis of the area-constrained elastic flow. While such an analysis has recently been carried out for the circle under the free elastic flow \cite{Miura_2025}, our interest lies in the multi-loop configuration depicted in \Cref{fig:limit-shape}. Our numerical simulations suggest that this configuration exhibits unbounded growth while asymptotically preserving its shape; more precisely, we expect that the corresponding blowdown converges. However, rigorously proving this behavior presents a formidable challenge, especially considering that a comparable stability analysis remains an open problem even for the lemniscate of Bernoulli under the free elastic flow.

\begin{figure}[htbp]
    \centering
    \includegraphics[width=0.8\textwidth]{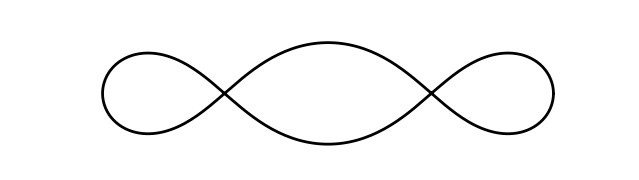}
\caption{A numerical picture of a shape conjectured to be stable under the area-constrained elastic flow.}
\label{fig:limit-shape}
\end{figure}

The article is organized as follows. In \Cref{sec:local_existence}, we establish the local existence of the area-constrained elastic flow. \Cref{sec:global_existence} is devoted to deriving the necessary energy estimates for long-time existence. In \Cref{sec:convergence}, we show subconvergence by proving a length bound using an analysis that relies on the 'optimal drop' result. Moreover, we demonstrate how a \L{}ojasiewicz--Simon inequality can be used to deduce full convergence. In \Cref{sec:energy_profile}, we analyze the energy profile curve, which also yields conditions for the convergence of the flow. Furthermore, we show that this analysis can be used to specify conditions under which the length of the flow is guaranteed to diverge. \Cref{sec:numerical_simulations} illustrates our findings with numerical simulations.

\section{Local existence}
\label{sec:local_existence}

Following an approach similar to \cite[Chapter~3, Section~2]{Rupp_Phd}, we introduce a tangential motion to transform \cref{eq:evolution_gradient_flow} into a quasilinear parabolic system, for which we then establish local existence using the Inverse Function Theorem. A comparable strategy was applied in \cite{dallacqua2017elasticflowcurveshyperbolic}.

To begin, we require an explicit formulation of \cref{eq:evolution_gradient_flow} in terms of the standard derivatives $\partial_x$.

\begin{proposition}[{\cite[Proposition~A.1]{Rupp2024}}]
\label{prop:a_explicit_formulation}
    Suppose $\gamma \in C^\infty(\mathbb{S}^1, \mathbb{R}^2)$ is immersed. Then, 
    \begin{enumerate}[label=(\arabic*)]
        \item $\vec{\kappa} \coloneqq \kappa \nu = \partial_s^2 \gamma = \frac{\partial_x^2 \gamma}{\abs{\partial_x \gamma}^2} - \frac{\partial_x^2 \gamma \cdot \partial_x \gamma}{\abs{\partial_x \gamma}^4} \partial_x \gamma = \frac{(\partial_x^2 \gamma)^\perp}{\abs{\partial_x \gamma}^2}$,
        \item $(\partial_s \kappa) \nu = \nabla_s \vec{\kappa} = \frac{\partial_x^3 \gamma}{\abs{\partial_x \gamma}^3} - \frac{\partial_x^3 \gamma \cdot \partial_x \gamma}{\abs{\partial_x \gamma}^5} \partial_x \gamma - 3 \frac{\partial_x^2 \gamma \cdot \partial_x \gamma}{\abs{\partial_x \gamma}^5} \partial_x^2 \gamma + 3 \frac{(\partial_x^2 \gamma \cdot \partial_x \gamma)^2}{\abs{\partial_x \gamma}^7} \partial_x \gamma$,
        \item \begin{equation*}
        \begin{split}
            (\partial_s^2 \kappa) \nu &= \nabla_s^2 \vec{\kappa} = \Bigg( \frac{\partial_x^{4}\gamma}{\abs{\partial_x \gamma}^{4}} -6 \frac{\partial_{x}^{2} \gamma \cdot \partial_x \gamma}{\abs{\partial_x \gamma}^{6}}\partial_x^{3}\gamma - 4 \frac{\partial_x^{3}\gamma \cdot \partial_x \gamma}{\abs{\partial_x \gamma}^{6}}\partial_x^{2}\gamma \\
			&\quad - 3 \frac{\abs{\partial_x^{2} \gamma}^{2}}{\abs{\partial_x \gamma}^{6}} \partial_x^{2} \gamma + 18 \frac{(\partial_x^{2}\gamma \cdot \partial_x \gamma)^{2}}{\abs{\partial_x \gamma}^{8}} \partial_x^{2}\gamma\Bigg)^{\perp},
        \end{split}
        \end{equation*}
        \item
        \label{prop:a_explicit_formulation_4}
        \begin{equation*}
            \nabla \mathcal{E}(\gamma) = \mathcal{T}(\gamma)^{\perp},
        \end{equation*}
        with
        \begin{equation*}
        \begin{split}
            \mathcal{T}(\gamma) &= 2\frac{\partial_x^{4}\gamma}{\abs{\partial_x \gamma}^{4}} - 12 \frac{\partial_{x}^{2} \gamma \cdot \partial_x \gamma}{\abs{\partial_x \gamma}^{6}}\partial_x^{3}\gamma - 8 \frac{\partial_x^{3}\gamma \cdot \partial_x \gamma}{\abs{\partial_x \gamma}^{6}}\partial_x^{2}\gamma\\
            &\quad- 5 \frac{\abs{\partial_x^{2} \gamma}^{2}}{\abs{\partial_x \gamma}^{6}} \partial_x^{2} \gamma + 35 \frac{(\partial_x^{2}\gamma \cdot \partial_x \gamma)^{2}}{\abs{\partial_x \gamma}^{8}} \partial_x^{2}\gamma,
        \end{split}
        \end{equation*}
    \end{enumerate}
    where $f^\perp = (f \cdot \nu) \nu$ and $\nabla_s f = (\partial_s f)^\perp$.
\end{proposition}

By adding the tangential motion $-(\mathcal{T}(\gamma) \cdot \tau) \tau$, the problem \cref{eq:evolution_gradient_flow} takes the form
\begin{equation}
\label{eq:evolution_gradient_flow_explicit}
\begin{cases}
\begin{aligned}
    \partial_t \gamma &= -\mathcal{T}(\gamma) - \lambda(\gamma)\nu\\
    &= \tilde{F}(\abs{\partial_x \gamma}^{-1}, \partial_x \gamma, \ldots, \partial_x^4 \gamma) - \lambda(\gamma)\nu && \text{on } [0,T) \times \mathbb{S}^1\\
    \gamma(0,\cdot) &= \gamma_0 && \text{on } \mathbb{S}^1,
\end{aligned}
\end{cases}
\end{equation}
for some smooth function $\tilde{F}: (0,\infty) \times (\mathbb{R}^2)^4 \to \mathbb{R}^2$.

It is sufficient to solve \cref{eq:evolution_gradient_flow_explicit}. Indeed, suppose $\gamma \in C^\infty([0,T)\times \mathbb{S}^1, \mathbb{R}^2)$ solves it. Define the reparametrized version $\tilde{\gamma} \coloneqq \gamma(t, \phi(t,x))$ for a family of diffeomorphisms $\phi(t,\cdot): \mathbb{S}^1 \to \mathbb{S}^1$ yet to be determined. Taking the derivative gives
\begin{equation*}
    \partial_t \tilde{\gamma}(t,x) = \partial_t \gamma(t,\phi(t,x)) + \partial_x \gamma(t,\phi(t,x)) \partial_t \phi(t,x).
\end{equation*}
We claim that there is a smooth solution $\phi: [0,T) \times \mathbb{S}^1 \to \mathbb{S}^1=\mathbb R/2\pi\mathbb Z$ of the ordinary differential equation
\begin{equation}\label{eq:reparametrization_flow}
\begin{cases}
        \partial_t \phi(t,x) &= \varphi(t,\phi(t,x)),\\
        \phi(0,x) &= x,
\end{cases}
\end{equation}
where $\varphi(t,x) \coloneqq \frac{\mathcal{T}(\gamma)(t,x)\cdot \tau(t,x)}{\abs{\partial_x \gamma(t,x)}}$. Observe from the initial condition that $\phi$ is orientation preserving.
Because $\lambda$ and $\mathcal{T}^\perp = \nabla \mathcal{E}$ are invariant under orientation preserving reparametrizations, $\tilde{\gamma}$ solves $\tilde{\gamma}(0,x) = \gamma_0(x)$ and
\begin{equation*}
\begin{split}
    \partial_t \tilde{\gamma}(t,x) &= (\partial_t \gamma)(t,\phi(t,x)) + (\mathcal{T}(\gamma)(t,\phi(t,x))\cdot \tau(t,\phi(t,x))) \tau(t,\phi(t,x))\\
    &= -\mathcal{T}^\perp(\gamma)(t,\phi(t,x)) - \lambda(\gamma)\nu(t,\phi(t,x))\\
    &= -\mathcal{T}^\perp(\tilde{\gamma})(t,x) - \lambda(\tilde{\gamma}) \tilde{\nu}(t,x),
\end{split}
\end{equation*}
where $\tilde{\nu}$ is the normal of $\tilde{\gamma}$.

Notice that by the global existence and uniqueness theorem for ordinary differential equations \cite[(7.6)]{MR1071170}, the maximal integral curve solving \eqref{eq:reparametrization_flow} corresponding to any given initial data $x\in\mathbb S^1$ exists on all of $[0,T)$. Thus, after extending $\varphi$ to $(-T,T)\times \mathbb S^1$ according to \cite{seeley1964extension}, the existence of a solution $\phi$ for \eqref{eq:reparametrization_flow} follows from \cite[Theorem~9.48]{lee2003smooth}. 

To solve \cref{eq:evolution_gradient_flow_explicit}, we require the following setup.

\begin{definition}[{\cite[Chapter~II.3.1]{Eidelman1998}}]
\label{def:par_hölder_spaces}
    For $s \in \mathbb{N}, \alpha \in (0,1)$, define the parabolic Hölder space $H^{\frac{s+\alpha}{4}, s+\alpha}([0,T]\times \mathbb{S}^1, \mathbb{R}^n)$ as the set of functions $\gamma: [0,T]\times \mathbb{S}^1 \to \mathbb{R}^n$, for which the norm
    \begin{equation*}
    \begin{split}
        \norm{\gamma}_{H^{\frac{s+\alpha}{4}, s+\alpha}} &\coloneqq \sum_{k+4l \leq s} \sup_{(t,x) \in [0,T]\times \mathbb{S}^1} \abs{\partial_x^k \partial_t^l \gamma(t,x)}\\
        &\quad+ \sum_{k+4l=s} \sup_{t \in [0,T],\,x\neq y \in \mathbb{S}^1} \frac{\abs{\partial_x^k \partial_t^l \gamma(t,x)-\partial_x^k \partial_t^l \gamma(t,y)}}{\operatorname{dist}(x,y)^\alpha}\\
        &\quad+ \sum_{0<s+\alpha-4l-k < 4} \sup_{t\neq \tilde{t} \in [0,T],\, x \in \mathbb{S}^1} \frac{\abs{\partial_x^k \partial_t^l \gamma(t,x)-\partial_x^k\partial_t^l \gamma(\tilde{t},x)}}{\abs{t-\tilde{t}}^{\frac{s+\alpha-4l-k}{4}}}
    \end{split}
    \end{equation*}
    is finite. Here, $\operatorname{dist}(x,y)$ denotes the geodesic distance between $x$ and $y$ on $\mathbb{S}^1$. 
\end{definition}

\begin{remark}
    Parabolic Hölder spaces are Banach spaces.
\end{remark}

\begin{lemma}
\label{lem:Hölder_norm_estimate}
    Let $s \in \{1,2,3\}$ and $v \in H^{\frac{s+\alpha}{4}, s+\alpha}([0,T]\times \mathbb{S}^1, \mathbb{R}^n)$ with $v(0,\cdot) = 0$ and $T \leq 1$. Then, there is a constant $c>0$ independent of $v$ such that for every $\tilde{s} \in \{0,\ldots,s-1\}$ we have
    \begin{equation*}
        \norm{v}_{H^{\frac{\tilde{s}+\alpha}{4}, \tilde{s}+\alpha}} \leq c T^{\frac{s+\alpha-\tilde{s}-1}{4}} \norm{v}_{H^{\frac{s+\alpha}{4}, s+\alpha}}.
    \end{equation*}
\end{lemma}

\begin{proof}
    We bound each term of the norm $\norm{v}_{H^{\frac{\tilde{s}+\alpha}{4}, \tilde{s}+\alpha}}$. For the supremum term, since $v(0,\cdot) = 0$, we have $\partial_x^k v(0,x) = 0$. The condition $k+4l \leq \tilde{s} \leq 2$ requires $l=0$, hence no time derivatives appear in the expression. Thus,
    \begin{equation*}
    \begin{split}
        \sum_{k+4l \leq \tilde{s}} \sup_{(t,x)} \abs{\partial_x^k \partial_t^l v(t,x)} 
        &\leq \sum_{k+4l \leq \tilde{s}} \sup_{(t,x)} \abs{t}^{\frac{s+\alpha-4l-k}{4}} \frac{\abs{\partial_x^k \partial_t^l v(t,x)-\partial_x^k \partial_t^l v(0,x)}}{\abs{t}^{\frac{s+\alpha-4l-k}{4}}}\\
        &\leq T^{\frac{s+\alpha-\tilde{s}}{4}} \norm{v}_{H^{\frac{s+\alpha}{4}, s+\alpha}}.
    \end{split}
    \end{equation*}
    For the spatial Hölder term, since the geodesic diameter of $\mathbb{S}^1$ is bounded, the Mean Value Theorem yields
    \begin{equation*}
    \begin{split}
        &\quad\sum_{k+4l=\tilde{s}} \sup_{t,x,y} \frac{\abs{\partial_x^k \partial_t^l v(t,x)-\partial_x^k \partial_t^l v(t,y)}}{\operatorname{dist}(x,y)^\alpha}\\
        &\leq c \sum_{k+4l=\tilde{s}} \sup_{(t,x)} \abs{\partial_x^{k+1} \partial_t^l v(t,x)}\\
        &\leq c \sum_{k+4l=\tilde{s}} \sup_{(t,x)} \abs{t}^{\frac{s+\alpha-4l-k-1}{4}} \frac{\abs{\partial_x^{k+1} \partial_t^l v(t,x)-\partial_x^{k+1} \partial_t^l v(0,x)}}{\abs{t}^{\frac{s+\alpha-4l-k-1}{4}}}\\
        &\leq c T^{\frac{s+\alpha-\tilde{s}-1}{4}} \norm{v}_{H^{\frac{s+\alpha}{4}, s+\alpha}},
    \end{split}
    \end{equation*}
    where the last inequality holds since $0 < s + \alpha - \tilde{s} - 1 < 4$.
    Finally, for the time Hölder term, we obtain
    \begin{equation*}
    \begin{split}
        &\quad\sum_{0<\tilde{s}+\alpha-4l-k < 4} \sup_{t,\tilde{t}, x} \frac{\abs{\partial_x^k \partial_t^l v(t,x)-\partial_x^k\partial_t^l v(\tilde{t},x)}}{\abs{t-\tilde{t}}^{\frac{\tilde{s}+\alpha-4l-k}{4}}}\\
        &= \sum_{0<\tilde{s}+\alpha-4l-k < 4} \sup_{t,\tilde{t}, x} \abs{t-\tilde{t}}^{\frac{s-\tilde{s}}{4}} \frac{\abs{\partial_x^k \partial_t^l v(t,x)-\partial_x^k\partial_t^l v(\tilde{t},x)}}{\abs{t-\tilde{t}}^{\frac{s+\alpha-4l-k}{4}}}\\
        &\leq T^{\frac{s-\tilde{s}}{4}} \norm{v}_{H^{\frac{s+\alpha}{4}, s+\alpha}}.
    \end{split}
    \end{equation*}
    Since $\alpha \in (0,1)$, the smallest exponent of $T$ extracted across all terms is $\frac{s+\alpha-\tilde{s}-1}{4}$.
\end{proof}

\begin{lemma}
\label{lem:hölder_interpolation}
    Let $\gamma \in H^{\frac{\alpha}{4}, \alpha}([0,T] \times \mathbb{S}^1, \mathbb{R}^n)$ and $0 < \beta < \alpha < 1$. Then, $\gamma \in H^{\frac{\beta}{4}, \beta}([0,T] \times \mathbb{S}^1, \mathbb{R}^n)$ and there is a constant $c>0$ such that we have
    \begin{equation*}
        \norm{\gamma}_{H^{\frac{\beta}{4}, \beta}} \leq c \norm{\gamma}_{C^0}^{1-\frac{\beta}{\alpha}} \norm{\gamma}_{H^{\frac{\alpha}{4}, \alpha}}^{\frac{\beta}{\alpha}}.
    \end{equation*}
\end{lemma}

\begin{proof}
    The bound on the supremum term in the parabolic Hölder norm holds trivially. For the spatial Hölder term we obtain
    \begin{equation*}
    \begin{split}
        \frac{\abs{\gamma(t,x)-\gamma(t,y)}}{\operatorname{dist}(x,y)^\beta} &\leq \abs{\gamma(t,x)-\gamma(t,y)}^{1-\frac{\beta}{\alpha}} \Bigg(\frac{\abs{\gamma(t,x)-\gamma(t,y)}}{\operatorname{dist}(x,y)^\alpha}\Bigg)^{\frac{\beta}{\alpha}}\\
        &\leq (2\norm{\gamma}_{C^0})^{1-\frac{\beta}{\alpha}} \norm{\gamma}_{H^{\frac{\alpha}{4}, \alpha}}^{\frac{\beta}{\alpha}}.
    \end{split}
    \end{equation*}
    The time Hölder term can be treated in the same way.
\end{proof}

\begin{corollary}
\label{cor:compact_hölder_embedding}
    Let $0 < \beta < \alpha < 1$. The embedding $H^{\frac{\alpha}{4}, \alpha}([0,T] \times \mathbb{S}^1, \mathbb{R}^n) \hookrightarrow H^{\frac{\beta}{4}, \beta}([0,T] \times \mathbb{S}^1, \mathbb{R}^n)$ is compact.
\end{corollary}

\begin{proof}
    This follows for instance from the Arzel\`a--Ascoli theorem.
\end{proof}

\subsection{The linearized problem}

Consider the following problem for $\gamma: [0,T] \times \mathbb{S}^1 \to \mathbb{R}^2$.

\begin{equation}
\label{eq:lin_equation}
\begin{cases}
\begin{aligned}
    L \gamma &= \partial_t \gamma - \sum_{k=0}^4 a_k \partial_x^k \gamma = \psi && \text{on } [0,T] \times \mathbb{S}^1\\
    \gamma(0,\cdot) &= \gamma_0 && \text{on } \mathbb{S}^1,
\end{aligned}
\end{cases}
\end{equation}
where $a_0,\ldots,a_4$ are $2\times2$ matrices.

We will see that $L$ is an isomorphism of Banach spaces under appropriate assumptions. We use the notation for Hölder spaces $C^{4+\alpha} = C^{\lfloor 4+\alpha \rfloor, 4+\alpha - \lfloor 4+\alpha \rfloor}$.

\begin{theorem}
\label{thm:sol_linearized_problem}
    Suppose $L$ is parabolic in the sense of Petrovski\v{\i} (cf. \cite[Definition~I.3]{Eidelman1998}) and $a_k \in H^{\frac{\alpha}{4}, \alpha}([0,T_0] \times \mathbb{S}^1, \mathbb{R}^{2\times 2})$ for $k=0,\ldots,4$ for some $\alpha \in (0,1)$. Then, for all $(\psi, \gamma_0) \in H^{\frac{\alpha}{4}, \alpha}([0,T_0] \times \mathbb{S}^1, \mathbb{R}^2) \times C^{4+\alpha}(\mathbb{S}^1, \mathbb{R}^2)$ there exists a unique $\gamma \in H^{\frac{4+\alpha}{4}, 4+\alpha}([0,T] \times \mathbb{S}^1, \mathbb{R}^2)$ for some $0 < T \leq T_0$ such that \cref{eq:lin_equation} is satisfied. Moreover, there exists a constant $c > 0$ independent of $(\psi, \gamma_0)$ such that the following estimate holds
    \begin{equation*}
        \norm{\gamma}_{H^{\frac{4+\alpha}{4}, 4+\alpha}([0,T] \times \mathbb{S}^1, \mathbb{R}^2)} \leq c(\norm{\psi}_{H^{\frac{\alpha}{4}, \alpha}([0,T] \times \mathbb{S}^1, \mathbb{R}^2)} + \norm{\gamma_0}_{C^{4+\alpha}(\mathbb{S}^1, \mathbb{R}^2)}).
    \end{equation*}
\end{theorem}

We prove this with a partition of unity and the analogue of this theorem on an interval. Consider the following problem for $\gamma: [0,T] \times [0,\pi]\to \mathbb{R}^2$

\begin{equation}
\label{eq:lin_equation_interval}
\begin{cases}
\begin{aligned}
    L \gamma &= \partial_t \gamma - \sum_{k=0}^4 a_k \partial_x^k \gamma = \psi && \text{on } [0,T] \times [0,\pi]\\
    \gamma(0,\cdot) &= \gamma_0 && \text{on } [0,\pi]\\
    \gamma(t,x) &= 0 &&\text{for } t \in [0,T], x \in \{0, \pi\},\\
    \partial_x \gamma(t,x) &= 0 &&\text{for } t \in [0,T], x \in \{0, \pi\}.
\end{aligned}
\end{cases}
\end{equation}
The Shapiro-Lopatinski\v{\i} condition (cf. \cite[p. 10-15]{Eidelman1998}) can be checked as in \cite[p. 2070]{Spener2017}.

Since our boundary conditions are identically zero, the required compatibility conditions of order $\lfloor\frac{4+\alpha}{4}\rfloor$ (see \cite[p.~217]{Eidelman1998}) ensure that the initial data matches the boundary data at the corners $(t,x) \in \{0\} \times \{0, \pi\}$. Specifically, the zero-order condition requires 
\begin{equation*}
	\gamma_0(x) = 0,\quad\partial_x \gamma_0(x) = 0\qquad\text{for $x \in \{0, \pi\}$}.
\end{equation*}
Because $\partial_t \gamma(t,x) = 0$ on the boundary, the first-order compatibility condition requires evaluating \cref{eq:lin_equation_interval} at $t=0$, yielding
\begin{equation*}
    \psi(0, x) + \sum_{k=0}^4 a_k(0, x) \partial_x^k \gamma_0(x) = 0 \quad \text{for } x \in \{0, \pi\}.
\end{equation*}
For $\alpha > 4$, higher-order compatibility conditions mandate that \cref{eq:lin_equation_interval} differentiated in time, evaluated at $t=0$, vanishes at the corners to match the higher time derivatives of the boundary condition.

We refer to \cite[Chapter~II.3.1]{Eidelman1998} for the definition of parabolic Hölder spaces on an interval instead of $\mathbb{S}^1$.

\begin{theorem}[{\cite[Theorem~VI.21]{Eidelman1998}}]
\label{thm:sol_linearized_problem_interval}
    Suppose $L$ is parabolic in the sense of Petrovski\v{\i} (see \cite[Definition~I.3]{Eidelman1998}) and $a_k \in H^{\frac{\alpha}{4}, \alpha}([0,T] \times [0,\pi], \mathbb{R}^{2\times 2})$ for $k=0,\ldots,4$ for some $\alpha>0, \alpha \notin \mathbb{N}$. Then, for $(\psi, \gamma_0) \in H^{\frac{\alpha}{4}, \alpha}([0,T] \times [0,\pi], \mathbb{R}^2) \times C^{4+\alpha}([0,\pi], \mathbb{R}^2)$ that satisfy the compatibility conditions of order $\lfloor\frac{4+\alpha}{4}\rfloor$, there exists a unique $\gamma \in H^{\frac{4+\alpha}{4}, 4+\alpha}([0,T] \times [0,\pi], \mathbb{R}^2)$ such that \cref{eq:lin_equation_interval} is satisfied. Moreover, there exists a constant $c > 0$ independent of $(\psi, \gamma_0)$ such that the following estimate holds
    \begin{equation*}
        \norm{\gamma}_{H^{\frac{4+\alpha}{4}, 4+\alpha}([0,T] \times [0,\pi], \mathbb{R}^2)} \leq c(\norm{\psi}_{H^{\frac{\alpha}{4}, \alpha}([0,T] \times [0,\pi], \mathbb{R}^2)} + \norm{\gamma_0}_{C^{4+\alpha}([0,\pi], \mathbb{R}^2)}).
    \end{equation*}
\end{theorem}

\begin{proof}[{Proof of \Cref{thm:sol_linearized_problem}}]
    We want to apply \Cref{thm:sol_linearized_problem_interval}. For this, consider the charts of $\mathbb{S}^1$
    \begin{equation}
    \label{eq:parametrizations_S_1}
    \begin{split}
        \phi_1: [0,\pi] &\to \{(x,y) \in \mathbb{S}^1 \mid y\geq0 \} = V_1 \subset \mathbb{S}^1,\\
        \phi_2: [0,\pi] &\to \{(x,y) \in \mathbb{S}^1 \mid x\leq0 \} = V_2 \subset \mathbb{S}^1,\\
        \phi_3: [0,\pi] &\to \{(x,y) \in \mathbb{S}^1 \mid y\leq0 \} = V_3 \subset \mathbb{S}^1,\\
        \phi_4: [0,\pi] &\to \{(x,y) \in \mathbb{S}^1 \mid x\geq0 \} = V_4 \subset \mathbb{S}^1,
    \end{split}
    \end{equation}
    where the charts are chosen to be isometries. Define the corresponding charts
    \begin{equation*}
        \tilde{\phi}_i: [0,T] \times [0,\pi] \to [0,T] \times V_i
    \end{equation*}
    on $[0,T] \times \mathbb{S}^1$. Choose a partition of unity $\chi_1+\ldots+\chi_4 = 1$ of $\mathbb S^1$ subordinate to $\{V_i^\circ\}$, where $V_i^\circ$ denotes the interior of $V_i$. Moreover, choose bump functions $\eta_i: \mathbb{S}^1 \to [0,1]$ with $\operatorname{supp}\eta_i \subset V_i^\circ$ and $\eta_i|_{\operatorname{supp}\chi_i}=1$. Then, for $(\tilde{\psi},\tilde{\gamma}_0) \in H^{\frac{\alpha}{4}, \alpha}([0,T] \times \mathbb{S}^1, \mathbb{R}^2) \times C^{4+\alpha}(\mathbb{S}^1, \mathbb{R}^2)$ consider the equation
    \begin{equation}
    \label{eq:u_i}
    \begin{cases}
        \partial_t u_i - \sum_{k=0}^4 a_k\circ \tilde{\phi}_i \partial_x^k u_i &= (\tilde{\psi} \eta_i) \circ \tilde{\phi}_i\\
        u_i(0,\cdot) &= (\tilde{\gamma}_0 \eta_i) \circ \tilde{\phi}_i\\
        u_i|_{\{0, \pi\}} &= 0,\\
        \partial_x u_i|_{\{0, \pi\}} &= 0.
    \end{cases}
    \end{equation}
    Since $\eta_i$ is compactly supported in $V_i^\circ$, the compatibility conditions are satisfied. We find with \Cref{thm:sol_linearized_problem_interval} a unique solution $u_i \in H^{\frac{4+\alpha}{4}, 4+\alpha}([0,T] \times [0,\pi], \mathbb{R}^2)$ of \cref{eq:u_i}. Moreover, there are constants $C_i = C_i(\alpha, T, a_k)$ such that
    \begin{equation}
    \label{eq:bound_norm_u_i}
    \begin{split}
        &\quad\norm{u_i}_{H^{\frac{4+\alpha}{4}, 4+\alpha}([0,T] \times [0,\pi], \mathbb{R}^2)}\\
        &\leq C_i (\norm{(\tilde{\psi} \eta_i)\circ \tilde{\phi}_i}_{H^{\frac{\alpha}{4}, \alpha}([0,T] \times [0,\pi], \mathbb{R}^2)} + \norm{(\tilde{\gamma}_0 \eta_i)\circ \tilde{\phi}_i}_{C^{4+\alpha}([0,\pi], \mathbb{R}^2)})\\
        &\leq C_i(\norm{\tilde{\psi} \eta_i}_{H^{\frac{\alpha}{4}, \alpha}([0,T] \times \mathbb{S}^1, \mathbb{R}^2)} + \norm{\tilde{\gamma}_0 \eta_i}_{C^{4+\alpha}(\mathbb{S}^1, \mathbb{R}^2)})\\
        &\leq C_i(\norm{\tilde{\psi}}_{H^{\frac{\alpha}{4}, \alpha}([0,T] \times \mathbb{S}^1, \mathbb{R}^2)} + \norm{\tilde{\gamma}_0}_{C^{4+\alpha}(\mathbb{S}^1, \mathbb{R}^2)}).
    \end{split}
    \end{equation}
    In the following, we will need the commutator $[L,f]v \coloneqq L(fv)-f(Lv)$ for functions $f\in H^{\frac{4+\alpha}{4}, 4+\alpha}([0,T] \times \mathbb{S}^1, \mathbb{R})$ and $v \in H^{\frac{4+\alpha}{4}, 4+\alpha}([0,T] \times \mathbb{S}^1, \mathbb{R}^2)$.
    First, consider solutions $u_i$ to \cref{eq:u_i} for $\tilde{\gamma}_0 = \gamma_0$ and $\tilde{\psi}=0$. Let $\tilde{\gamma}_1 \coloneqq \sum_{i=1}^4 \chi_i u_i \circ \tilde{\phi}_i^{-1}$. Since the $\phi_i$ are isometries, we have $\partial_x^k (u_i \circ \tilde{\phi}_i^{-1}) = \partial_x^k u_i \circ \tilde{\phi}_i^{-1}$. Hence,
    \begin{equation*}
    \begin{split}
        L \tilde{\gamma}_1 &= \sum_{i=1}^4 L(\chi_i u_i \circ \tilde{\phi}_i^{-1})\\
        &= \sum_{i=1}^4 \Big(\chi_i \partial_t u_i \circ \tilde{\phi}_i^{-1} - \chi_i \sum_{k=0}^4 a_k \partial_x^k u_i \circ \tilde{\phi}_i^{-1} + \big[L, \chi_i\big] u_i \circ \tilde{\phi}_i^{-1}\Big)\\
        &= \sum_{i=1}^4 (\chi_i \eta_i \tilde{\psi}) + E = E,
    \end{split}
    \end{equation*}
    where we define the error term $E \coloneqq \sum_{i=1}^4 \big[L, \chi_i\big] u_i \circ \tilde{\phi}_i^{-1}$.
    
    Now, we are left with solving \cref{eq:lin_equation} for initial data zero. Consider solutions $\tilde{u}_i$ to \cref{eq:u_i} with $\tilde{\gamma}_0=0$. We will choose $\tilde{\psi}$ later. Let $\tilde{\gamma}_2 \coloneqq \sum_{i=1}^4 \chi_i \tilde{u}_i \circ \tilde{\phi}_i^{-1}$. We have
    \begin{equation*}
        L \tilde{\gamma}_2 = \sum_{i=1}^4 (\chi_i \eta_i \tilde{\psi}) + \sum_{i=1}^4 \big[L, \chi_i\big] \tilde{u}_i \circ \tilde{\phi}_i^{-1} = \tilde{\psi} + K \tilde{\psi} = (I+K)\tilde{\psi},
    \end{equation*}
    where $I$ is the identity and $K$ is a linear operator defined by
    \begin{equation*}
        K \tilde{\psi} \coloneqq \sum_{i=1}^4 \big[L, \chi_i\big] \tilde{u}_i \circ \tilde{\phi}_i^{-1}. 
    \end{equation*}
    Here we need the initial data zero, otherwise the map $\tilde{\psi} \mapsto \tilde{u}_i$ is not linear. The commutator $[L, \chi_i]$ is a differential operator containing derivatives only up to order three and no time derivatives
    \begin{equation*}
        [L, \chi_i] = -\sum_{k=1}^4 \sum_{l=0}^{k-1} \binom{k}{l} a_k (\partial_x^{k-l} \chi_i)\partial_x^l.
    \end{equation*}
    Hence, we have the bound for some constant $c = c(a_1,\ldots,a_4,\chi_i) > 0$
    \begin{equation*}
        \norm{[L, \chi_i]\tilde{u}_i \circ \tilde{\phi}_i^{-1}}_{H^{\frac{\alpha}{4},\alpha}} \leq  c \sum_{k=0}^3\norm{\partial_x^k\tilde{u}_i}_{H^{\frac{\alpha}{4},\alpha}}.
    \end{equation*}
    By \Cref{lem:Hölder_norm_estimate} (more specifically, an analogue on an interval), which we can apply because $\tilde{\gamma}_0=0$, we have
    \begin{equation}\label{eq:time-Holder}
        \norm{[L, \chi_i]\tilde{u}_i \circ \tilde{\phi}_i^{-1}}_{H^{\frac{\alpha}{4},\alpha}} \leq c T^\alpha\sum_{k=0}^3\norm{\partial_x^k\tilde{u}_i}_{H^{\frac{1+\alpha}{4},1+\alpha}} \leq c T^\alpha \norm{\tilde{u}_i}_{H^{\frac{4+\alpha}{4},4+\alpha}}.
    \end{equation}
    Then, by \cref{eq:bound_norm_u_i} we may choose $T$ sufficiently small such that 
    \begin{equation}
    \label{eq:norm_K}
        \norm{K}_{H^{\frac{\alpha}{4}, \alpha}([0,T] \times \mathbb{S}^1, \mathbb{R}^2) \to H^{\frac{\alpha}{4}, \alpha}([0,T] \times \mathbb{S}^1, \mathbb{R}^2)} < 1.
    \end{equation}
    Hence, $I+K: H^{\frac{\alpha}{4}, \alpha}([0,T] \times \mathbb{S}^1, \mathbb{R}^2) \to H^{\frac{\alpha}{4}, \alpha}([0,T] \times \mathbb{S}^1, \mathbb{R}^2)$ is invertible.
    Consequently, for $\tilde{\psi} = (I+K)^{-1}(\psi-E)$, the function $\gamma\coloneqq\tilde{\gamma}_1 + \tilde{\gamma}_2$ solves $L\gamma = L (\tilde{\gamma}_1+\tilde{\gamma}_2) = E+\psi-E = \psi$. Since $\tilde{\gamma}_1(0,\cdot)=\gamma_0$ and $\tilde{\gamma}_2(0,\cdot)=0$, the inital condition is satisfied. \Cref{eq:bound_norm_u_i} applied to the solutions $u_i$ as well as $\tilde{u}_i$ yields

    \begin{equation*}
    \begin{split}
        \norm{\tilde{\gamma}_1}_{H^{\frac{4+\alpha}{4},4+\alpha}} &\leq c \sum_{i=1}^4 \norm{u_i}_{H^{\frac{4+\alpha}{4},4+\alpha}} \leq c \norm{\gamma_0}_{C^{4+\alpha}},\\
        \norm{\tilde{\gamma}_2}_{H^{\frac{4+\alpha}{4},4+\alpha}} &\leq c \sum_{i=1}^4 \norm{\tilde{u}_i}_{H^{\frac{4+\alpha}{4},4+\alpha}} \leq c\norm{(I+K)^{-1}(\psi-E)}_{H^{\frac{\alpha}{4},\alpha}}\\
        &\leq c \norm{(I+K)^{-1}}(\norm{\psi}_{H^{\frac{\alpha}{4},\alpha}}+\norm{E}_{H^{\frac{\alpha}{4},\alpha}})\\
        &\leq c (\norm{\psi}_{H^{\frac{\alpha}{4},\alpha}}+\norm{\gamma_0}_{C^{4+\alpha}})
    \end{split}
    \end{equation*}
    for some constant $c>0$, which proves the desired bound on the norm of $\gamma$.

    In order to prove uniqueness, it is enough to show that any solution of \eqref{eq:lin_equation_interval} satisfies the $H^{\frac{4+\alpha}{4}, 4+\alpha}$-estimate of \Cref{thm:sol_linearized_problem}. To this end, let $T$ be as in \eqref{eq:norm_K}, and $\gamma: [0,T]\times \mathbb S^1\to\mathbb S^1$ be a solution of \eqref{eq:lin_equation}. Then, $\chi_i\gamma$ solves
    \begin{equation}\label{eq:localized_linear_system}
    	\begin{cases}
    		\begin{aligned}
    			L_i((\chi_i \gamma) \circ \tilde{\phi}_i) &= (\chi_i L(\gamma))\circ \tilde{\phi}_i + ([L,\chi_i] \gamma) \circ \tilde{\phi}_i\\
    			&= (\chi_i \psi) \circ \tilde{\phi}_i + ([L,\chi_i] \gamma) \circ \tilde{\phi}_i &&\text{ on }[0,T] \times [0,\pi]\\
    			(\chi_i \gamma) \circ \tilde{\phi}_i(0,\cdot) &= (\chi_i \gamma_0) \circ \phi_i &&\text{ on }[0,\pi]\\
    			(\chi_i \gamma) \circ \tilde{\phi}_i(t,x) &= 0 &&\text{ for }x \in \{0,\pi\},\\
    			\partial_x((\chi_i \gamma) \circ \tilde{\phi}_i(t,x)) &= 0 &&\text{ for }x \in \{0,\pi\},
    		\end{aligned}
    	\end{cases} 
    \end{equation}
	where $L_i$ denotes the operator $L$ written in local coordinates on $[0,T] \times [0,\pi]$, that is, $L_i = \partial_t - \sum_{k=0}^4 a_k \circ \tilde{\phi}_i \partial_x^k$. \Cref{thm:sol_linearized_problem_interval} and Equation \eqref{eq:time-Holder} imply
	\begin{align*}
		\|(\chi_i\gamma)\circ\tilde\phi_i\|_{H^{\frac{4+\alpha}{4},4+\alpha}}&\le c(\|(\chi_i\psi)\circ\tilde\phi_i + ([L,\chi_i]\gamma)\circ\tilde\phi_i)\|_{H^{\frac{\alpha}{4},\alpha}} + \|(\chi_i\gamma_0)\circ\phi_i\|_{C^{4+\alpha}})\\
		&\le c(\|\psi\|_{H^{\frac{\alpha}{4},\alpha}}+T^\alpha\|\gamma\|_{H^{\frac{4+\alpha}{4},4+\alpha}}+\|\gamma_0\|_{C^{4+\alpha}}).
	\end{align*}
	Choosing $T>0$ small enough, it follows that 
	\begin{equation*}
		\|\gamma\|_{H^{\frac{4+\alpha}{4},4+\alpha}}\le c(\|\psi\|_{H^{\frac{\alpha}{4},\alpha}}+\|\gamma_0\|_{C^{4+\alpha}})
	\end{equation*}
	which implies the conclusion.
\end{proof}

So far, we have only shown that there is a solution for $\alpha \in (0,1)$. To obtain higher regularity, we apply \Cref{thm:sol_linearized_problem_interval} on four patches.

\begin{corollary}
\label{cor:sol_linearized_problem}
    Suppose $L$ is parabolic in the sense of Petrovski\v{\i} (see \cite[Definition~I.3]{Eidelman1998}) and $a_k \in H^{\frac{\alpha}{4}, \alpha}([0,T_0] \times \mathbb{S}^1, \mathbb{R}^{2\times 2})$ for $k=0,\ldots,4$ for some $\alpha > 0, \alpha \notin \mathbb{N}$. Then, for all $(\psi, \gamma_0) \in H^{\frac{\alpha}{4}, \alpha}([0,T_0] \times \mathbb{S}^1, \mathbb{R}^2) \times C^{4+\alpha}(\mathbb{S}^1, \mathbb{R}^2)$ there exists a unique $\gamma \in H^{\frac{4+\alpha}{4}, 4+\alpha}([0,T] \times \mathbb{S}^1, \mathbb{R}^2)$ for some $0 < T \leq T_0$ such that \cref{eq:lin_equation} is satisfied. Moreover, there exists a constant $c > 0$ independent of $(\psi, \gamma_0)$ such that the following estimate holds
    \begin{equation*}
        \norm{\gamma}_{H^{\frac{4+\alpha}{4}, 4+\alpha}([0,T] \times \mathbb{S}^1, \mathbb{R}^2)} \leq c(\norm{\psi}_{H^{\frac{\alpha}{4}, \alpha}([0,T] \times \mathbb{S}^1, \mathbb{R}^2)} + \norm{\gamma_0}_{C^{4+\alpha}(\mathbb{S}^1, \mathbb{R}^2)}).
    \end{equation*}
\end{corollary}

\begin{proof}
    By \Cref{thm:sol_linearized_problem}, we may assume $\alpha > 1$ and let $\alpha_0 \coloneqq \alpha - \lfloor \alpha \rfloor$. Then, \Cref{thm:sol_linearized_problem} yields a unique solution $\gamma \in H^{\frac{4+\alpha_0}{4}, 4+\alpha_0}([0,T] \times \mathbb{S}^1, \mathbb{R}^2)$ to \cref{eq:lin_equation}. It remains to show that $\gamma \in H^{\frac{4+\alpha}{4}, 4+\alpha}$ and that the estimate holds. 
   	Let $V_i,\phi_i,\tilde\phi_i,\chi_i$ be as in the proof of \Cref{thm:sol_linearized_problem}.
   	The function $\chi_i \gamma$ solves \eqref{eq:localized_linear_system}. 
    Since the commutator is a differential operator containing derivatives only up to order three and no time derivatives, the source term $(\chi_i \psi) \circ \tilde{\phi}_i + ([L,\chi_i] \gamma) \circ \tilde{\phi}_i$ belongs to $H^{\frac{1+\alpha_0}{4}, 1+\alpha_0}([0,T] \times [0,\pi], \mathbb{R}^2)$. Because $\chi_i$ is compactly supported strictly inside $V_i^\circ$, the source term and initial data vanish near the boundary points $x \in \{0,\pi\}$, meaning all higher-order compatibility conditions required by \Cref{thm:sol_linearized_problem_interval} are trivially satisfied. \Cref{thm:sol_linearized_problem_interval} gives $\chi_i \gamma \circ \tilde{\phi}_i \in H^{\frac{5+\alpha_0}{4}, 5+\alpha_0}([0,T] \times [0,\pi], \mathbb{R}^2)$ with
    \begin{equation*}
    \begin{split}
        &\quad \norm{\chi_i \gamma \circ \tilde{\phi}_i}_{H^{\frac{5+\alpha_0}{4}, 5+\alpha_0}}\\
        &\leq c(\norm{(\chi_i \psi) \circ \tilde{\phi}_i + ([L,\chi_i] \gamma) \circ \tilde{\phi}_i}_{H^{\frac{1+\alpha_0}{4},1+\alpha_0}} + \norm{(\chi_i \gamma_0)\circ \phi_i}_{C^{5+\alpha_0}})\\
        &\leq c(\norm{\psi}_{H^{\frac{1+\alpha_0}{4},1+\alpha_0}} + \norm{\gamma}_{H^{\frac{4+\alpha_0}{4},4+\alpha_0}} + \norm{\gamma_0}_{C^{5+\alpha_0}})\\
        &\leq c (\norm{\psi}_{H^{\frac{1+\alpha_0}{4},1+\alpha_0}} + \norm{\gamma_0}_{C^{5+\alpha_0}})
    \end{split}
    \end{equation*}
    for some constant $c>0$. It follows that $\gamma \in H^{\frac{5+\alpha_0}{4}, 5+\alpha_0}([0,T] \times \mathbb{S}^1, \mathbb{R}^2)$. Iterating this argument yields $\gamma \in H^{\frac{4+\alpha}{4}, 4+\alpha}([0,T] \times \mathbb{S}^1, \mathbb{R}^2)$.
\end{proof}

\subsection{The nonlinear problem}
In this chapter, we assume $\alpha \in (0,1)$.
Define
\begin{equation*}
    W \coloneqq \{\gamma \in C^4(\mathbb{S}^1, \mathbb{R}^2) \mid \gamma \text{ is immersed}\}.
\end{equation*}
Then, $W$ is an open subset of $C^4(\mathbb{S}^1, \mathbb{R}^2)$ and we assume $\gamma_0 \in W$. Since $\tilde{F}$ from \cref{eq:evolution_gradient_flow_explicit} is a polynomial componentwise with smooth coefficients, $\tilde{F}$ induces a smooth mapping
\begin{equation}
\label{eq:definition_F}
\begin{split}
    \mathcal{F}&: H^{\frac{4+\alpha}{4}, 4+\alpha}([0,T] \times \mathbb{S}^1, \mathbb{R}^2) \cap W \to H^{\frac{\alpha}{4}, \alpha}([0,T] \times \mathbb{S}^1, \mathbb{R}^2),\\
    \gamma &\mapsto \mathcal{F}(\gamma) \coloneqq \tilde{F}(\abs{\partial_x \gamma}^{-1}, \partial_x \gamma, \ldots, \partial_x^4 \gamma).
\end{split}
\end{equation}
We write $H^{\frac{4+\alpha}{4}, 4+\alpha}([0,T] \times \mathbb{S}^1, \mathbb{R}^2) \cap W$ for the subset of functions $\gamma \in H^{\frac{4+\alpha}{4}, 4+\alpha}([0,T] \times \mathbb{S}^1, \mathbb{R}^2)$ such that $\gamma(t,\cdot) \in W$ for all $0\leq t \leq T$. Computing the derivative of $\mathcal{F}$ at $\gamma \in H^{\frac{4+\alpha}{4}, 4+\alpha}([0,T] \times \mathbb{S}^1, \mathbb{R}^2) \cap W$ yields
\begin{equation}
\label{eq:derivative_F}
    D\mathcal{F}[\gamma]v = \sum_{k=0}^4 a_k \partial_x^k v.
\end{equation}
We have
\begin{equation*}
    a_4(t,x) = \begin{pmatrix}
        \frac{-2}{\abs{\partial_x \gamma}^4} & 0\\
        0 & \frac{-2}{\abs{\partial_x \gamma}^4}
    \end{pmatrix}
\end{equation*}
and $a_k(t,x) = \tilde{a}_k (\abs{\partial_x \gamma}^{-1}, \partial_x \gamma, \ldots, \partial_x^4 \gamma)$ for smooth functions $\tilde{a}_k: (0,\infty) \times (\mathbb{R}^2)^4 \to \mathbb{R}^{2 \times 2}$ that are polynomials componentwise. The computation is given in \cref{eq:computation_derivative_F}.
Moreover, let
\begin{equation}
\label{eq:definition_G}
\begin{split}
        \mathcal{G}&: H^{\frac{4+\alpha}{4}, 4+\alpha}([0,T] \times \mathbb{S}^1, \mathbb{R}^2) \cap W \to H^{\frac{2+\alpha}{4}, 2+\alpha}([0,T] \times \mathbb{S}^1, \mathbb{R}^2),\\
        \gamma &\mapsto \mathcal{G}(\gamma) \coloneqq -\lambda(\gamma)\nu_\gamma.
\end{split}
\end{equation}
As can be seen in \cref{eq:estimate_norm_derivative_G}, the norm can be estimated
\begin{equation*}
    \norm{D\mathcal{G}[\gamma]v}_{H^{\frac{\alpha}{4}, \alpha}} \leq c(\abs{\partial_x \gamma}^{-1}, \norm{\gamma}_{H^{\frac{4+\alpha}{4}, 4+\alpha}})(\norm{\partial_x v}_{H^{\frac{\alpha}{4},\alpha}}+\norm{\partial_x^2 v}_{H^{\frac{\alpha}{4},\alpha}}).
\end{equation*}
\Cref{lem:Hölder_norm_estimate} yields
\begin{equation}
\label{eq:estimate_derivative_G}
\begin{split}
    \norm{D\mathcal{G}[\gamma]v}_{H^{\frac{\alpha}{4}, \alpha}} &\leq c(\abs{\partial_x \gamma}^{-1}, \norm{\gamma}_{H^{\frac{4+\alpha}{4}, 4+\alpha}})\norm{v}_{H^{\frac{2+\alpha}{4},2+\alpha}}\\
    &\leq T^{\frac{\alpha}{4}}c(\abs{\partial_x \gamma}^{-1}, \norm{\gamma}_{H^{\frac{4+\alpha}{4}, 4+\alpha}}) \norm{v}_{H^{\frac{3+\alpha}{4}, 3+\alpha}}.
\end{split}
\end{equation}
Hence, the operator norm can be made arbitrarily small for fixed $\gamma$, if we choose $T$ small enough.

\begin{remark}
\label{rem:parabolicity}
    By continuity and compactness, for $\gamma \in H^{\frac{4+\alpha}{4}, 4+\alpha}([0,T] \times \mathbb{S}^1, \mathbb{R}^2)$ such that $\gamma(0,\cdot) \in W$, we find $\delta > 0$ such that $\delta \leq \frac{1}{\abs{\partial_x \gamma}} \leq \frac{1}{\delta}$ uniformly on $[0,T] \times \mathbb{S}^1$.  Hence, $L = \partial_t - \sum_{k=0}^{4} a_k \partial_x^k$ is parabolic in the sense of Petrovski\v{\i} for $a_k$ as in \cref{eq:derivative_F}.
\end{remark}
This allows us to solve \cref{eq:evolution_gradient_flow_explicit}.
\begin{theorem}
\label{thm:sol_nonlinear_problem}
    Let $\gamma_0 \in C^{4+\alpha}(\mathbb{S}^1, \mathbb{R}^2) \cap W$. Then, there exist $\varepsilon > 0$ and a unique solution $\gamma \in H^{\frac{4+\alpha}{4}, 4+\alpha}([0,\varepsilon] \times \mathbb{S}^1, \mathbb{R}^2) \cap W$ to
    \begin{equation*}
    \begin{cases}
        \partial_t \gamma &= \mathcal{F}(\gamma) + \mathcal{G}(\gamma)\\
        \gamma(0,\cdot) &= \gamma_0.
    \end{cases}
    \end{equation*}
\end{theorem}

\begin{proof}
    Let us first show existence. Notice that $\mathcal{F}(\gamma_0) + \mathcal{G}(\gamma_0) - D\mathcal{F}[\gamma_0]\gamma_0 \in H^{\frac{\alpha}{4}, \alpha}([0,T] \times \mathbb{S}^1, \mathbb{R}^2)$. Thus, we infer with \Cref{thm:sol_linearized_problem} that there exists a unique solution $\tilde{\gamma} \in H^{\frac{4+\alpha}{4}, 4+\alpha}([0,T] \times \mathbb{S}^1, \mathbb{R}^2)$ to the linear problem
    \begin{equation}
    \label{eq:sol_tilde_gamma}
    \begin{cases}
        \partial_t \tilde{\gamma} - D\mathcal{F}[\gamma_0]\tilde{\gamma} &= \mathcal{F}(\gamma_0) + \mathcal{G}(\gamma_0) - D\mathcal{F}[\gamma_0]\gamma_0,\\
        \tilde{\gamma}(0,\cdot) &= \gamma_0.
    \end{cases}
    \end{equation}
    For $T$ sufficiently small, we have $\tilde{\gamma}(t,\cdot) \in W$ for all $0\leq t \leq T$. Define $\tilde{f} \coloneqq \partial_t\tilde{\gamma} - \mathcal{F}(\tilde{\gamma}) - \mathcal{G}(\tilde{\gamma})$. Then, $\tilde{f} \in H^{\frac{\alpha}{4}, \alpha}([0,T] \times \mathbb{S}^1, \mathbb{R}^2)$. Continuity of $\mathcal{F}$ and $\mathcal{G}$ implies
    \begin{equation}
    \label{eq:tilde_f_time_zero}
        \tilde{f}(0,\cdot) = \partial_t\tilde{\gamma}|_{t=0} - \mathcal{F}(\gamma_0) - \mathcal{G}(\gamma_0) = \Big( \mathcal{F}(\gamma_0) + \mathcal{G}(\gamma_0) \Big) - \mathcal{F}(\gamma_0) - \mathcal{G}(\gamma_0) = 0.
    \end{equation}
    We apply the Inverse Function Theorem to the following spaces for $0<\beta<\alpha$
    \begin{equation*}
    \begin{split}
        \mathcal{X}_T &\coloneqq \{\eta \in H^{\frac{4+\beta}{4}, 4+\beta}([0,T] \times \mathbb{S}^1, \mathbb{R}^2) \mid \eta(0,\cdot) = 0\},\\
        \mathcal{Y}_T &\coloneqq H^{\frac{\beta}{4}, \beta}([0,T] \times \mathbb{S}^1, \mathbb{R}^2).
    \end{split}
    \end{equation*}
    Consider the map
    \begin{equation*}
        \Phi: \mathcal{X}_T \to \mathcal{Y}_T, \eta \mapsto \partial_t (\tilde{\gamma}+ \eta) - \mathcal{F}(\tilde{\gamma}+\eta) - \mathcal{G}(\tilde{\gamma}+\eta).
    \end{equation*}
    The map $\Phi$ is well defined and smooth on a neighborhood of $0$ and $\Phi(0) = \tilde{f}$. We deduce from \Cref{rem:parabolicity} and \Cref{thm:sol_linearized_problem} that $\partial_t - D\mathcal{F}[\tilde{\gamma}] \coloneqq L_{\tilde{\gamma}}$ is a linear isomorphism $L_{\tilde{\gamma}}: \mathcal{X}_T \to \mathcal{Y}_T$. Write
    \begin{equation*}
        D\Phi[0] = L_{\tilde{\gamma}}-D\mathcal{G}[\tilde{\gamma}] =  L_{\tilde{\gamma}}(I-L_{\tilde{\gamma}}^{-1}D\mathcal{G}[\tilde{\gamma}]),
    \end{equation*}
    where $L_{\tilde{\gamma}}^{-1}$ has bounded operator norm by \Cref{thm:sol_linearized_problem} (with a bound independent of $T$).
    Moreover, by reducing $T$, the operator norm of $D\mathcal{G}[\tilde{\gamma}]$ can be made arbitrarily small, as remarked in \cref{eq:estimate_derivative_G}. Hence, $I-L_{\tilde{\gamma}}^{-1}D\mathcal{G}[\tilde{\gamma}]$ is invertible. Therefore, the Inverse Function Theorem \cite[Chapter~3, Theorem~8]{Doria2021} yields the existence of neighborhoods $0 \in U \subset \mathcal{X}_T$ and $\tilde{f} = \Phi(0) \in V \subset \mathcal{Y}_T$ such that $\Phi: U \to V$ is a diffeomorphism. It remains to show $0 \in V$. For all $0 < \varepsilon < \min \{1, \frac{T}{2}\}$, define a cut-off function $ \phi_\varepsilon \in C^\infty([0,T])$ satisfying
    \begin{equation*}
        0 \leq \phi_\varepsilon \leq 1, \quad 0 \leq \partial_t \phi_\varepsilon \leq \frac{2}{\varepsilon} \quad \text{and} \quad \phi_\varepsilon(t) = \begin{cases}
            0, \quad 0 \leq t \leq \varepsilon,\\
            1, \quad 2 \varepsilon \leq t \leq T.
        \end{cases}
    \end{equation*}
    Let $f_\varepsilon \coloneqq \phi_\varepsilon \tilde{f} \in H^{\frac{\alpha}{4}, \alpha}$. By the proof of \cite[Lemma~2.5.8]{gerhardt2006curvature} (setting $\gamma=\frac{\alpha}{4}$), we obtain that the family $(f_\varepsilon)_\varepsilon$ is uniformly bounded in $H^{\frac{\alpha}{4}, \alpha}([0,T] \times \mathbb{S}^1, \mathbb{R}^2)$. By \Cref{cor:compact_hölder_embedding}, the embedding $H^{\frac{\alpha}{4}, \alpha}([0,T] \times \mathbb{S}^1, \mathbb{R}^2) \hookrightarrow H^{\frac{\beta}{4}, \beta}([0,T] \times \mathbb{S}^1, \mathbb{R}^2)$ is compact. Therefore, there exists $\hat{f} \in \mathcal{Y}_T$ with $f_{\varepsilon_k} \to \hat{f}$ in $\mathcal{Y}_T$ for some subsequence $\varepsilon_k \to 0$. On the other hand, we have
    \begin{equation*}
        \sup_{(t,x) \in [0,T] \times \mathbb{S}^1} \abs{f_\varepsilon - \tilde{f}} \leq \sup_{t \in [0,2\varepsilon]} \norm{\tilde{f}}_{C(\mathbb{S}^1, \mathbb{R}^2)}.
    \end{equation*}
    Since $\tilde{f}(0,\cdot) = 0$ by \cref{eq:tilde_f_time_zero}, we obtain $f_\varepsilon \to \tilde{f}$ uniformly. Hence, $\tilde{f} = \hat{f}$ and $f_{\varepsilon_k} \to \tilde{f}$ in $\mathcal{Y}_T$. Thus, there exists $\varepsilon \coloneqq \varepsilon_{k_0} > 0$ such that $f_\varepsilon \in V$. Since $\Phi: U \to V$ is a diffeomorphism, there exists $\eta \in U$ such that $f_\varepsilon = \Phi(\eta) = \partial_t (\tilde{\gamma} +\eta) - \mathcal{F}(\tilde{\gamma}+\eta)-\mathcal{G}(\tilde{\gamma}+\eta)$. Then, $\gamma \coloneqq \tilde{\gamma} + \eta \in H^{\frac{4+\beta}{4}, 4+\beta}([0,T] \times \mathbb{S}^1, \mathbb{R}^2)$ satisfies
    \begin{equation*}
    \begin{cases}
        \partial_t \gamma &= \mathcal{F}(\gamma) + \mathcal{G}(\gamma) + f_\varepsilon\\
        \gamma(0,\cdot) &= 0 + \tilde{\gamma}(0,\cdot) = \gamma_0.
    \end{cases}
    \end{equation*}
    Since $f_\varepsilon = 0$ on $[0,\varepsilon]$, we find that $\gamma$ solves \cref{eq:evolution_gradient_flow_explicit} on $[0,\varepsilon] \times \mathbb{S}^1$.
    
    We still need to show that $\gamma \in H^{\frac{4+\alpha}{4}, 4+\alpha}([0,\varepsilon] \times \mathbb{S}^1, \mathbb{R}^2)$. We can rewrite \cref{eq:evolution_gradient_flow_explicit} as a linear equation in $\gamma$ by freezing the argument $\gamma$ in the coefficients of $\partial_x \gamma, \ldots, \partial_x^4 \gamma$. Since the coefficients of the lower-order terms depend on derivatives of maximal order three, they lie in the space $H^{\frac{1+\beta}{4}, 1+\beta}([0,\varepsilon]\times \mathbb{S}^1, \mathbb{R}^{2 \times 2}) \subset H^{\frac{\alpha}{4}, \alpha}([0,\varepsilon]\times \mathbb{S}^1, \mathbb{R}^{2 \times 2})$. Note that the coefficient 
    $\begin{pmatrix}
        \frac{-2}{\abs{\partial_x \gamma}^4} & 0\\
        0 & \frac{-2}{\abs{\partial_x \gamma}^4}
    \end{pmatrix}$
    of the principal part $\partial_x^4 \gamma$ makes the problem uniformly parabolic by the compactness of $[0,T] \times \mathbb{S}^1$, as before. By \Cref{cor:sol_linearized_problem}, $\gamma \in H^{\frac{4+\beta}{4}, 4+\beta}([0,\varepsilon]\times \mathbb{S}^1, \mathbb{R}^{2})$ is the unique solution to this frozen linear problem with initial data $\gamma_0 \in C^{4+\alpha}(\mathbb{S}^1, \mathbb{R}^2)$ and coefficients in $H^{\frac{\alpha}{4}, \alpha}$. Consequently, \Cref{cor:sol_linearized_problem} yields the improved regularity $\gamma \in H^{\frac{4+\alpha}{4}, 4+\alpha}([0,\varepsilon] \times \mathbb{S}^1, \mathbb{R}^2)$, by potentially reducing $\varepsilon$.

    It remains to show uniqueness. Suppose there are two solutions $\gamma_1, \gamma_2$ to \cref{eq:evolution_gradient_flow_explicit} in the space $H^{\frac{4+\alpha}{4}, 4+\alpha}([0,T_i] \times \mathbb{S}^1, \mathbb{R}^2)$. Without loss of generality assume $0 < T_1 \leq T_2$. We want to show that the solutions coincide on the interval $[0,T_1]$.
    First, we show that they coincide on an interval $[0,T]$ for some $T>0$.
    
    Define $u \coloneqq \gamma_1 - \gamma_2$. It solves
    \begin{equation}
    \label{eq:sol_uniqueness}
    \begin{cases}
        &\quad\partial_t u - \mathcal{F}(\gamma_1) + \mathcal{F}(\gamma_2) - \mathcal{G}(\gamma_1) + \mathcal{G}(\gamma_2)\\
        &= (\partial_t - \int_0^1 D\mathcal{F}[u_r] dr - \int_0^1 D\mathcal{G}[u_r] dr)u = 0\\
        u(0,\cdot) &= 0,
    \end{cases}
    \end{equation}
    where $u_r \coloneqq r \gamma_1 + (1-r) \gamma_2$. Since $\gamma_0 \in W$, there is $c>0$ such that $\abs{\partial_x \gamma_0} \geq c$. Moreover, by continuity for $T>0$ sufficiently small, we find $\abs{\partial_x \gamma_i - \partial_x \gamma_0} \leq \frac{c}{2}$. The reverse triangle inequality yields
    \begin{equation*}
    \begin{split}
        \abs{\partial_x u_r} &= \abs{\partial_x \gamma_0 + r (\partial_x \gamma_1 - \partial_x \gamma_0) + (1-r)(\partial_x \gamma_2 - \partial_x \gamma_0)}\\
        &\geq c - r \frac{c}{2} - (1-r)\frac{c}{2} = \frac{c}{2}.
    \end{split} 
    \end{equation*}
    Therefore, $D\mathcal{F}[u_r], D\mathcal{G}[u_r]$ are well defined. Furthermore, $\partial_t-\int_0^1 D\mathcal{F}[u_r]dr$ is parabolic and, by \Cref{thm:sol_linearized_problem}, an isomorphism. Similarly to \cref{eq:estimate_derivative_G}, we can show that the operator norm of $\int_0^1 D\mathcal{G}[u_r]dr$ is arbitrarily small for $T>0$ small enough. Consequently, $\partial_t - \int_0^1 D\mathcal{F}[u_r] dr - \int_0^1 D\mathcal{G}[u_r] dr$ is an isomorphism and \cref{eq:sol_uniqueness} has the unique solution $u=0$.
    To conclude the argument, assume that $0 < T_s < T_1$ where
    \begin{equation*}
        T_s \coloneqq \sup\{t \in [0,T_1) \mid \gamma_1 = \gamma_2 \text{ in } H^{\frac{4+\alpha}{4}, 4+\alpha}([0,t] \times \mathbb{S}^1, \mathbb{R}^2)\}.
    \end{equation*}
    Define $v_0 \coloneqq \gamma_1(T_s) = \gamma_2(T_s)$. Since $\gamma_1$ is a solution to \cref{eq:evolution_gradient_flow_explicit} on $[0,T_1]$ and $T_s < T_1$, the function $\gamma_1(T_s)$ is an immersion. We may use the same argument to find $\delta > 0$ with $\gamma_1(T_s+\cdot)=\gamma_2(T_s+\cdot)$ in $H^{\frac{4+\alpha}{4}, 4+\alpha}([0,\delta] \times \mathbb{S}^1, \mathbb{R}^2)$, contradicting the maximality of $T_s$. 
\end{proof}

Using a similar bootstrap argument as in the end of the proof above, we deduce that higher regularity of the initial data $\gamma_0$ yields higher regularity of the solution. Note that this step requires \Cref{cor:sol_linearized_problem} instead of \Cref{thm:sol_linearized_problem}. 

\begin{corollary}
    If $\gamma_0 \in C^{4+m+\alpha}(\mathbb{S}^1, \mathbb{R}^2)$ for the initial data from \Cref{thm:sol_nonlinear_problem}, where $m \in \mathbb{N}$, $\alpha \in (0,1)$, then the solution $\gamma$ satisfies $\gamma \in H^{\frac{4+m+\alpha}{4}, 4+m+\alpha}([0,\varepsilon] \times \mathbb{S}^1, \mathbb{R}^2)$. In particular, if $\gamma_0 \in C^\infty(\mathbb{S}^1, \mathbb{R}^2)$, we have $\gamma \in C^\infty([0,\varepsilon] \times \mathbb{S}^1, \mathbb{R}^2)$.
\end{corollary}

\section{Global existence}
\label{sec:global_existence}
We argue by contradiction to establish the global existence of the area-constrained elastic flow. Suppose the maximal time of existence $T$ is finite. Using a compactness argument, we will show that the solution can be smoothly extended past $T$, contradicting its maximality. For the compactness argument, we need bounds on the solution $\gamma$ and all its derivatives. These bounds are obtained using Gagliardo-Nirenberg inequalities and the fact that the elastic energy is bounded, as in \cite{Dziuk2002}.

First, we present the following statements from the literature, which will be needed later in this chapter.
\begin{lemma}[{\cite[Lemma~2.1]{Dziuk2002}}]
    \label{lem:differential_properties}
    Suppose $\gamma \in C^2([0,T) \times \mathbb{S}^1, \mathbb{R}^2)$ is immersed and evolves by $\partial_t \gamma = \xi \nu + \varphi \tau$ for some $\xi, \varphi \in C^1([0,T), C^2(\mathbb{S}^1,\mathbb{R}))$. Then,
    \begin{enumerate}[label=(\arabic*)]
        \item \label{lem:differential_properties_i} $\partial_t(ds) = (\partial_s \varphi - \kappa \xi)ds$
        \item \label{lem:differential_properties_ii}$\partial_t \partial_s - \partial_s \partial_t = (\kappa \xi - \partial_s \varphi)\partial_s$
        \item \label{lem:differential_properties_iii}$\partial_t \tau = (\partial_s \xi + \varphi \kappa) \nu$
        \item $\partial_t\nu = - (\partial_s \xi + \varphi \kappa) \tau$
        \item $\partial_t \kappa = \partial_s^2 \xi + \kappa^2 \xi + \varphi \partial_s \kappa$.
    \end{enumerate}
\end{lemma}

\begin{lemma}[{\cite[Lemma~2.2]{Dziuk2002}}]
\label{lem:main_inequality}
    Suppose $\gamma \in C^2([0,T) \times \mathbb{S}^1, \mathbb{R}^2)$ is immersed with $\partial_t \gamma = \xi \nu$ and let $\phi \in C^5([0,T)\times \mathbb{S}^1, \mathbb{R})$. Moreover, let $Y= \partial_t \phi + 2\partial_s^4 \phi$. Then,
    \begin{equation*}
        \frac{1}{2} \frac{d}{dt} \int_{\mathbb{S}^1} \phi^2 ds + 2\int_{\mathbb{S}^1} (\partial_s^2 \phi)^2 ds = \int_{\mathbb{S}^1} Y \phi ds - \frac{1}{2} \int_{\mathbb{S}^1} \phi^2 \kappa \xi ds.
    \end{equation*}
    Furthermore, we have
    \begin{equation*}
        \partial_t \partial_s \phi + 2\partial_s^5 \phi = \partial_s Y + \xi \kappa \partial_s \phi.
    \end{equation*}
\end{lemma}

\begin{definition}[{\cite[Equation (2.13)]{Dziuk2002}}]
    For $\gamma \in C^\infty(\mathbb{S}^1, \mathbb{R}^2)$ immersed, we denote for $0\leq i \leq k$, $2 \leq p < \infty$
    \begin{equation*}
        \|\partial_s^i \kappa \|_p \coloneqq \mathcal{L}(\gamma)^{i+1-\frac{1}{p}} \left(\int_{\mathbb{S}^1} \vert \partial_s^i \kappa \vert^p ds\right)^\frac{1}{p}
    \end{equation*}
    and
    \begin{equation*}
        \norm{\partial_s^i \kappa}_{\infty} \coloneqq \mathcal{L}(\gamma)^{i+1} \norm{\partial_s^i \kappa}_{L^\infty}.
    \end{equation*}
    Define the norms for $2 \leq p < \infty$
    \begin{equation*}
        \|\kappa\|_{k,p} \coloneqq \sum_{i=0}^k \|\partial_s^i \kappa \|_p.
    \end{equation*}
\end{definition}

\begin{remark}\label{rem:def:norm:scaling}
    These norms are scale invariant since for $r>0$ it holds
    \begin{equation*}
    \begin{split}
        \|\partial_{s_{r\gamma}}^i \kappa_{r\gamma} \|_p &= \mathcal{L}(r\gamma)^{i+1-\frac{1}{p}} \left(\int_{\mathbb{S}^1} \vert \partial_{s_{r\gamma}}^i \frac{1}{r} \kappa_\gamma\vert^p ds_{r\gamma} \right)^\frac{1}{p}\\
        &= r^{i+1-\frac{1}{p}} \mathcal{L}(\gamma)^{i+1-\frac{1}{p}} \left(\int_{\mathbb{S}^1} r^{-ip-p+1} \vert \partial_{s_{\gamma}}^i \kappa_\gamma \vert^p ds_{\gamma} \right)^{\frac{1}{p}}\\
        &= \mathcal{L}(\gamma)^{i+1-\frac{1}{p}} \left(\int_{\mathbb{S}^1} \vert \partial_{s_{\gamma}}^i \kappa_\gamma \vert^p ds_{\gamma}\right)^\frac{1}{p}.
    \end{split}
    \end{equation*}
\end{remark}

\begin{lemma}[{\cite[Lemma~2.4]{Dziuk2002}}]
\label{lem:interpolation_inequality_1}
    Let $\gamma \in C^\infty(\mathbb{S}^1, \mathbb{R}^2)$ be immersed. Then, for any $k \in \mathbb{N}$, $p \geq 2$, $0 \leq i <k$ there is a constant $c=c(k,p)>0$ such that we have
    \begin{equation*}
        \|\partial_s^i \kappa \|_p \leq c \|\kappa\|_2^{1-\alpha} \|\kappa\|_{k,2}^\alpha,
    \end{equation*}
    where $\alpha = \frac{i+\frac{1}{2}-\frac{1}{p}}{k}$ and $c = c(p,k)$.
\end{lemma}

\begin{lemma}
\label{lem:interpolation_inequality_infinity}
    Let $\gamma \in C^\infty(\mathbb{S}^1, \mathbb{R}^2)$ be immersed. Then, for any $k \in \mathbb{N}$ there is a constant $c=c(k)>0$ such that we have
    \begin{equation*}
        \norm{\partial_s^k \kappa}_{\infty} \leq c \norm{\partial_s^k \kappa}_2^{\frac{1}{2}} (\norm{\partial_s^k \kappa}_{2} + \norm{\partial_s^{k+1} \kappa}_2 )^{\frac{1}{2}}.
    \end{equation*}
\end{lemma}

\begin{proof}
    We will use the identification $[0,1]$ of $\mathbb{S}^1$ in this proof. In view of \Cref{rem:def:norm:scaling}, the inequality is invariant under both scaling and reparametrization. Thus, we may assume that $\gamma$ has length $1$ and is parametrized by arc length. Hence, $\partial_{{s}} = \partial_x$ and $d{s} = dx$. Using \cite[Comments on Chapter~8, 1.(iii)]{brezis2011functional}, we obtain
    \begin{equation*}
    \begin{split}
        \norm{\partial_s^k \kappa}_{\infty} &= \norm{\partial_x^k {\kappa}}_{L^\infty}\\
        &\leq c \norm{\partial_x^k {\kappa}}_{L^2(dx)}^{\frac{1}{2}} \norm{\partial_x^k {\kappa}}_{W^{1,2}(dx)}^{\frac{1}{2}}\\
        &= c \norm{\partial_x^k {\kappa}}_{L^2(dx)}^{\frac{1}{2}} (\norm{\partial_x^k {\kappa}}_{L^2(dx)} + \norm{\partial_x^{k+1} {\kappa}}_{L^2(dx)})^{\frac{1}{2}}\\
        &= c \norm{\partial_{{s}}^k {\kappa}}_{L^2(d{{s}})}^{\frac{1}{2}} (\norm{\partial_{{s}}^k {\kappa}}_{L^2(d{{s}})} + \norm{\partial_{{s}}^{k+1} {\kappa}}_{L^2(d{{s}})})^{\frac{1}{2}}\\
        &= c \norm{\partial_{{s}}^k {\kappa}}_2^{\frac{1}{2}} (\norm{\partial_{{s}}^k {\kappa}}_2 + \norm{\partial_{{s}}^{k+1} {\kappa}}_2)^{\frac{1}{2}}.
    \end{split}
    \end{equation*}
\end{proof}

\begin{lemma}[{\cite[Equation~2.17]{Dziuk2002}}]
\label{lem:interpolation_inequality_2}
    For $\gamma \in C^\infty(\mathbb{S}^1, \mathbb{R}^2)$ immersed and $k \in \mathbb{N}$ we have
    \begin{equation*}
        \|\kappa\|_{k,2}^2 \leq c(k)(\|\partial_s^k \kappa \|_2^2+\|\kappa\|_2^2).
    \end{equation*}
\end{lemma}

We introduce new notation. For a function $\phi: \mathbb{S}^1 \to \mathbb{R}$, we denote by $P_\eta^\mu(\phi)$ any linear combination of terms of the type $\partial_s^{i_1} \phi\cdots \partial_s^{i_\eta}\phi$ with constant coefficients and $\mu = i_1 + \ldots + i_\eta$ is the total number of derivatives. Observe the following properties
\begin{equation*}
\begin{split}
    P_\eta^\mu(\phi) P_\beta^\alpha(\phi) &= P_{\eta+\beta}^{\mu+\alpha}(\phi),\\
    \partial_s P_\eta^\mu (\phi) &= P_\eta^{\mu+1} (\phi).
\end{split}
\end{equation*}

\begin{lemma}
\label{lem:a_evolution_equality}
    Suppose $\gamma$ evolves by $\partial_t \gamma =(-2\partial_s^2\kappa - \kappa^3 -\lambda)\nu$ where $\lambda \in \mathbb{R}$. For $m \in \mathbb{N}$, the derivatives $\phi_m = \partial_s^m \kappa$ satisfy
    \begin{equation*}
        \partial_t \phi_m + 2\partial_s^4 \phi_{m} = \lambda P_2^m(\kappa)+P_5^m(\kappa) + P_3^{m+2}(\kappa).
    \end{equation*}
    This holds even if $\lambda$ is time dependent.
\end{lemma}
\begin{proof}
    For $m=0$, \Cref{lem:differential_properties} gives us
    \begin{equation*}
    \begin{split}
        \partial_t \kappa &= \partial_s^2(-2\partial_s^2\kappa - \kappa^3-\lambda)+\kappa^2(-2\partial_s^2 \kappa - \kappa^3 - \lambda)\\
        &= -2\partial_s^4 \kappa + \lambda P_2^0(\kappa)+ P_5^0(\kappa) + P_3^2(\kappa).
    \end{split}
    \end{equation*}
    We proceed by induction on $m$. Again with \Cref{lem:differential_properties}, we have
    \begin{equation*}
    \begin{split}
        \partial_t \phi_{m+1} &= \partial_t \partial_s \phi_m = \partial_s \partial_t \phi_m + \kappa (-2\partial_s^2 \kappa - \kappa^3 - \lambda) \partial_s \phi_m\\
        &=\partial_s(-2\partial_s^4 \phi_m + \lambda P_2^m(\kappa)+ P_5^m(\kappa)+ P_3^{m+2}(\kappa))\\
        &\quad+P_3^{m+3}(\kappa) + P_5^{m+1}(\kappa)+\lambda P_2^{m+1}(\kappa)\\
        &= -2\partial_s^4 \phi_{m+1} + \lambda P_2^{m+1}(\kappa)+ P_5^{m+1}(\kappa) + P_3^{m+3}(\kappa).
    \end{split}
    \end{equation*}
\end{proof}

\begin{proposition}[{\cite[Proposition~2.5]{Dziuk2002}}]
\label{prop:inequality_P}
    Let $\gamma \in C^\infty(\mathbb{S}^1, \mathbb{R}^2)$ be immersed and $k \in \mathbb{N}$. Furthermore, let $P_\eta^\mu(\kappa)$ be a term with $\eta \geq 2$, which only contains derivatives of $\kappa$ of order at most $k-1$, and denote $\sigma=\frac{\mu+\frac{1}{2}\eta-1}{k}$.
    If $\mu+\frac{1}{2} \eta < 2k+1$, we have $\sigma<2$ and there exists $c=c(k,\mu,\eta) > 0$ such that for any $\varepsilon>0$
    \begin{equation*}
    \begin{split}
        \int_{\mathbb{S}^1} \vert P_\eta^\mu(\kappa)\vert ds &\leq \varepsilon \int_{\mathbb{S}^1} \vert\partial_s^k \kappa \vert^2 ds + c \varepsilon^{-\frac{\sigma}{2-\sigma}}\left(\int_{\mathbb{S}^1} \vert \kappa \vert^2 ds \right)^{\frac{\eta-\sigma}{2-\sigma}}\\
        &\quad + c\left(\int_{\mathbb{S}^1} \vert \kappa \vert^2 ds \right)^{\mu+\eta-1}.
    \end{split}
    \end{equation*}
\end{proposition}

To estimate the absolute value of $\lambda(\gamma) = - \frac{\int_{\mathbb{S}^1} \kappa^3 ds}{\mathcal{L}(\gamma)}$, we first establish a uniform lower bound on the length of the curve.

\begin{proposition}
\label{prop:length_bound_below}
    Suppose $\gamma: [0,T) \times \mathbb{S}^1 \to \mathbb{R}^2$ smoothly evolves according to \cref{eq:evolution_gradient_flow}. Then, there is a constant $c(\gamma_0) > 0$ such that
    \begin{equation*}
         c(\gamma_0) \leq \mathcal{L}(\gamma_t)
    \end{equation*}
    for all $t \in [0,T)$.
\end{proposition}

\begin{proof}
    As the circle minimizes the elastic energy among curves of fixed length by \cite[Corollary~1.8]{LANGER198575}, we obtain by \eqref{eq:energy_identity}
    \begin{equation*}
         \frac{4\pi^2}{\mathcal{L}(\gamma)} \leq \mathcal{E}(\gamma) \leq \mathcal{E}(\gamma_0).
    \end{equation*}
    Rearranging this inequality yields $\frac{4\pi^2}{\mathcal{E}(\gamma_0)} \leq \mathcal{L}(\gamma)$.
\end{proof}

\subsection{Uniform curvature bounds}

\begin{lemma}
\label{lem:a_bound_lambda}
    For a solution $\gamma$ to \cref{eq:evolution_gradient_flow}, $\lambda$ chosen as in \cref{eq:a_lagrange_multiplier}, and every $m \in \mathbb{N}$ there is a constant $c_m(\gamma_0)>0$ such that we have the bound
    \begin{equation*}
        \vert \lambda \vert \leq c_{m}(\gamma_0)(1+\|\partial_s^{m+2} \kappa \|_{L^2(ds)}^{\frac{1}{2m+4}}).
    \end{equation*}
\end{lemma}

\begin{proof}
    Denote $\mathcal{L} \coloneqq \mathcal{L}(\gamma)$. Estimate $\abs{\lambda}$ with \Cref{lem:interpolation_inequality_1,lem:interpolation_inequality_2}
    \begin{equation*}
    \begin{split}
        \vert \lambda \vert &\leq \mathcal{L}^{-1} \norm{\kappa}_{L^3(ds)}^3 = \mathcal{L}^{-3} \norm{\kappa}_3^3\\
        &\leq c_m \mathcal{L}^{-3} \norm{\kappa}_2^{3-\frac{1}{2m+4}} \norm{\kappa}_{m+2,2}^{\frac{1}{2m+4}}\\
        &\leq c_m \mathcal{L}^{-3} \norm{\kappa}_2^{\frac{6m+11}{2m+4}} (\norm{\kappa}_2^2 + \norm{\partial_s^{m+2} \kappa}_2^2)^{\frac{1}{4m+8}}\\
        &\leq c_m \mathcal{L}^{-3} \big(\norm{\kappa}_2^{3}+ \norm{\kappa}_2^{\frac{6m+11}{2m+4}}\norm{\partial_s^{m+2}\kappa}_2^{\frac{1}{2m+4}}\big)\\
        &\leq c_m \mathcal{L}^{-3} \big(\mathcal{L}^{\frac{3}{2}} \norm{\kappa}_{L^2(ds)}^3 + \mathcal{L}^{\frac{6m+11}{4m+8}}\norm{\kappa}_{L^2(ds)}^{\frac{6m+11}{2m+4}}\mathcal{L}^{\frac{2m+5}{4m+8}}\norm{\partial_s^{m+2}\kappa}_{L^2(ds)}^{\frac{1}{2m+4}}\big)\\
        &\leq c_m(\gamma_0)(\mathcal{L}^{-\frac{3}{2}}+\mathcal{L}^{-1}\norm{\partial_s^{m+2}\kappa}_{L^2(ds)}^{\frac{1}{2m+4}}).
    \end{split}
    \end{equation*}
    Here, we used that the energy is non-increasing along the gradient flow, meaning $\norm{\kappa}_{L^2(ds)}^2 = \mathcal{E}(\gamma) \leq \mathcal{E}(\gamma_0)$. The length bound from \Cref{prop:length_bound_below} yields the result.
\end{proof}

We will apply the previous statements to obtain bounds on the curvature and its derivatives. These bounds form the basis of the global existence proof.

\begin{lemma}
\label{lem:a_pre_main_theorem}
    Suppose $\gamma \in C^\infty([0,T) \times \mathbb{S}^1, \mathbb{R}^2)$ evolves according to \cref{eq:evolution_gradient_flow}.
    For all $\varepsilon > 0$ and $m \in \mathbb{N}$ there exists a constant $c_m(\gamma_0, \varepsilon)>0$ such that
    \begin{equation*}
    \begin{split}
        \frac{1}{2} \frac{d}{dt} \int_{\mathbb{S}^1} (\partial_s^m \kappa)^2 ds+ 2\int_{\mathbb{S}^1} (\partial_s^{m+2} \kappa)^2 ds
        &\leq 4\varepsilon \int_{\mathbb{S}^1} \vert \partial_s^{m+2} \kappa \vert^2 ds+c_m(\gamma_0, \varepsilon).
    \end{split}
    \end{equation*}
\end{lemma}
\begin{proof}
    From \Cref{lem:main_inequality} and \Cref{lem:a_evolution_equality} with $\phi_m = \partial_s^m \kappa$, it follows that
    \begin{equation}
    \begin{split}
    \label{eq:main_inequality}
        &\quad\frac{1}{2}\frac{d}{dt} \int_{\mathbb{S}^1} \phi_m^2 ds+ 2\int_{\mathbb{S}^1} \phi_{m+2}^2 ds\\
        &= \int_{\mathbb{S}^1} (\partial_t \phi_m+2\partial_s^4 \phi_m)\phi_m ds\\
        &- \int_{\mathbb{S}^1}\frac{1}{2} \phi_m^2 \kappa (-2\partial_s^2 \kappa-\kappa^3-\lambda)ds\\
        &= \int_{\mathbb{S}^1} \phi_m(\lambda P_2^m(\kappa)+P_5^m(\kappa) + P_3^{m+2}(\kappa))ds.
    \end{split}
    \end{equation}
    By using integration by parts, we can ensure that only derivatives of maximal order $m+1$ appear in the expression. Hence, we can use \Cref{prop:inequality_P} with $\mu=2m+2$, $\eta=4$, $k=m+2$, $\sigma=\frac{2m+3}{m+2}$:
    \begin{equation}
    \label{eq:inequality_P_1}
        \int_{\mathbb{S}^1} \phi_m P_3^{m+2}(\kappa)ds \leq \varepsilon \int_{\mathbb{S}^1} \vert \partial_s^{m+2} \kappa\vert^2 ds + c(\varepsilon)\left(\int_{\mathbb{S}^1} \vert \kappa \vert ^2 ds\right)^{2m+5}.
    \end{equation}
    Similarly, with $\mu=2m$, $\eta=3$, $k=m+2$, $\sigma=\frac{4m+1}{2m+4}$, one obtains:
    \begin{equation}
    \label{eq:inequality_P_2}
    \begin{split}
        &\int_{\mathbb{S}^1} \phi_m P_2^m(\kappa) ds \leq \varepsilon \int_{\mathbb{S}^1} \vert \partial_s^{m+2} \kappa \vert^2 ds + c \varepsilon^{-\frac{4m+1}{7}} \left(\int_{\mathbb{S}^1} \vert \kappa \vert^2 ds \right)^{\frac{2m+11}{7}}\\
        &\quad + c \left(\int_{\mathbb{S}^1} \vert \kappa \vert^2ds\right)^{2m+2};
    \end{split}
    \end{equation}
    and lastly with $\mu=2m$, $\eta=6$, $k=m+2$, $\sigma=\frac{2m+2}{m+2}$:
    \begin{equation}
    \label{eq:inequality_P_3}
        \int_{\mathbb{S}^1} \phi_m P_5^m(\kappa) ds \leq \varepsilon \int_{\mathbb{S}^1} \vert \partial_s^{m+2} \kappa \vert^2 ds + c(\varepsilon)\left(\int_{\mathbb{S}^1} \vert \kappa \vert^2 ds \right)^{2m+5}.
    \end{equation}
    Formally replacing $\varepsilon$ by $\frac{\varepsilon}{\abs{\lambda}}$ in \cref{eq:inequality_P_2} yields
    \begin{equation}
    \label{eq:inequality_P_4}
        \abs{\lambda}\Big\vert\int_{\mathbb{S}^1} \phi_m P_2^m(\kappa) ds\Big\vert \leq \varepsilon \|\partial_s^{m+2}\kappa\|_{L^2(ds)}^2 + c_m(\varepsilon, \gamma_0)(\abs{\lambda}^{\frac{4m+8}{7}}+\abs{\lambda}).
    \end{equation}
    If $\lambda = 0$, the term is already $0$.
    Combining \cref{eq:main_inequality,eq:inequality_P_1,eq:inequality_P_3,eq:inequality_P_4}, we have
    \begin{equation*}
    \begin{split}
        &\quad\frac{1}{2}\frac{d}{dt} \int_{\mathbb{S}^1} \phi_m^2 ds+ 2\int_{\mathbb{S}^1} \phi_{m+2}^2 ds\\
        &\leq 3\varepsilon \|\partial_s^{m+2}\kappa\|_{L^2(ds)}^2 + c_m(\varepsilon, \gamma_0)(\abs{\lambda}^{\frac{4m+8}{7}}+\abs{\lambda}+1).
    \end{split}
    \end{equation*}
    Recall \Cref{lem:a_bound_lambda} and observe for the exponents $\frac{4m+8}{7(2m+4)}=\frac{2}{7}<2$ and $\frac{1}{2m+4} < 2$. So we can use Young's inequality to arrive at
    \begin{equation*}
        \frac{1}{2}\frac{d}{dt} \int_{\mathbb{S}^1} \phi_m^2 ds+ 2\int_{\mathbb{S}^1} \phi_{m+2}^2 ds
        \leq 4\varepsilon \|\partial_s^{m+2}\kappa\|_{L^2(ds)}^2 + c_m(\varepsilon, \gamma_0).
    \end{equation*}
\end{proof}

\begin{lemma}
\label{lem:infinity_bound_kappa}
    For a smooth solution $\gamma: [0,T) \times \mathbb{S}^1 \to \mathbb{R}^2$ to \cref{eq:evolution_gradient_flow} and all $m \in \mathbb{N}$ we have bounds
    \begin{equation*}
        \| \partial_s^m \kappa \|_{L^\infty} \leq c_m(\gamma_0)
    \end{equation*}
    for some constant $c_m(\gamma_0) > 0$.
\end{lemma}

\begin{proof}
    Using \Cref{lem:interpolation_inequality_1}, \Cref{lem:interpolation_inequality_2}, and \Cref{prop:length_bound_below}, we find
    \begin{equation*}
    \begin{split}
        \norm{\partial_s^m \kappa}_{L^2(ds)} &= \mathcal{L}^{-m-\frac{1}{2}} \norm{\partial_s^m \kappa}_2\\
        &\leq c_m \mathcal{L}^{-m-\frac{1}{2}} \norm{\kappa}_2^{1-\frac{m}{m+2}} \norm{\kappa}_{m+2,2}^{\frac{m}{m+2}}\\
        &\leq c_m \mathcal{L}^{-m-\frac{1}{2}} \norm{\kappa}_2^{\frac{2}{m+2}} (\norm{\kappa}_2^{\frac{m}{m+2}}+\norm{\partial_s^{m+2}\kappa}_2^{\frac{m}{m+2}})\\
        &= c_m (\mathcal{L}^{-m} \norm{\kappa}_{L^2(ds)} + \mathcal{L}^{-m-\frac{1}{2}+\frac{1}{m+2}+\frac{(m+5/2)m}{m+2}} \norm{\partial_s^{m+2}\kappa}_{L^2(ds)}^{\frac{m}{m+2}}\|\kappa\|_{L^2(ds)}^{\frac{2}{m+2}})\\
        &\leq c_m(\gamma_0)(1+\norm{\partial_s^{m+2}\kappa}_{L^2(ds)}^{\frac{m}{m+2}})\\
        &\leq c_m(\gamma_0)(1+\norm{\partial_s^{m+2}\kappa}_{L^2(ds)}).
    \end{split}
    \end{equation*}
    This implies
    \begin{equation*}
    \begin{split}
        &\quad\frac{1}{2}\frac{d}{dt} \int_{\mathbb{S}^1} \vert \partial_s^{m} \kappa\vert^2 ds + \frac{1}{c_m(\gamma_0)}\int_{\mathbb{S}^1} \vert\partial_s^{m} \kappa\vert^2ds\\
        &\leq \frac{1}{2}\frac{d}{dt} \int_{\mathbb{S}^1} \vert \partial_s^{m} \kappa\vert^2 ds + \int_{\mathbb{S}^1} \vert\partial_s^{m+2} \kappa\vert^2ds + 1.
    \end{split}
    \end{equation*}
    Thus, from \Cref{lem:a_pre_main_theorem} with $\varepsilon=\frac{1}{4}$ it follows
    \begin{equation}
    \label{eq:ode_comparison}
    \begin{split}
        \frac{1}{2}\frac{d}{dt} \int_{\mathbb{S}^1} \vert \partial_s^{m} \kappa\vert^2ds + \frac{1}{c_m(\gamma_0)}\int_{\mathbb{S}^1} \vert\partial_s^{m} \kappa\vert^2ds &\leq c_m(\gamma_0).
    \end{split}
    \end{equation}
    Denote $f(t) \coloneqq \int_{\mathbb{S}^1} \vert \partial_s^{m} \kappa\vert^2ds$. Multiplying \cref{eq:ode_comparison} by $2e^{\frac{2}{c_m(\gamma_0)}t}$ gives an equation of the form
    \begin{equation*}
        \frac{d}{dt} \big(e^{\frac{2}{c_m(\gamma_0)}t}f(t)\big) \leq 2c_m(\gamma_0)e^{\frac{2}{c_m(\gamma_0)}t}.
    \end{equation*}
    By integrating in time and absorbing constants, we obtain
    \begin{equation}
    \label{eq:2_norm_kappa}
        \|\partial_s^{m} \kappa\|_{L^2(ds)} \leq c_m(\gamma_0).
    \end{equation}
    Using \Cref{lem:interpolation_inequality_infinity} and \Cref{prop:length_bound_below} gives
    \begin{equation*}
    \begin{split}
        \|\partial_s^m\kappa \|_{L^\infty} &= \mathcal{L}^{-m-1} \norm{\partial_s^m \kappa}_{\infty}\\
        &\leq c \mathcal{L}^{-m-1} \norm{\partial_s^m \kappa}_2^{\frac{1}{2}} (\norm{\partial_s^m \kappa}_2 + \norm{\partial_s^{m+1}\kappa}_2)^{\frac{1}{2}}\\
        &\leq c \mathcal{L}^{-m-1} (\norm{\partial_s^m \kappa}_2 + \norm{\partial_s^m \kappa}_2^{\frac{1}{2}} \norm{\partial_s^{m+1}\kappa}_2^{\frac{1}{2}})\\
        &= c \mathcal{L}^{-m-1}(\mathcal{L}^{m+\frac{1}{2}}\norm{\partial_s^m \kappa}_{L^2(ds)}+\mathcal{L}^{m+1} \norm{\partial_s^m \kappa}_{L^2(ds)}^{\frac{1}{2}}\norm{\partial_s^{m+1}\kappa}_{L^2(ds)}^{\frac{1}{2}})\\
        &\leq c(\gamma_0) (\norm{\partial_s^m \kappa}_{L^2(ds)}+\norm{\partial_s^m \kappa}_{L^2(ds)}^{\frac{1}{2}}\norm{\partial_s^{m+1} \kappa}_{L^2(ds)}^{\frac{1}{2}}).
    \end{split}
    \end{equation*}
    Because \cref{eq:2_norm_kappa} holds for all $m \in \mathbb{N}$, both $L^2$-norms on the right-hand side are uniformly bounded by a constant depending only on $\gamma_0$ and $m$, which completes the proof.
\end{proof}

\subsection{Proof of global existence}

\begin{proof}[Proof of \Cref{thm:global_existence}]
    Suppose $\gamma: [0,T) \times \mathbb{S}^1 \to \mathbb{R}^2$ is a maximal solution to \cref{eq:evolution_gradient_flow}. We seek bounds for all derivatives of $\gamma$, and start with $\norm{\gamma}_{L^\infty}$. From \cref{eq:a_lagrange_multiplier} and \Cref{lem:infinity_bound_kappa}, we obtain for $\xi(\gamma) \coloneqq -2\partial_s^2 \kappa - \kappa^3 - \lambda(\gamma)$,
    \begin{equation}
    \begin{split}
    \label{eq:bound_xi}
        \abs{\xi} &\leq 2 \abs{\partial_s^2 \kappa} + \abs{\kappa^3} + \abs{\lambda}\\
        &\leq c(\gamma_0).
    \end{split}
    \end{equation}
    Then, \cref{eq:bound_xi} shows
    \begin{equation*}
        \abs{\gamma} \leq \abs{\gamma_0} + \int_0^t \abs{\partial_t \gamma_u} du = \abs{\gamma_0} + \int_{0}^{t} \abs{\xi(\gamma_u)} du \leq c(\gamma_0, T).
    \end{equation*}
    For the first derivative $\norm{\partial_x\gamma}_{L^\infty}$, we deduce with \Cref{lem:differential_properties}
    \begin{equation*}
        \partial_t\vert\partial_x \gamma\vert = -\kappa \xi \abs{\partial_x \gamma} \leq c(\gamma_0)\vert \partial_x \gamma\vert.
    \end{equation*}
    Gr\"onwall's lemma gives $\vert \partial_x \gamma \vert \leq c(\gamma_0,T)$.
    Similarly,
    \begin{equation*}
        \partial_t\vert\partial_x \gamma\vert \geq -c(\gamma_0)\vert \partial_x \gamma\vert.
    \end{equation*}
    Again, Gr\"onwall's lemma yields
    \begin{equation}
    \label{eq:lower_bound_derivative}
        \vert\partial_x\gamma\vert \geq c^{-1}(\gamma_0, T).
    \end{equation}
    \Cref{lem:infinity_bound_kappa} together with $\partial_s^2 \gamma = \kappa \nu$ leads to bounds $\|\partial_s^m \gamma\|_{L^\infty} \leq c_m(\gamma_0)$ for $m \geq 2$, using the following observation
    \begin{equation}
    \label{eq:polynomial_identity_1}
        \partial_s^{m} \gamma = \partial_s^{m-2}\kappa \nu + P_1(\kappa,\ldots,\partial_s^{m-3} \kappa)\nu+P_2(\kappa,\ldots,\partial_s^{m-3} \kappa)\tau.
    \end{equation}
    This follows from the Leibniz rule and the Frenet equations. The polynomials $P_1, P_2$ have at most degree $m-1$.
    
    So far, we have established bounds for $\|\gamma\|_{L^\infty}$, $\|\partial_x \gamma\|_{L^\infty}$ and, using \Cref{lem:infinity_bound_kappa}, $\|\partial_s^m \gamma\|_{L^\infty}$ for $m\geq 2$. From this, we want to show bounds on $\partial_x^m\gamma$ for $m \geq 2$.
    First, note that for any function $h: \mathbb{S}^1 \to \mathbb{R}$, we have the identity
    \begin{equation}
    \label{eq:polynomial_identity_2}
        \partial_x^m h-\vert \partial_x \gamma\vert^m \partial_s^mh = P_m(|\partial_x\gamma|^{-1},\vert \partial_x \gamma\vert,\ldots,\partial_x^{m-1}\vert \partial_x \gamma\vert,h,\ldots,\partial_s^{m-1}h),
    \end{equation}
    where $P_m$ is a polynomial. This identity can be proven by induction.
    Assume inductively that $\|\partial_x^j \vert \partial_x \gamma\vert \|_{L^\infty} \leq c(j, \gamma_0, T)$ for $0\leq j \leq m-1$. We apply \cref{eq:polynomial_identity_2} to $h=\partial_s^i\kappa$ for all $i \in \mathbb{N}$. This gives bounds $\|\partial_x^m \partial_s^i\kappa\|_{L^\infty} \leq c(m,i, \gamma_0, T)$.
    From \Cref{lem:differential_properties}, we have
    \begin{equation*}
    \begin{split}
        \partial_t\vert\partial_x \gamma\vert &= \kappa(2\partial_s^2 \kappa+\kappa^3+\lambda)\vert \partial_x \gamma \vert,\\
        \partial_t \partial_x^m \vert\partial_x \gamma\vert &= \partial_x^m(\kappa(2\partial_s^2 \kappa+\kappa^3+\lambda)\vert \partial_x \gamma \vert)\\
        &= Q_m(\lambda, \kappa, \ldots, \partial_x^m \kappa, \partial_s^2 \kappa, \ldots, \partial_x^m \partial_s^2 \kappa, \abs{\partial_x \gamma}, \ldots, \partial_x^{m-1}\abs{\partial_x \gamma})\\
        &\quad + \kappa(2\partial_s^2 \kappa + \kappa^3 + \lambda) \partial_x^m \abs{\partial_x \gamma},
    \end{split}
    \end{equation*}
    where $Q_m$ is a polynomial. Taking the absolute value gives
    \begin{equation*}
        \abs{\partial_t \partial_x^m \abs{\partial_x \gamma}} \leq c_1(m, \gamma_0, T) + c_2(\gamma_0) \abs{\partial_x^m \abs{\partial_x \gamma}}.
    \end{equation*}
    Observe
    \begin{equation}
    \label{eq:abs_continuity}
    \begin{split}
        \abs{\partial_x^m \abs{\partial_x \gamma}} &= \Bigl|\partial_x^m \abs{\partial_x \gamma_0} + \int_0^t \partial_t \partial_x^m \abs{\partial_x \gamma_u} du\Bigr|\\
        &\leq \abs{\partial_x^m \abs{\partial_x \gamma_0}} + \int_0^t \abs{\partial_t \partial_x^m \abs{\partial_x \gamma_u}} du\\
        &\leq c_m(\gamma_0) + c_1(m, \gamma_0, T)t+c_2(\gamma_0) \int_0^t \abs{\partial_x^m \abs{\partial_x \gamma_u}} du.
    \end{split}
    \end{equation}
    Applying Gr\"onwall's inequality yields $\|\partial_x^m \vert \partial_x \gamma\vert \|_{L^\infty} \leq c(m, \gamma_0, T)$.
    Again, we apply identity \cref{eq:polynomial_identity_2} to $h=\gamma$ for $m \geq 2$ and obtain bounds $\|\partial_x^m\gamma\|_{L^\infty} \leq c_m(\gamma_0,T)$.

    Because of the bounded velocity $\|\partial_t \gamma\|_{L^\infty} = \norm{\xi}_{L^\infty} \leq c(\gamma_0,T)$, $\gamma_{t_i}$ is a Cauchy sequence for an arbitrary sequence $t_i \to T$ in $C^0(\mathbb{S}^1, \mathbb{R}^2)$. Therefore, it converges uniformly to $\gamma_T \in C^0(\mathbb{S}^1, \mathbb{R}^2)$. By Arzel\`a--Ascoli and the bounds on all derivatives of $\gamma$, we obtain smooth convergence and $\gamma_T \in C^\infty(\mathbb{S}^1, \mathbb{R}^2)$. The bound \cref{eq:lower_bound_derivative} ensures that $\gamma_T$ is immersed.
    
    Lastly, we need to show that the derivative from the left satisfies the gradient flow equation \cref{eq:evolution_gradient_flow}
    \begin{equation*}
    \begin{split}
        \lim_{t \to T^-} \frac{\gamma_T-\gamma_t}{T-t} &= \lim_{t \to T^-} \frac{\int_{t}^T \partial_t \gamma_u du}{T-t}\\
        &= \lim_{t \to T^-} \frac{\int_t^T \xi(\gamma_u) \nu_u du}{T-t} = \xi(\gamma_T) \nu_T,
    \end{split}
    \end{equation*}
    where the last limit is in the $C^\infty(\mathbb{S}^1, \mathbb{R}^2)$-topology.
    Thus, we can extend $\gamma$ smoothly to $[0,T]$ and even beyond by local existence. 
    Therefore, $\gamma: [0,T) \to \mathbb{R}^2$ is not a maximal solution, which contradicts our initial assumption.
\end{proof}
\section{Convergence}
\label{sec:convergence}
\subsection{Embeddedness}
As indicated in the numerical simulations in \Cref{sec:numerical_simulations} later on, we expect that if the initial curve is sufficiently close to the circle, the gradient flow will converge to the circle. Importantly, the gradient flow is expected to remain simple, i.e.\ without developing self-intersections. In the following, we give explicit conditions on the initial data under which the gradient flow stays simple for all time, leveraging a  known result on the optimal drop \cite{bucur2014newisoperimetricinequalityelasticae}. A priori, self-intersections might occur. If this is the case, we show that the first one is tangential.

For $[x],[y] \in \mathbb{R} / 2\pi\mathbb{Z} = \mathbb{S}^1$, we define the interval $[x,y] \subset \mathbb{R} / 2\pi\mathbb{Z} = \mathbb{S}^1$ by
\begin{equation*}
    [x,y] \coloneqq \{[z] \in \mathbb{R} / 2\pi\mathbb{Z} \mid z \in [x,y+2\pi k] \subset \mathbb{R}\}
\end{equation*}
with $k \in \mathbb{Z}$ minimal such that $y+2 \pi k \geq x$. For simplicity, from now on we write $x$ instead of $[x]$.

\begin{lemma}
\label{lem:a_tangential_self_intersections}
    Let $\gamma$ be the global smooth solution to \cref{eq:evolution_gradient_flow} for simple immersed initial data $\gamma_0 \in C^\infty(\mathbb{S}^1, \mathbb{R}^2)$. Suppose $\gamma$ develops self-intersections
    and denote $t_s = \inf \{t > 0 \mid \gamma \text{ is not simple}\}$. Then, $\gamma(t_s)$ is not simple and every self-intersection is tangential, i.e. for a self-intersection $\gamma(t_s,x) = \gamma(t_s,y)$, we have $\tau_{t_s}(x) = \pm \tau_{t_s}(y)$.
\end{lemma}
\begin{proof}
    Note that by \cite[Lemma~4.1, Lemma~4.3]{Mueller_2021} the set of simple and immersed $C^1(\mathbb{S}^1, \mathbb{R}^2)$-curves is open in $C^1(\mathbb{S}^1, \mathbb{R}^2)$. If $\gamma(t_s, \cdot)$ were simple, then $\gamma(t,\cdot)$ would remain simple for a small time interval $[t_s, t_s+\delta]$. This contradicts the definition of $t_s$. Hence, $\gamma(t_s, \cdot)$ is not simple.
    Let $x_0,y_0 \in \mathbb{S}^1$ be an arbitrary self-intersection, meaning $x_0 \neq y_0$ and $\gamma(t_s, x_0)=\gamma(t_s, y_0)$. Suppose that it is not tangential, i.e. $\tau_{t_s}(x_0) \neq \pm \tau_{t_s}(y_0)$. By continuous differentiability in time, we obtain a contradiction
    to $t_s = \inf \{t > 0 \mid \gamma \text{ is not simple}\}$. Indeed, consider the function
    \begin{equation*}
    \begin{split}
        &f:[t_s-\varepsilon, t_s+\varepsilon] \times [x_0-\varepsilon, x_0+\varepsilon] \times [y_0-\varepsilon, y_0+\varepsilon] \to \mathbb{R}^2,\\
        &f(t,x,y) = \gamma(t,x)-\gamma(t,y).
    \end{split}
    \end{equation*}
    Compute
    \begin{equation*}
        Df(t,x,y) = \begin{pmatrix}
            \partial_t \gamma(t,x)-\partial_t \gamma(t,y) & \partial_x \gamma(t,x) & -\partial_y \gamma(t,y)
        \end{pmatrix}.
    \end{equation*}
    The differential $D_{(x,y)}f$ with respect to the spatial variables is invertible if and only if $\partial_x \gamma(t,x)$ and $\partial_y \gamma(t,y)$ are linearly independent. This is the case if the tangents are not parallel. Then, we obtain by the Implicit Function Theorem a continuous map $\varphi$ that maps a neighborhood of $t_s$ to a neighborhood of $(x_0,y_0)$ with $f(t,\varphi(t))=0$. In particular, there has to be $t_0 < t_s$ such that $\gamma(t_0, \cdot)$ has a self-intersection, a contradiction.
\end{proof}

We check the necessary conditions to apply the proof of \cite[Theorem~3.1]{bucur2014newisoperimetricinequalityelasticae}. The following Lemma is required to compare the energy of a curve to the energies of two drops with smaller combined area.

\begin{lemma}
\label{lem:a_geometric_information}
    Let $\gamma \in W^{2,2}(\mathbb{S}^1, \mathbb{R}^2)$ be immersed and simple. Let $x,y \in \mathbb{S}^1$. If one of the following conditions
    \begin{enumerate}[label=(\arabic*)]
        \item $0 \notin [x,y]$ and $\theta(y) \leq \theta(x)-\pi$,
        \item $0 \in [x,y]$ and $\theta(y) \leq \theta(x)-3\pi$
    \end{enumerate}
    holds, then there exist drops $\gamma_1,\gamma_2$ with
    \begin{equation*}
    \begin{split}
        \mathcal{E}(\gamma_1)+\mathcal{E}(\gamma_2) &\leq \mathcal{E}(\gamma),\\
        \mathcal{A}(\gamma_1)+\mathcal{A}(\gamma_2) &\leq \mathcal{A}(\gamma).
    \end{split}
    \end{equation*}
\end{lemma}

\begin{proof}
    The proof is essentially the same as the proof of \cite[Lemma~3.3]{bucur2014newisoperimetricinequalityelasticae}. The idea is to extend the segment $\gamma([x,y])$ by straight line segments such that it touches the remaining curve. In the newly constructed curve, we can identify two drops with the desired properties. The second condition accounts for the $2\pi$ shift of $\theta$ if $0 \in [x,y]$.
\end{proof}

\begin{lemma}
\label{lem:min_two_drops}
    Let $\gamma_1, \gamma_2$ be two arbitrary drops and denote  $a_1 = \mathcal{A}(\gamma_1), a_2 = \mathcal{A}(\gamma_2)$. If $a_1 + a_2 \leq a$ for some $a > 0$, then 
    \begin{equation*}
        \mathcal{E}(\gamma_1) + \mathcal{E}(\gamma_2) \geq 2 \mathcal{E}(\gamma^*_{\frac{a}{2}}).
    \end{equation*}
\end{lemma}

\begin{proof}
    Define $\tilde{a} \coloneqq a_1 + a_2$ and obtain
    \begin{equation*}
    \begin{split}
        \mathcal{E}(\gamma_1) + \mathcal{E}(\gamma_2) &\geq \mathcal{E}(\gamma^*_{a_1})+\mathcal{E}(\gamma^*_{a_2})\\
        &= \sqrt{\frac{\mathcal{A}(\gamma^*)}{a_1}} \mathcal{E}(\gamma^*) + \sqrt{\frac{\mathcal{A}(\gamma^*)}{a_2}} \mathcal{E}(\gamma^*)\\
        &= \sqrt{\mathcal{A}(\gamma^*)}\mathcal{E}(\gamma^*)\Big(\frac{1}{\sqrt{a_1}}+\frac{1}{\sqrt{\tilde{a}-a_1}}\Big).
    \end{split}
    \end{equation*}
    The minimum of this function is attained at $a_1 = \frac{\tilde{a}}{2}$. This shows
    \begin{equation*}
    \begin{split}
        \mathcal{E}(\gamma_1) + \mathcal{E}(\gamma_2) &\geq \sqrt{\mathcal{A}(\gamma^*)}\mathcal{E}(\gamma^*)\frac{2}{\sqrt{\frac{\tilde{a}}{2}}}\\
        &\geq \sqrt{\mathcal{A}(\gamma^*)}\mathcal{E}(\gamma^*)\frac{2}{\sqrt{\frac{a}{2}}}\\
        &= 2 \mathcal{E}(\gamma^*_{\frac{a}{2}}).
    \end{split}
    \end{equation*}
\end{proof}

We have gathered all the necessary results to prove simplicity of the flow for all time under certain conditions on the energy.

\begin{theorem}
\label{thm:a_embeddedness}
    Let $\gamma_0 \in C^\infty(\mathbb{S}^1, \mathbb{R}^2)$ be simple and immersed with $\mathcal{A}(\gamma_0)= a$ and $\mathcal{E}(\gamma_0) < 2 \mathcal{E}(\gamma^*_{\frac{a}{2}})$. Then, the global smooth solution $\gamma$ to \cref{eq:evolution_gradient_flow} will be simple for all time.
\end{theorem}

\begin{proof}
    Suppose $\gamma$ develops a self-intersection and denote by
    \begin{equation*}
        t_s = \inf \{t > 0 \mid \gamma(t) \text{ is not simple}\}
    \end{equation*}
    the first time, where a self-intersection occurs. Following the argument in the proof of \cite[Theorem~3.1]{bucur2014newisoperimetricinequalityelasticae}, we can identify two self-intersections (which are potentially the same)
    \begin{equation*}
        \gamma_{t_s}(x_1) = \gamma_{t_s}(y_1), \quad \gamma_{t_s}(x_2) = \gamma_{t_s}(y_2),
    \end{equation*}    
    such that $\gamma_1 = \gamma_{t_s}|_{[x_1,y_1]}, \gamma_2 = \gamma_{t_s}|_{[x_2,y_2]}$ are drops. They satisfy $\mathcal{A}(\gamma_1) + \mathcal{A}(\gamma_2) \leq \mathcal{A}(\gamma(t_s))$ and necessarily $\mathcal{E}(\gamma_1) + \mathcal{E}(\gamma_2) \leq \mathcal{E}(\gamma(t_s))$. Applying \Cref{lem:min_two_drops} leads to a contradiction, since the energy is non-increasing along the flow.
\end{proof}

\subsection{Subconvergence}

To show convergence, we first want to obtain subconvergence by using a compactness argument. A key requirement for this strategy is a bound on the length of the curve. As we shall see, the curve cannot have arbitrarily narrow pieces, if the flow starts with an energy below twice that of the optimal drop for a given area. However, if the length of the curve tends to infinity, the curve must develop arbitrarily narrow pieces. Therefore, the length has to be bounded. In the following, we formalize this argument.

\begin{definition}
    Let $\gamma \in C^\infty(\mathbb{S}^1, \mathbb{R}^2)$ be simple and immersed. Define
    \begin{equation*}
    \begin{aligned}
        &d: \mathbb{S}^1 \times \mathbb{S}^1 \to [0, \infty],\\
        &d(x,y) = \inf\{t>0 \mid \exists\, 0\le r_1,r_2\le t\colon \gamma(x)+r_1 \nu(x) = \gamma(y) + r_2 \nu(y)\},
    \end{aligned}
    \end{equation*}
    where the infimum is $\infty$, if the set is empty. This is shown in \Cref{fig:a_distance}.
\end{definition}

\begin{remark}
    Even though this map is denoted by $d$, it is not a metric, since it does not satisfy the triangle inequality.
\end{remark}

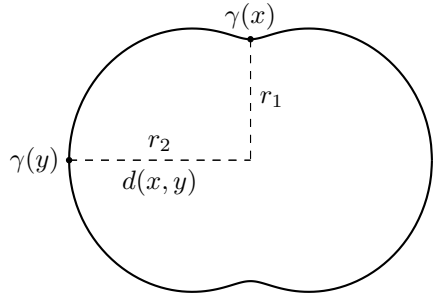
\begin{figure}[htbp]
    \centering
    \begin{tikzpicture}
        % Draw the curve gamma
        \draw[thick, black, domain=0:2*pi, samples=100, smooth]
            plot ({2*cos(\x r) + 0.4*cos(3*\x r)}, {2*sin(\x r) + 0.4*sin(3*\x r)});
        % Choose a point gamma(x) on the curve gamma
        \coordinate (gamma_x) at (0,1.6);
        \coordinate (gamma_y) at (-2.4,0);
        \coordinate (gamma_intersect) at (0, 0);
        \filldraw (gamma_x) circle (1pt) node[above] {$\gamma(x)$};
        \filldraw (gamma_y) circle (1pt) node[left] {$\gamma(y)$};
        % Draw the normal line at gamma(x)
        \draw[dashed] (gamma_x) -- (gamma_intersect) node[midway, right] {$r_1$};
        \draw[dashed] (gamma_y) -- (gamma_intersect) node[midway, below] {$d(x,y)$};
        \draw[dashed] (gamma_y) -- (gamma_intersect) node[midway, above] {$r_2$};
        \draw[dashed] (gamma_x) -- (gamma_intersect);
    \end{tikzpicture}
    \caption{A sketch illustrating the distance $d(x,y)$.}
\label{fig:a_distance}
\end{figure}

\begin{lemma}
\label{lem:a_thin_piece}
    Let $\gamma \in C^\infty(\mathbb{S}^1, \mathbb{R}^2)$ be simple and immersed and $\delta > 0$ with
    \begin{equation*}
        3 \frac{\mathcal{A}(\gamma)}{\mathcal{L}(\gamma)} \leq \delta \leq \frac{1}{\norm{\kappa}_{L^\infty}}.
    \end{equation*}
    Then, there are $x \neq y \in \mathbb{S}^1$ with $d(x,y) \leq \delta$.
\end{lemma}

\begin{proof}
    Suppose that for all $x \neq y \in \mathbb{S}^1$ it holds $d(x,y) > \delta$. Denote by $\tilde{\gamma}: [0,\mathcal{L}(\gamma)] \to \mathbb{R}^2$ the constant speed reparametrization of $\gamma$ with speed $1$. Then, define the map
    \begin{equation*}
    \begin{split}
        &f: [0,\mathcal L(\gamma)] \times [0, \delta] \to A \coloneqq \{p \in \mathbb{R}^2 \mid p = \gamma(x)+r \nu(x), x \in \mathbb{S}^1, r \in [0, \delta]\},\\
        &f(l,r) = \tilde{\gamma}(l) + r \tilde{\nu}(l).
    \end{split}
    \end{equation*}
    A sketch of the set $A$ is given in \Cref{fig:a_delta_hose}. Observe that $f$ is injective (up to a set of measure zero) by the assumption $d(x,y) > \delta$ for all $x,y \in \mathbb{S}^1$. Hence, by the area formula we have
    \begin{equation*}
        \mathcal{H}^2(A) = \int_{[0, \mathcal{L}(\gamma)] \times [0, \delta]} J(Df(l,r)) d\mathcal{L}^2(l,r)
    \end{equation*}
    for the two-dimensional Hausdorff measure $\mathcal{H}^2$ and Lebesgue measure $\mathcal{L}^2$.
    Compute
    \begin{equation*}
    \begin{split}
        Df(l,r) &= \begin{pmatrix}
            \tilde{\tau}(l)-r \tilde{\kappa}(l) \tilde{\tau}(l) & \tilde{\nu}(l)
        \end{pmatrix},\\
        J(Df(l,r)) &= \sqrt{\det(Df(l,r)^\top Df(l,r))} = \sqrt{\det\begin{pmatrix}
            (1-r \tilde{\kappa}(l))^2 & 0\\
            0 & 1
        \end{pmatrix}}\\
        &= 1 - r \tilde{\kappa}.
    \end{split}
    \end{equation*}
    This yields
    \begin{equation}
    \label{eq:area_contradiction}
    \begin{split}
        \mathcal{H}^2(A) &= \int_0^\delta \int_0^{\mathcal{L}(\gamma)} (1-r\tilde{\kappa}) dl dr\\
        &= \int_0^\delta \mathcal{L}(\gamma) - r 2\pi dr\\
        &= \delta \mathcal{L}(\gamma) - \delta^2 \pi \geq \frac{\delta}{2} \mathcal{L}(\gamma) \geq \frac{3}{2} \mathcal{A}(\gamma),
    \end{split}
    \end{equation}
    where we used $\delta \leq \frac{1}{\norm{\kappa}_{L^\infty}} \leq \frac{\mathcal{L}(\gamma)}{2\pi}$. We claim that $A$ is contained in the set that is enclosed by $\gamma$, which leads to a contradiction to \cref{eq:area_contradiction}. If $A$ is not contained in the set that is enclosed by $\gamma$, then there are $x,y \in \mathbb{S}^1$ and $r \leq \delta$ such that $\gamma(x) + r \nu(x) = \gamma(y)$. This is impossible by the assumption $d(x,y) > \delta$.
\end{proof}

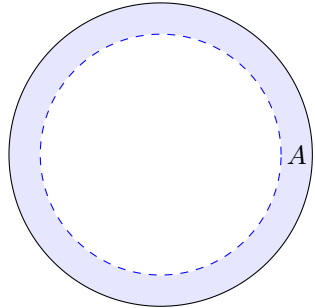
\begin{figure}[htbp]
    \centering
    \begin{tikzpicture}[
        dot/.style={circle,fill,inner sep=1pt},]
        \draw[thick, black, domain=0:2*pi, samples=100, smooth, variable=\t]
            plot ({2*cos(\t r)}, {2*sin(\t r)});

        % delta hose
        \draw[dashed, blue, thick, domain=0:2*pi, samples=100, smooth, variable=\t]
            plot ({2*0.8*cos(\t r)}, {2*0.8*sin(\t r)});

        % Area A between gamma and delta hose
        \begin{scope}
            \clip[domain=0:2*pi, samples=100, smooth, variable=\t]
                plot ({2*cos(\t r)}, {2*sin(\t r)}) -- plot[domain=2*pi:0, samples=100, smooth, variable=\t]
                ({2*0.8*cos(\t r)}, {2*0.8*sin(\t r)}) -- cycle;
            \fill[blue!10] (0,0) circle (2.5);
        \end{scope}
        \node at (1.8,0) {$A$};
    \end{tikzpicture}
    \caption{A sketch illustrating the set $A$ which is the tubular neighborhood inside of $\gamma$.}
\label{fig:a_delta_hose}
\end{figure}

\begin{lemma}
\label{lem:a_thin_piece_tangential}
    Let $\gamma \in C^\infty(\mathbb{S}^1, \mathbb{R}^2)$ be simple and immersed. Suppose that there are $\tilde{x}, \tilde{y} \in \mathbb{S}^1$ with $d(\tilde{x}, \tilde{y}) \leq \frac{\delta}{2}$ for $\delta < \frac{1}{\norm{\kappa}_{L^\infty}}$. Then, there are $x_0, y_0$ such that
    \begin{equation*}
    \begin{split}
        \abs{\gamma(x_0)-\gamma(y_0)} &\leq \delta,\\
         \tau(x_0) &= - \tau(y_0),\\
         (\gamma(x_0)-\gamma(y_0))\cdot \tau(x_0) &= 0,
    \end{split}
    \end{equation*}
    and the straight line segment connecting $\gamma(x_0), \gamma(y_0)$ does not intersect $\gamma$.
\end{lemma}
\begin{proof}
    In this proof, we will use the identification $\mathbb{S}^1 = \mathbb{R}/2\pi\mathbb{Z}$. Define the function
    \begin{equation*}
    \begin{split}
        \varphi_1: \mathbb{S}^1 &\to \mathbb{S}^1,\\
        \varphi_1([x]) &= [\min\{z \in \mathbb{R} \mid z > x, \tau([z]) \cdot \tau([x]) = 0\}].
    \end{split}
    \end{equation*}
    Similarly, we define
    \begin{equation*}
    \begin{split}
        \varphi_2: \mathbb{S}^1 &\to \mathbb{S}^1,\\
        \varphi_2([x]) &= [\max\{z \in \mathbb{R} \mid z < x, \tau([z]) \cdot \tau([x]) = 0\}].
    \end{split}
    \end{equation*}
    The functions $\varphi_1, \varphi_2$ are not continuous in general, but for a sequence $x_n \in \mathbb{S}^1$ with $x_n \to x_0$ we claim that
    \begin{equation}
    \begin{split}
    \label{eq:a_prop_x1_x2}
        \varphi_1(x_0) &\in [x_0, \liminf_{n \to \infty} \varphi_1(x_n)],\\
        \varphi_2(x_0) &\in [\limsup_{n \to \infty} \varphi_2(x_n), x_0],
    \end{split}
    \end{equation}
    where $\liminf$, $\limsup$ are taken with respect to the identification $[x_0,x_0+2\pi)$ of $\mathbb{S}^1$. To see this, choose a subsequence $x_{n_k}$ such that $\varphi_1(x_{n_k}) \to x_\infty \coloneqq \liminf_{n \to \infty} \varphi_1(x_n)$ as $k \to \infty$. It holds $\tau(x_{n_k}) \cdot \tau (\varphi_1(x_{n_k})) = 0$. This gives $\tau(x_0) \cdot \tau(x_\infty) = 0$, which shows the claim. The result for the supremum follows in a similar manner.

    Let $x,y \in \mathbb{S}^1$ be such that $y \in [\varphi_2(x), \varphi_1(x)]$. Assume $y \in [x, \varphi_1(x)]$. The case $y \in [\varphi_2(x), x]$ can be treated by a similar calculation.
    
    We claim that $d(x,y) \geq \frac{1}{\norm{\kappa}_{L^\infty}}$. We may assume $d(x,y) < \infty$. By the choice of $y$ it holds $\tau(z) \cdot \tau(x) \geq 0$ for every $z \in [x,y]$.
    Define $b = \int_{[x,y]} \tau(s)\cdot \tau(x)ds$. Since $\partial_s \theta(x) = \kappa(x)$, we estimate
    \begin{equation}
    \begin{split}
    \label{eq:a_delta_min}
        \norm{\kappa}_{L^\infty} b &= \norm{\kappa}_{L^\infty}\int_{[x,y]} \tau(s)\cdot \tau(x)ds\\
        &\geq \int_{[x,y]} \abs{\kappa(s)}\tau(z) \cdot \tau(x)ds\\
        &= \int_{[x,y]} \abs{\kappa(s)} \cos(\theta(z)-\theta(x))ds\\
        &\geq \Big\vert\int_{\theta(x)}^{\theta(y)} \cos(\alpha-\theta(x))d\alpha\Big\vert\\
        &= \Big\vert\int_0^{\theta(y)-\theta(x)} \cos(\alpha)d\alpha\Big\vert = \sin(\abs{\theta(y)-\theta(x)}).
    \end{split}
    \end{equation}

     Since $b = (\gamma(y) - \gamma(x))\cdot\tau(x)$, it holds $\sin(\abs{\theta(y)-\theta(x)}) = \frac{b}{\abs{\gamma(y)-S}}$ for the intersection $S$ of the rays from $x,y$ in direction of their normals, which exists by $d(x,y) < \infty$, see \Cref{fig:a_sin_equation}. This shows the claim
    \begin{equation}
    \label{eq:lower_bound_d}
        d(x,y) \geq \abs{\gamma(y)-S} = \frac{b}{\sin(\abs{\theta(y)-\theta(x)})} \geq \frac{1}{\norm{\kappa}_{L^\infty}}.
    \end{equation}
    By assumption, there are $\tilde{x}, \tilde{y} \in \mathbb{S}^1$ such that $d(\tilde{x}, \tilde{y}) \leq \frac{\delta}{2} < \frac{1}{2\norm{\kappa}_{L^\infty}}$. \Cref{eq:lower_bound_d} then yields $\tilde{y} \in [\varphi_1(\tilde{x}),\varphi_2(\tilde{x})]$. The triangle inequality gives
    \begin{equation*}
        \abs{\gamma(\tilde{y})-\gamma(\tilde{x})} \leq 2 d(\tilde{x}, \tilde{y}) \leq \delta < \frac{1}{\norm{\kappa}_{L^\infty}}.
    \end{equation*}
    We want to find a minimizer of $\abs{\gamma(y)-\gamma(x)}$ for $x,y \in \{(x,y) \in \mathbb{S}^1 \times \mathbb{S}^1 \mid y \in [\varphi_1(x),\varphi_2(x)]\}$. Choose a minimizing sequence $(x_n,y_n)$. Compactness of $\mathbb{S}^1 \times \mathbb{S}^1$ shows that up to a subsequence, $(x_n, y_n) \to (x_0,y_0)$.
    We have $y_n \in [\varphi_1(x_n),x_n]$ and using \cref{eq:a_prop_x1_x2} gives $y_0 \in [\liminf_{n \to \infty}\varphi_1(x_n), x_0] \subset [\varphi_1(x_0), x_0]$. Similarly, we see $y_0 \in [x_0, \varphi_2(x_0)]$ and thus $y_0 \in [\varphi_1(x_0), \varphi_2(x_0)]$.
    Therefore, $(x_0, y_0)$ is a minimizer, and it holds
    \begin{equation*}
        \abs{\gamma(y_0)-\gamma(x_0)} \leq \abs{\gamma(\tilde{y})-\gamma(\tilde{x})} \leq \delta < \frac{1}{\norm{\kappa}_{L^\infty}}.
    \end{equation*}
    Also, by \cref{eq:a_delta_min} $y_0 \in (\varphi_1(x), \varphi_2(x))$, because if $y_0 = \varphi_1(x_0)$ we have
    \begin{equation}
    \label{eq:lower_estimate_distance}
        \abs{\gamma(y_0)-\gamma(x_0)} \geq \int_{[x_0,y_0]} \tau(z) \cdot \tau(x_0) ds \geq \frac{1}{\norm{\kappa}_{L^\infty}}\sin(\abs{\theta(y_0)-\theta(x_0)}) = \frac{1}{\norm{\kappa}_{L^\infty}}.
    \end{equation}
    A similar argument works for $y=\varphi_2(x_0)$. Therefore, we can take the derivative and deduce
    \begin{equation*}
        0 = \frac{d}{d\varepsilon}\Big|_{\varepsilon=0} \abs{\gamma(x_0)-\gamma(y_0+\varepsilon)}^2 = 2 \partial_x \gamma(y_0) \cdot (\gamma(y_0)-\gamma(x_0)).
    \end{equation*}
    The tangent is just a dilation of $\partial_x \gamma(y_0)$.
    We want to apply a similar argument for the tangent at $x_0$. We have to show that there is $r > 0$ such that $y_0 \in [\varphi_1(x),\varphi_2(x)]$ for all $x \in [x_0-r, x_0 + r]$ to be able to take the derivative. If for all $z \in [x_0,y_0]$ we have $\tau(z)\cdot \tau(x_0) \geq 0$, it holds 
    \begin{equation*}
        \abs{\gamma(y_0)-\gamma(x_0)} \geq \int_{[x_0,y_0]} \tau(z)\cdot\tau(x_0)ds \geq \int_{[x_0,\varphi_1(x_0)]} \tau(z) \cdot \tau(x_0)ds
    \end{equation*}
    and by \cref{eq:lower_estimate_distance} we arrive at a contradiction.
    Thus, there is $z \in [x_0, y_0]$ such that $\tau(z) \cdot \tau(x_0) < 0$. By continuity, there is $r>0$ such that $\tau(z) \cdot \tau(x) < 0$ for all $x \in [x_0-r,x_0+r]$. We conclude $y_0 \in [z,x] \subset [\varphi_1(x),x]$ for all $x \in [x_0-r,x_0+r]$. 
    Similarly, we see $y_0 \in [x, \varphi_2(x)]$ and hence $y_0 \in [\varphi_1(x), \varphi_2(x)]$ for all $x \in [x_0-r,x_0+r]$ for sufficiently small $r>0$. As before, taking the derivative yields $\tau(x_0)\cdot(\gamma(y_0)-\gamma(x_0))=0$.

    Lastly, assume that the line segment from $\gamma(x_0)$ to $\gamma(y_0)$ intersects $\gamma$ at a point $\gamma(z_0)$. We make the following case distinction:\\\\
    \textbf{Case 1: $\tau(z_0) \cdot \tau(x_0) \leq 0$}\\
    Then $z_0 \in [\varphi_1(x_0),\varphi_2(x_0)]$, which is a contradiction to $(x_0,y_0)$ being a minimizer of $\abs{\gamma(x)-\gamma(y)}$ on the set $\{(x,y) \in \mathbb{S}^1 \times \mathbb{S}^1 \mid y \in [\varphi_1(x),\varphi_2(x)]\}$.\\\\
    \textbf{Case 2.1: $\tau(z_0) \cdot \tau(x_0) > 0$ and $\tau(y_0) = -\tau(x_0)$}\\
    It follows that $\tau(z_0) \cdot \tau(y_0) < 0$, which again leads to case 1, since $z_0 \in [\varphi_1(y_0), \varphi_2(y_0)]$.\\\\
    \textbf{Case 2.2: $\tau(z_0) \cdot \tau(x_0) > 0$ and $\tau(y_0) = \tau(x_0)$}\\
    If $z_0 \in [x_0, y_0]$, assume that $\gamma(z_0)$ is the geometrically closest (in the Euclidean distance) intersection to $\gamma(x_0)$ of $\gamma$ with the straight line connecting $\gamma(x_0)$ and $\gamma(y_0)$. We are in one of the two situations illustrated in \Cref{fig:a_thin_piece_nothing_between_case_2.2}. By the Jordan curve theorem, the curve $\tilde{\gamma}$ consisting of $\gamma|_{[x_0,z_0]}$ and the straight line segment from $\gamma(z_0)$ to $\gamma(x_0)$ separates the plane in two open connected components $A,B$. There is $\varepsilon > 0$ such that for all $0 < \delta \leq \varepsilon$ it holds $\gamma(x_0-\delta) \in A$ and $\gamma(z_0+\delta) \in B$. Hence, the segment $\gamma|_{[z_0+\varepsilon, x_0-\varepsilon]}$ has to intersect with $\tilde{\gamma}$, which can only happen at the straight line from $\gamma(z_0)$ to $\gamma(x_0)$. However, then there is another intersection of $\gamma$ with the straight line from $\gamma(y_0)$ to $\gamma(x_0)$, which is geometrically closer to $\gamma(x_0)$ than $\gamma(z_0)$, a contradiction. If $z_0 \in [y_0, x_0]$, we can swap the roles of $x_0$ and $y_0$ in the argument before.

    Using the same argument as for case 2.2, one can show that the tangents have to be opposed to one another, i.e. $\tau(x_0) = -\tau(y_0)$.
\end{proof}

\begin{figure}[htbp]
    \centering
    \begin{tikzpicture}
        % Define radius
        \def\r{4}

        \coordinate (S) at (0,0);
        \coordinate (gamma_x) at (\r,0);
        \coordinate (gamma_y) at ({\r/sqrt(2)},{\r/sqrt(2)});
        \coordinate (A) at (0, {\r/sqrt(2)});

        % Draw part of circle
        \draw (S) -- (gamma_x) arc (0:45:\r) -- cycle;

        % Draw perpendicular line
        \draw (S) -- (A) node[midway,left] {$b$};
        \draw (A) -- (gamma_y) node[midway,above] {};

        % Label points
        \filldraw (gamma_x) circle (1.5pt) node[below] {$\gamma(x)$};
        \filldraw (gamma_y) circle (1.5pt) node[above] {$\gamma(y)$};
        \filldraw (S) circle (1.5pt) node[below] {$S$};

        \pic[draw,angle radius=1cm,angle eccentricity=0.7, "$\beta$"] {angle = A--gamma_y--S};
    \end{tikzpicture}
    \caption{Relation between $\beta \coloneqq \theta(y)-\theta(x)$ and $b$.}
\label{fig:a_sin_equation}
\end{figure}
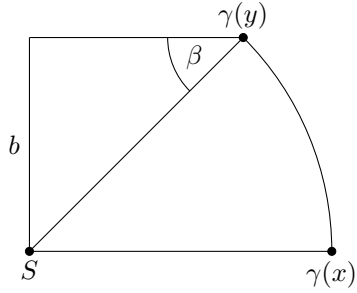

\begin{figure}[htbp]
    \begin{subfigure}[b]{0.48\textwidth}
        \centering
        \begin{tikzpicture}
            \coordinate (gamma_x) at (0,0);
            \coordinate (gamma_y) at (0, 2);
            \coordinate (gamma_z) at (0, 1);

            \filldraw[fill=blue!10] 
                (gamma_x) 
                to[out=0, in=90] (1,0)
                to[out=-90, in=0] (0,-1)
                to[out=180, in=-90] (-1, 0)
                to[out=90, in=180] (gamma_z)
                to[out=-90, in=90]   (gamma_x) -- cycle;
            \draw (gamma_x) -- (gamma_y);

            \filldraw (gamma_x) circle (1.5pt) node[below] {$\gamma(x_0)$};
            \filldraw (gamma_y) circle (1.5pt) node[above] {$\gamma(y_0)$};
            \filldraw (gamma_z) circle (1.5pt) node[right] {$\gamma(z_0)$};
            \filldraw (-0.5,0.2) node[] {$A$};
            \filldraw (1,0.2) node[right] {$B$};
        \end{tikzpicture}
    \end{subfigure}
    \hfill
    \begin{subfigure}[b]{0.48\textwidth}
        \centering
        \begin{tikzpicture}
            \coordinate (gamma_x) at (0,0);
            \coordinate (gamma_y) at (0, 2);
            \coordinate (gamma_z) at (0, 1);

            \filldraw[fill=blue!10] 
                (gamma_x) 
                to[out=0, in=-90] (1.5,1.2)
                to[out=90, in=0] (0,3)
                to[out=180, in=90] (-1, 2)
                to[out=-90, in=180] (gamma_z)
                to[out=-90, in=90]   (gamma_x) -- cycle;
            \draw (gamma_x) -- (gamma_y);

            \filldraw (gamma_x) circle (1.5pt) node[below] {$\gamma(x_0)$};
            \filldraw (gamma_y) circle (1.5pt) node[above] {$\gamma(y_0)$};
            \filldraw (gamma_z) circle (1.5pt) node[right] {$\gamma(z_0)$};
            \filldraw (-0.5,0.4) node[] {$A$};
            \filldraw (0.5,1.6) node[] {$B$};
        \end{tikzpicture}
    \end{subfigure}
    \caption{Case 2.2: A straight line connecting $\gamma(x_0)$ and $\gamma(y_0)$ intersects $\gamma$ at $\gamma(z_0)$.}
\label{fig:a_thin_piece_nothing_between_case_2.2}
\end{figure}
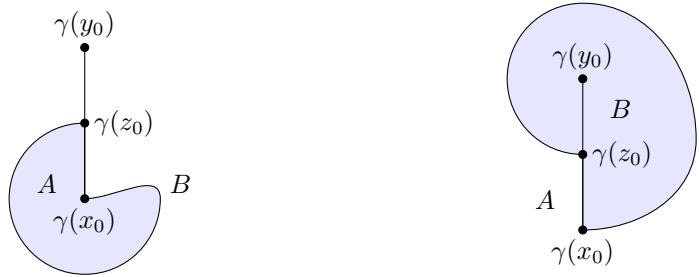

\begin{lemma}
\label{lem:a_gamma_w_touch}
    Suppose $\gamma \in C^\infty(\mathbb{S}^1, \mathbb{R}^2)$ is simple and immersed with $\mathcal{A}(\gamma)=a$, $\mathcal{E}(\gamma) < 2 \mathcal{E}(\gamma_{\frac{a}{2}}^*)$. Moreover, suppose there are $x_0,y_0 \in \mathbb{S}^1$ with
    \begin{equation*}
    \begin{split}
        (\gamma(x_0)-\gamma(y_0))\cdot \tau(x_0) &= 0,\\
        \tau(x_0)&=\ - \tau(y_0),
    \end{split}
    \end{equation*}
    the straight line segment connecting $\gamma(x_0), \gamma(y_0)$ does not intersect $\gamma$ and $\abs{\gamma(y_0)-\gamma(x_0)} \leq \delta$ for some $\delta > 0$. Then, for every $l > 0$ there is a curve $\tilde{\gamma} \in W^{2,2}(\mathbb{S}^1, \mathbb{R}^2)$ with a self-intersection $(x,y)$ such that
    \begin{equation*}
    \begin{split}
        \tilde{\gamma}(x) &= \tilde{\gamma}(y),\\
        \tilde{\tau}(x) &= - \tilde{\tau}(y),
    \end{split}
    \end{equation*}
    which separates $\tilde{\gamma}$ in two drops $\gamma_1=\tilde{\gamma}|_{[x,y]}, \gamma_2=\tilde{\gamma}|_{[y,x]}$.
    Furthermore, it holds
    \begin{equation*}
    \begin{split}
        \mathcal{L}(\tilde{\gamma}) &\geq \mathcal{L}(\gamma),\\
        \mathcal{A}(\tilde{\gamma}) &\leq \mathcal{A}(\gamma) + l \delta,\\
        \mathcal{E}(\tilde{\gamma}) &\leq \mathcal{E}(\gamma) + 512 \sqrt{l^2 + 4\delta^2}\frac{\delta^2}{l^4}.
    \end{split}
    \end{equation*}
\end{lemma}

\begin{proof}
    The idea is to add arcs at $\gamma(x_0)$ and $\gamma(y_0)$ such that the resulting curve has a self-intersection. To ensure that the resulting curve can actually be separated in two drops, we have to show $\theta(x_0) = \theta(y_0) + \pi$ if $0 \notin [y_0, x_0]$ or $\theta(x_0) = \theta(y_0) - \pi$ if $0 \in [y_0,x_0]$. Without loss of generality assume $0 \notin [y_0, x_0]$. Suppose $\theta(x_0) \geq \theta(y_0) + 3 \pi$. However, then $0 \in [x_0,y_0]$ and $\theta(y_0) \leq \theta(x_0) - 3\pi$. By \Cref{lem:a_geometric_information} and \Cref{lem:min_two_drops} we arrive at a contradiction to $\mathcal{E}(\gamma) < 2 \mathcal{E}(\gamma_{\frac{a}{2}}^*)$. If $\theta(x_0) \leq \theta(y_0) - \pi$, we again have a contradiction with \Cref{lem:a_geometric_information} and \Cref{lem:min_two_drops}.
    
    Without loss of generality we may assume equality $\abs{\gamma(y_0)-\gamma(x_0)} = \delta$. Consider a bump function $\psi \in C^\infty(\mathbb{R}, [0, \infty))$ with
    \begin{equation*}
    \begin{aligned}
        \operatorname{supp} \psi &\subset (0, l), \quad \psi\Big(\frac{l}{2}\Big) &&= \frac{\delta}{2},\\
        \abs{\psi'(x)} &\leq \frac{2\delta}{l}, \quad \abs{\psi''(x)} &&\leq 16\frac{\delta}{l^2}.
    \end{aligned}
    \end{equation*}
    Define $\tilde{\psi}: [0,l] \to \mathbb{R}^2, \tilde{\psi}(x) \coloneqq (x, \psi(x))$. It holds
    \begin{equation*}
    \begin{split}
        \tilde{\psi}'(x) &= \begin{pmatrix}
            1\\
            \psi'(x)
        \end{pmatrix},\\
        \abs{\tilde{\psi}'(x)} &= \sqrt{1+\psi'(x)^2},\\
        \partial_s \tilde{\psi}(x) &= \frac{1}{\sqrt{1+\psi'(x)^2}} \begin{pmatrix}
            1\\
            \psi'(x)
        \end{pmatrix}.
    \end{split}
    \end{equation*}
    We compute
    \begin{equation*}
    \begin{split}
        \partial_s^2 \tilde{\psi}(x) &= \frac{1}{\sqrt{1+\psi'(x)^2}} \Bigg(- \frac{\psi'(x)\psi''(x)}{(1+\psi'(x)^2)^{\frac{3}{2}}}\begin{pmatrix}
            1\\
            \psi'(x)
        \end{pmatrix}\\
        &\quad+\frac{1}{\sqrt{1+\psi'(x)^2}}\begin{pmatrix}
            0\\
            \psi''(x)
        \end{pmatrix}\Bigg)\\
        &= \frac{\psi''(x)}{(1+\psi'(x)^2)^2} \begin{pmatrix}
            -\psi'(x)\\
            1
        \end{pmatrix},\\
        \abs{\kappa_{\tilde{\psi}}} &= \abs{\partial_s^2 \tilde{\psi}(x)} \leq \abs{\psi''(x)} \leq 16 \frac{\delta}{l^2}.
    \end{split}
    \end{equation*}
    The tangents at $x_0, y_0$ are parallel and we may assume $\tau(x_0)=(1,0)$. So we can cut the curve $\gamma$ at $\gamma(x_0)$ and $\gamma(y_0)$ and insert the arc $\tilde{\psi}$ at $\gamma(x_0)$ respectively $\gamma(y_0)$ to retrieve a curve $\tilde{\gamma}$ as illustrated in \Cref{fig:a_dumbell_with_touch}. The resulting curve $\tilde{\gamma}$ is $C^1$ and except for $x_0, y_0$ it is $C^\infty$. In particular, it is in $W^{2,2}(\mathbb{S}^1, \mathbb{R}^2)$. Moreover, it has a self-intersection with the desired properties by construction. The bound on the absolute value of the curvature of $\tilde{\psi}$ directly yields the energy bound for $\tilde{\gamma}$, if we account for the additional length $\int_0^l \abs{\tilde{\psi}'(x)}dx \leq l \sqrt{1+4\frac{\delta^2}{l^2}}$. Also, the added area is contained in a rectangle with length $l$ and width $\delta$.
\end{proof}

\begin{figure}[htbp] 
    \begin{subfigure}[b]{0.48\textwidth}
        \centering
        \begin{tikzpicture}
            \coordinate (gamma_x) at (0, -0.5);
            \coordinate (gamma_y) at (0, 0.5);

            \draw[smooth] 
                (gamma_x) 
                to[out=0, in=-90]   (1, 0)
                to[out=90, in=0]    (gamma_y) 
                to[out=180, in=90]  (-1, 0)
                to[out=-90, in=180] (gamma_x);

            \filldraw (gamma_x) circle (1.5pt) node[below] {$\gamma(x_0)$};
            \filldraw (gamma_y) circle (1.5pt) node[above] {$\gamma(y_0)$};
                
            \node at (1.2,0) {$\gamma$};
        \end{tikzpicture}
    \end{subfigure}
    \hfill
    \begin{subfigure}[b]{0.48\textwidth}
        \centering
        \begin{tikzpicture}
            \coordinate (Center) at (0,0);
            \coordinate (gamma_x1) at (1, -0.5);
            \coordinate (gamma_x2) at (-1, -0.5);

            \coordinate (gamma_y1) at (1, 0.5);
            \coordinate (gamma_y2) at (-1, 0.5);
            
            \draw[smooth] 
                (Center)
                to[out=0, in=180] (gamma_x1)
                to[out=0, in=-90]   (2, 0)
                to[out=90, in=0]    (gamma_y1) 
                to[out=180, in=0]  (Center)
                to[out=180, in=0]   (gamma_y2)
                to[out=180, in=90]  (-2, 0)
                to[out=-90, in=180] (gamma_x2)
                to[out=0, in=180]   (Center);
            
            \draw[dashed] (gamma_y1) -- (gamma_x1);
            \draw[dashed] (gamma_x2) -- (gamma_x1);

            \filldraw (gamma_x1) circle (1.5pt) node[below] {$\gamma(x_0)$};
            \filldraw (gamma_x2) circle (1.5pt) node[below] {$\gamma(x_0)$};
            \filldraw (gamma_y1) circle (1.5pt) node[above] {$\gamma(y_0)$};
            \filldraw (gamma_y2) circle (1.5pt) node[above] {$\gamma(y_0)$};
            
            \node at (2.2,0) {$\tilde{\gamma}$};
            \node[right] at (1,0) {$\delta$};
            \node[below] at (0,-0.5) {$l$};
        \end{tikzpicture}
    \end{subfigure}
    \caption{We add a smooth segment at $x_0$ and $y_0$ such that the resulting curve $\tilde{\gamma}$ has a self-intersection.}
\label{fig:a_dumbell_with_touch}   
\end{figure}
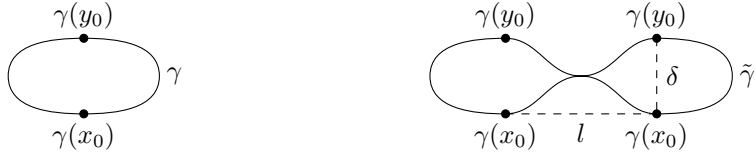

\begin{theorem}
\label{thm:a_length_bound_above}
    Let $\gamma_0 \in C^\infty(\mathbb{S}^1, \mathbb{R}^2)$ be simple and immersed with $\mathcal{A}(\gamma_0)= a$ and $\mathcal{E}(\gamma_0) < 2 \mathcal{E}(\gamma^*_{\frac{a}{2}})$. Then, the global smooth and simple
    solution $\gamma$ to \cref{eq:evolution_gradient_flow} has bounded length $\mathcal{L}(\gamma) \leq c(\gamma_0)$.
\end{theorem}

\begin{proof}
    Suppose that there is a sequence of times $t_n \to \infty$ such that $\mathcal{L}(\gamma_{t_n}) \to \infty$. Define the energy gap $\Delta \coloneqq 2 \mathcal{E}(\gamma^*_{\frac{a}{2}})-\mathcal{E}(\gamma_0)$. From \Cref{lem:infinity_bound_kappa} we have $\norm{\kappa}_{L^\infty} \leq c(\gamma_0)$. Choose $\delta > 0$ such that it holds
    \begin{equation*}
    \begin{split}
        \delta &< \frac{1}{c(\gamma_0)},\\
        512\sqrt{1+4\delta^2}\delta^2 &\leq \frac{\Delta}{4},\\
        2 \sqrt{\frac{a}{a+\delta}} \mathcal{E}(\gamma_{\frac{a}{2}}^*) &> 2 \mathcal{E}(\gamma_{\frac{a}{2}}^*) - \frac{\Delta}{4}.
    \end{split}
    \end{equation*}
    By assumption there exists $n \in \mathbb{N}$ such that $3\frac{\mathcal{A}(\gamma_{t_n})}{\mathcal{L}(\gamma_{t_n})} = 3\frac{a}{\mathcal{L}(\gamma_{t_n})} \leq \frac{\delta}{2}$. Hence, we can apply \Cref{lem:a_thin_piece} to $\gamma_{t_n}$ and $\frac{\delta}{2}$. Then, we apply \Cref{lem:a_thin_piece_tangential}. This yields the necessary conditions for \Cref{lem:a_gamma_w_touch} with $\delta$. So apply \Cref{lem:a_gamma_w_touch} with $l=1$ to find $\tilde{\gamma}_{\delta} \in W^{2,2}(\mathbb{S}^1, \mathbb{R}^2)$ with
    \begin{equation}
    \begin{split}
    \label{eq:a_area_energy_gamma_touch}
        \mathcal{A}(\tilde{\gamma}_{\delta}) &\leq \mathcal{A}(\gamma_{t_n})+\delta,\\
        \mathcal{E}(\tilde{\gamma}_{\delta}) &\leq \mathcal{E}(\gamma_{t_n}) + 512 \sqrt{1+4\delta^2}\delta^2 \leq \mathcal{E}(\gamma_{t_n}) + \frac{\Delta}{4}.
    \end{split}
    \end{equation}
    We know that $\tilde{\gamma}$ has a self-intersection, which separates two drops $\gamma_1, \gamma_2$ of $\tilde{\gamma}_{\delta}$ with area $a_1 \coloneqq \mathcal{A}(\gamma_1), a_2 \coloneqq \mathcal{A}(\gamma_2)$. By \cref{eq:a_area_energy_gamma_touch} we obtain $a_\delta \coloneqq a_1 + a_2 \leq a + \delta$. \Cref{lem:min_two_drops} yields
    \begin{equation*}
    \begin{split}
        \mathcal{E}(\gamma_{t_n}) &\geq \mathcal{E}(\tilde{\gamma}_{\delta}) - \frac{\Delta}{4} = \mathcal{E}(\gamma_1) + \mathcal{E}(\gamma_2) - \frac{\Delta}{4}\\
        &\geq 2 \mathcal{E}(\gamma^*_{\frac{a_{\delta}}{2}}) - \frac{\Delta}{4} = 2 \sqrt{\frac{a}{a_\delta}}\mathcal{E}(\gamma^*_{\frac{a}{2}}) - \frac{\Delta}{4} \\
        &\geq 2 \sqrt{\frac{a}{a+\delta}} \mathcal{E}(\gamma_{\frac{a}{2}}^*) - \frac{\Delta}{4} > 2 \mathcal{E}(\gamma_{\frac{a}{2}}^*) - \frac{\Delta}{2}\\
        &> \mathcal{E}(\gamma_0).
    \end{split}
    \end{equation*}
    This is a contradiction because the energy is non increasing along the flow.
\end{proof}

Now, we will use a compactness argument to obtain subconvergence of $\gamma$ and the energy identity to show that the limit is an area-constrained elastica.

\begin{theorem}
\label{thm:subconvergence}
    Let $\gamma_0 \in C^\infty(\mathbb{S}^1, \mathbb{R}^2)$ be simple and immersed with $\mathcal{A}(\gamma_0)=a$ and $\mathcal{E}(\gamma_0) < 2 \mathcal{E}(\gamma^*_{\frac{a}{2}})$. The global smooth and simple solution $\gamma$ to \cref{eq:evolution_gradient_flow} for initial data $\gamma_0$ subconverges up to translation and reparametrization to an area-constrained elastica $\gamma_\infty$.
\end{theorem}

\begin{proof}
    Define $u(t) \coloneqq \|\partial_t \gamma \|_{L^2(ds)}^2$. The energy identity \cref{eq:energy_identity} shows
    \begin{equation}
        u(t) = -\partial_t \mathcal{E}(\gamma).
    \end{equation}
    Hence, for $0 \leq T < \infty$ it holds $\int_0^T u(t) dt = \mathcal{E}(\gamma_0)-\mathcal{E}(\gamma_T)$ and $0 \leq \mathcal{E}(\gamma_T) \leq \mathcal{E}(\gamma_0)$, since the energy is non-increasing along the flow. We deduce $u \in L^1((0,\infty))$ by taking the limit. Consequently, there is a subsequence $t_i \to \infty$ with $u(t_i) \to 0$.

    From \Cref{thm:a_length_bound_above}, we obtain a length bound $\mathcal{L}(\gamma) \leq c(\gamma_0)$. By translating $\gamma(t_i)$ for each $i\in \mathbb N$ to a new curve $\tilde{\gamma}_i$, e.g. such that $\tilde{\gamma}_i(0)=0$, we arrive at $\norm{\tilde{\gamma}_i}_{L^\infty} \leq c(\gamma_0)$. Moreover, we reparametrize $\tilde{\gamma}_i$ to a constant speed curve. \Cref{prop:length_bound_below} demonstrates 
    \begin{equation}
    \label{eq:bound_first_derivative}
        c(\gamma_0) \leq \vert\partial_x \tilde{\gamma}_i\vert = \frac{\mathcal{L}(\gamma)}{2\pi} \leq c(\gamma_0).
    \end{equation}
    \Cref{lem:infinity_bound_kappa,eq:polynomial_identity_1} yield bounds $\|\partial_{s}^{m} \tilde{\gamma}_i\|_{L^\infty} = \|\partial_{s}^{m} \gamma(t_i)\|_{L^\infty} \leq c_m(\gamma_0)$ for $m \geq 2$. Thus, 
    \Cref{eq:bound_first_derivative} implies $\|\partial_x^m\tilde{\gamma}_i \|_{L^\infty} \le c_m(\gamma_0)$. Therefore, we have uniform bounds on all derivatives of $\tilde{\gamma}_i$. Compactness, which is obtained by Arzel\`a--Ascoli, and taking a subsequence proves that  $\tilde{\gamma}_i$ converges smoothly to an immersed curve $\gamma_\infty$. Observe
    \begin{equation}
    \label{eq:a_constrained_elasticae}
        u(t_i) = \int_{\mathbb{S}^1} (-2\partial_{s}^2 \tilde{\kappa}_i-\tilde{\kappa}_i^3-\lambda(\tilde{\gamma}_i))^2ds_{\tilde{\gamma}_i}
    \end{equation}
    as the curvature and $\lambda$ are invariant under reparametrizations. Since $u(t_i) \to 0$, we have that $2\partial_{s}^2 \tilde{\kappa}_i+\tilde{\kappa}_i^3+\lambda(\tilde{\gamma}_i) \to 0$ in $L^2(ds)$ as $i \to \infty$. Recall \cref{eq:a_lagrange_multiplier} and note that $\lambda$ is continuous with respect to the $C^\infty(\mathbb{S}^1, \mathbb{R}^2)$-topology. By smooth convergence, we conclude that $\gamma_\infty$ is an area-constrained elastica.
\end{proof}

\subsection{\L{}ojasiewicz--Simon inequality}
In order to prove convergence of gradient flows from subconvergence, we will apply a \L{}ojasiewicz--Simon inequality. In the derivation of said inequality we need to analyze analyticity of the functionals $\mathcal{E}$ and $\mathcal{A}$.

For Banach spaces $X,Y$ and an open set $D \subset X$, a map $f: D \to Y$ is analytic in $x_0 \in D$, if in a neighborhood of $x_0$ it holds
\begin{equation*}
    \sum_{k=0}^\infty \norm{a_k}\norm{x-x_0}^k \text{ converges and } f(x) = \sum_{k=0}^\infty a_k(x-x_0)^k.
\end{equation*}
Each $a_k$ is a $k$-linear, symmetric and continuous map $a_k: X^k \to Y$ and by $a_k(x-x_0)^k$ we mean $a_k(\underbrace{x-x_0,\ldots,x-x_0}_{k \text{ entries}})$. By $\norm{a_k}$ we denote the multilinear operator norm.
For a more detailed discussion on analyticity, we refer to \cite[Section~3.1]{Dall_Acqua_2016}.

First of all, we determine the suitable spaces and functionals. The following definitions and arguments are very similar to the ones in \cite[Article B, Section 4]{Rupp_Phd}.
\begin{definition}
\label{def:sobolev_space_immersion}
    Define 
    \begin{equation*}W^{4,2}_\text{Imm}(\mathbb{S}^1, \mathbb{R}^2) = \{\gamma \in W^{4,2}(\mathbb{S}^1, \mathbb{R}^2) \mid \partial_x \gamma(x) \neq 0 \quad \forall x \in \mathbb{S}^1\}.
    \end{equation*}
    We will use this subindex for other Sobolev spaces as well.
\end{definition}
By the Sobolev embedding theorem, the following inclusion holds
\begin{equation*}
    W^{4,2}(\mathbb{S}^1, \mathbb{R}^2) \hookrightarrow C^3(\mathbb{S}^1,\mathbb{R}^2).
\end{equation*}
Therefore, it makes sense to evaluate the derivatives $\gamma \in W^{4,2}(\mathbb{S}^1, \mathbb{R}^2)$ pointwise and for $i \in \{0,1,2,3\}$ we have $\partial_x^i\gamma(0)=\partial_x^i\gamma(2\pi)$.
\begin{definition}
    Fix $\bar{\gamma} \in W_\text{Imm}^{5,2}(\mathbb{S}^1, \mathbb{R}^2)$. We define
    \begin{equation*}
        V = W^{4,2}(\mathbb{S}^1,\mathbb{R}^2).
    \end{equation*}
    Then define the space of vector fields normal to $\bar{\gamma}$
    \begin{equation*}
        V^\perp = \{\gamma \in W^{4,2}(\mathbb{S}^1,\mathbb{R}^2)\mid \forall x \in \mathbb{S}^1 \quad \gamma(x) \cdot \partial_x \bar{\gamma}(x) = 0\}
    \end{equation*}
    and
    \begin{equation*}
        H^\perp = L^{2,\perp}(\mathbb{S}^1,\mathbb{R}^2) = \{u \in L^2(\mathbb{S}^1, \mathbb{R}^2)\mid u \cdot \partial_x \bar{\gamma} = 0 \text{ a.e.}\}.
    \end{equation*}
    Both of these spaces are Hilbert spaces with $\langle .,.\rangle_{W^{4,2}}$, $\langle .,.\rangle_{L^2}$ respectively.
    Lastly, for $\varepsilon > 0$ define
    \begin{equation*}
        U_\varepsilon = \{u \in V^\perp \mid \|u\|_{W^{4,2}} < \varepsilon\}.
    \end{equation*}
\end{definition}
The $L^2$-orthogonal projection on $H^\perp$ is given by
\begin{equation*}
    P^\perp(\gamma) = \gamma - (\gamma \cdot \partial_s \bar{\gamma}) \partial_s \bar{\gamma}.
\end{equation*}
Using the Sobolev embedding theorem again, we have for $u \in U_\varepsilon$ that $\|\partial_x u\|_\infty < c\varepsilon$ for some constant $c > 0$ independent of $u$. So, for $\varepsilon$ sufficiently small, $\bar{\gamma}+u$ is still immersed.

\begin{definition}
    Fix $\bar{\gamma} \in W^{5,2}_\text{Imm}(\mathbb{S}^1, \mathbb{R}^2)$ and define
    \begin{equation*}
    \begin{split}
        &A_\varepsilon: U_\varepsilon \to \mathbb{R}, u \mapsto \mathcal{A}(\bar{\gamma}+u),\\
        %&L:U_\varepsilon \to \mathbb{R},u\mapsto \mathcal{L}(\bar{\gamma}+u),\\
        &E_\varepsilon:U_\varepsilon \to [0,\infty), u \mapsto \mathcal{E}(\bar{\gamma}+u).
    \end{split}
    \end{equation*}
\end{definition}

\begin{proposition}[{\cite[Proof of Theorem~3.1, Remark~3.3, Corollary 3.13]{Dall_Acqua_2016}}]
\label{prop:energy_properties}
    The energy $E_\varepsilon$ satisfies for $\bar{\gamma} \in W^{5,2}_{\text{Imm}}(\mathbb{S}^1, \mathbb{R}^2)$
    \begin{enumerate}[label=(\arabic*)]
        \item $E_\varepsilon:U_\varepsilon \to [0,\infty)$ is analytic,
        \item The gradient $\nabla_{L^2} E_\varepsilon:U_\varepsilon \to H^\perp$ is analytic,
        \item The Fréchet derivative $(\nabla_{L^2} E_\varepsilon)'(0): V^\perp \to H^\perp$ is Fredholm with index zero.
    \end{enumerate}
\end{proposition}

% Todo: Check if this is needed
% This is not needed anymore, I think
% \begin{proposition}[{\cite[Proposition~4.5]{Rupp2024}}]
% \label{prop:length_properties}
%     The length functional satisfies the following properties for $\bar{\gamma} \in W^{4,2}_{\text{Imm}}(\mathbb{S}^1, \mathbb{R}^2)$
%     \begin{enumerate}[label=(\arabic*)]
%         \item $L:U_\varepsilon \to \mathbb{R}$ is analytic,
%         \item The gradient $\nabla_{L^2} L:U_\varepsilon \to H^\perp$ is analytic,
%         \item The Fréchet derivative $(\nabla_{L^2} L)'(0):V^\perp \to H^\perp$ is compact,
%         \item $L(0)=\ell$ and $\nabla L(0) \neq 0$.
%     \end{enumerate}
% \end{proposition}

% \begin{remark}
%     The proposition above is proven in the reference only for open curves. However, the canonical embedding $W^{k,2}(\mathbb{S}^1, \mathbb{R}^2) \to W^{k,2}([0,2\pi], \mathbb{R}^2)$ shows that $E$ is also analytic on $W^{k,2}(\mathbb{S}^1, \mathbb{R}^2)$. The results for the gradient and the Fréchet derivative also hold. Here, have to prove well-definedness, which is clear from $H^\perp = L^{2,\perp}(\mathbb{S}^1, \mathbb{R}^2) = L^{2,\perp}([0,2\pi], \mathbb{R}^2)$.
% \end{remark}

\begin{proposition}
\label{prop:area_properties}
    The area functional satisfies the following properties for $\bar{\gamma} \in W^{5,2}_{\text{Imm}}(\mathbb{S}^1, \mathbb{R}^2)$
    \begin{enumerate}[label=(\arabic*)]
        \item $A_\varepsilon: U_\varepsilon \to \mathbb{R}$ is analytic,
        \item The gradient $\nabla_{L^2} A_\varepsilon:U_\varepsilon \to H^\perp$ is analytic,
        \item The Fréchet derivative $(\nabla_{L^2} A_\varepsilon)'(0):V^\perp \to H^\perp$ is compact,
        \item $A_\varepsilon(0)=\mathcal{A}(\bar\gamma)$ and $\nabla_{L^2} A_\varepsilon(0)\neq 0$.
    \end{enumerate}
\end{proposition}
\begin{proof}
    \begin{enumerate}[label=(\arabic*)]
        \item The functional $A_\varepsilon$ is given by
        \begin{equation*}
        \begin{split}
            A_\varepsilon(u) = \mathcal{A}(\bar{\gamma}+u) = \int_{\mathbb{S}^1} (\bar{\gamma}_1+u_1)\partial_x (\bar{\gamma}_2+u_2)dx,
        \end{split}
        \end{equation*}
        where $\bar{\gamma} = (\bar{\gamma}_1, \bar{\gamma_2})$ and $u = (u_1, u_2)$.
        The projection on coordinates is linear and continuous, thus analytic. Similarly, the derivative $\partial_x$ is analytic. As the multiplication is analytic as well, and so is the integral, analyticity of $\mathcal{A}$ follows. $A_\varepsilon$ is just the concatenation of an affine function and $\mathcal{A}$ and therefore also analytic.
        \item The $H^\perp$-gradient is given by $\nabla_{L^2} A_\varepsilon(u) = P^\perp(-\nu_{\bar{\gamma}+u}\vert \partial_x (\bar{\gamma}+u)\vert)$. The map $u \mapsto \vert \partial_x (\bar{\gamma}+u)\vert$ is analytic by \cite[Lemma~3.4, 1.]{Dall_Acqua_2016}. Also from \cite[Lemma~3.4, 2.]{Dall_Acqua_2016} it follows that $\tau_{\bar{\gamma}+u} = \partial_s(\bar{\gamma}+u)$ is analytic in $u$. The normal field $\nu_{\bar{\gamma}+u}$ is just a rotation of $\tau_{\bar{\gamma}+u}$, which is a linear bounded map and hence analytic. By the same argument, the projection is analytic as well.
        \item We compute with \Cref{lem:differential_properties}
        \begin{equation*}
        \begin{split}
            (\nabla_{L^2} A_\varepsilon)'(0)u &= \partial_h\vert_{h=0} \nabla_{L^2} A_\varepsilon(hu) = \partial_h \vert_{h=0} P^\perp(-\nu_{\bar{\gamma}+hu}\vert \partial_x (\bar{\gamma}+hu)\vert)\\
            &= - P^\perp (\partial_h \vert_{h=0} \nu_{\bar{\gamma}+hu} \vert \partial_x\bar{\gamma}\vert)- P^\perp(\nu_{\bar{\gamma}} \partial_h\vert_{h=0}\vert \partial_x\bar{\gamma}+h \partial_x u\vert)\\
            &= - \partial_s u \cdot \tau_{\bar{\gamma}} \abs{\partial_x \bar{\gamma}} \nu_{\bar{\gamma}}.
        \end{split}
        \end{equation*}
        The Sobolev embedding gives $\bar{\gamma} \in C^4(\mathbb{S}^1, \mathbb{R}^2)$. Therefore, we can bound the $W^{3,2}$-norm of the term $- \partial_s u \cdot \tau_{\bar{\gamma}} \abs{\partial_x \bar{\gamma}} \nu_{\bar{\gamma}}$ by a constant depending on $\bar{\gamma}$ and the $W^{4,2}$-norm of $u$. Hence, by Rellich-Kondrachov $u \mapsto - \partial_s u \cdot \tau_{\bar{\gamma}} \abs{\partial_x \bar{\gamma}} \nu_{\bar{\gamma}}$ is compact as a map to $L^2(\mathbb{S}^1, \mathbb{R}^2)$.
        \item By definition, we have $A_\varepsilon(0)=\mathcal{A}(\bar{\gamma})$ and $\nabla_{L^2} A_\varepsilon(0) = -\nu_{\bar{\gamma}}\vert\partial_x\bar{\gamma}\vert\neq 0$, since $\bar{\gamma}$ is immersed.
    \end{enumerate}
\end{proof}

The main result we want to use is the following corollary. The symbol $C^\omega$ denotes the space of analytic functions between Banach spaces.
\begin{corollary}[{\cite[Corollary~5.2]{RUPP2020108708}}]
\label{cor:lojasiewicz_simon_manifolds}
    Let $V$ be a Hilbert space, $U \subset V$ be an open set, $m \in \mathbb{N}$ and let $\mathcal{E} \in C^\omega(U, \mathbb{R})$, $\mathcal{G} \in C^\omega(U,\mathbb{R}^m)$. Let $\bar{u} \in U$ and suppose
    \begin{enumerate}[label=(\arabic*)]
        \item there exists a Hilbert space $(H, \langle \cdot, \cdot \rangle)$ with $V \hookrightarrow H$ densely,
        \item $\mathcal{E}$ possesses an $H$-gradient $\nabla \mathcal{E}(u)$ at each $u \in U$ and the map $u \mapsto \nabla \mathcal{E}(u): U \to H$ is analytic,
        \item the second derivative $\mathcal{E}''(\bar{u}) = (\nabla \mathcal{E})'(\bar{u}): V \to H$ is Fredholm of index zero,
        \item for any $u \in U$, the components $\mathcal{G}_k: U \to \mathbb{R}$ of $\mathcal{G}$ possess $H$-gradients $\nabla \mathcal{G}_k$ such that $U \ni u \mapsto \nabla \mathcal{G}_k(u) \in H$ is analytic for all $k = 1, \ldots, m$,
        \item the Fréchet derivatives $(\nabla \mathcal{G}_k)'(\bar{u}): V \to H$ are compact for all $k = 1, \ldots, m$,
        \item $\mathcal{G}(\bar{u}) = 0$ and the $H$-gradients $\nabla \mathcal{G}_1(\bar{u}), \ldots, \nabla \mathcal{G}_m(\bar{u})$ are linearly independent.
    \end{enumerate}
    
    Then, $\mathcal{M} \coloneqq \{u \in U \mid \mathcal{G}(u) = 0\}$ is locally an analytic submanifold of $V$ of codimension $m$ near $\bar{u}$.
    
    If $\bar{u}$ is a critical point of $\mathcal{E}|_\mathcal{M}$, then the restriction satisfies a refined \L{}ojasiewicz--Simon gradient inequality at $\bar{u}$, i.e. there exist $C, \sigma > 0$ and $\theta \in (0, \frac{1}{2}]$ such that for any $u \in \mathcal{M}$ with $\|u - \bar{u}\|_V \leq \sigma$, we have
    \begin{equation*}
    |\mathcal{E}(u) - \mathcal{E}(\bar{u})|^{1-\theta} \leq C \|P(u)\nabla \mathcal{E}(u)\|_H,
    \end{equation*}
    where $P(u): H \to H$ is the orthogonal projection onto $\overline{\mathcal{T}_u \mathcal{M}} \coloneqq \overline{\mathcal{T}_u \mathcal{M}}^{\|\cdot\|_H}$, which denotes the closure of the tangent space at $u$ with respect to the $H$-norm.
\end{corollary}

First, we will only consider normal variations. To generalize the result to arbitrary variations, we use the following reparametrization argument.

\begin{lemma}[{\cite[Lemma~3.26]{POZZETTA2022112581}}]
\label{lem:normal_reparametrization}
    Let $\bar{\gamma} \in C^\infty(\mathbb{S}^1, \mathbb{R}^2)$ be an immersed curve. There is $\rho =\rho(\bar{\gamma})>0$ such that for any $0 < \sigma < \rho$ there exists $\tilde{\sigma}=\tilde{\sigma}(\sigma, \bar{\gamma}) > 0$ with the property that for any $\psi \in W^{4,2}(\mathbb{S}^1, \mathbb{R}^2)$ with $\|\psi\|_{W^{4,2}} \leq \tilde{\sigma}$, there exists a $W^{4,2}$-diffeomorphism $\Phi: \mathbb{S}^1 \to \mathbb{S}^1$ such that
    \begin{equation}
        (\bar{\gamma}+\psi) \circ \Phi = \bar{\gamma}+\phi
    \end{equation}
    for some $\phi \in V^\perp$ with $\|\phi\|_{W^{4,2}} \leq \sigma$.
\end{lemma}

As we shall see, we will show convergence of reparametrized versions of the gradient flows. Hence, we need the following relations between the original curve and its reparametrized version.

\begin{definition}[{\cite[Definition~4.9]{Rupp2024}}]
\label{def:const_speed_reparametrization}
    Let $T \in (0, \infty]$ and let $\gamma: [0, T) \times \mathbb{S}^1 \to \mathbb{R}^2$ be a family of immersed curves. The reparametrization $\tilde{\gamma}(t)$ of $\gamma(t)$ with constant speed $\frac{\mathcal{L}(\gamma(t))}{2\pi}$ is given by
    \begin{equation*}
    \tilde{\gamma}(t, x) \coloneqq \gamma\left(t, \psi(t, x)\right)
    \end{equation*}
    where $\psi(t, \cdot): \mathbb{S}^1 \to \mathbb{S}^1$ is the inverse of $\phi(t, \cdot): \mathbb{S}^1 \to \mathbb{S}^1$ given by
    \begin{equation*}
    \phi(t, x) \coloneqq \frac{2\pi}{\mathcal{L}(\gamma(t))} \int_0^x |\partial_x \gamma(t, z)| dz = \frac{2\pi}{\mathcal{L}(\gamma(t))} \int_0^x ds_{\gamma(t)}.
    \end{equation*}
    The inverse exists since $\gamma$ is immersed.
\end{definition}

\begin{lemma}[{\cite[Lemma~4.10]{Rupp2024}}]
\label{lem:const_speed_reparametrization_bounds}
    Suppose $T \in (0, \infty]$ and $\gamma: [0, T) \times \mathbb{S}^1 \to \mathbb{R}^2$ is a family of immersed curves in $\mathbb{R}^2$ and $\mathcal{L}(\gamma(t)) > 0$ for all $t \in [0, T)$. Then, if $\tilde{\gamma}(t)$ is the constant speed $\frac{\mathcal{L}(\gamma(t))}{2\pi}$ reparametrization of $\gamma(t)$, for all $t \in [0, T)$ we have
    \begin{equation*}
    \|\partial_t \tilde{\gamma}(t)\|_{L^2(dx)} \leq \sqrt{2\pi\Big(\frac{2}{\mathcal{L}(\gamma(t))} + 16 \mathcal{E}(\gamma(t))\Big)} \|\partial_t \gamma\|_{L^2(ds_{\gamma(t)})}.
    \end{equation*}
    In particular, if the length is uniformly bounded in time from below, i.e. $\mathcal{L}(\gamma) \geq c_1 > 0$, and the energy from above, i.e. $\mathcal{E}(\gamma) \leq c_2$, we have
    \begin{equation*}
    \|\partial_t \tilde{\gamma}(t)\|_{L^2(dx)} \leq C \|\partial_t \gamma\|_{L^2(ds_{\gamma})},
    \end{equation*}
    for all $t \in (0, T)$, where $C = \sqrt{2\pi(\frac{2}{c_1} + 16 c_2)}$.
\end{lemma}

As we work in the space of curves with fixed area, consider the following definition.

\begin{definition}
    Fix $a \in \mathbb{R}$ and define
    \begin{equation*}
        \mathcal{X} = \{\gamma \in W_{\text{Imm}}^{4,2}(\mathbb{S}^1,\mathbb{R}^2)\mid \mathcal{A}(\gamma)=a\}.
    \end{equation*}
\end{definition}
From the Sobolev embedding theorem it follows that we can embed $\mathcal{X} \hookrightarrow C^3(\mathbb{S}^1,\mathbb{R}^2)$.

\begin{theorem}
\label{thm:a_lojasiewicz_simon_normal}
    Let $\bar{\gamma} \in \mathcal{X} \cap W^{5,2}(\mathbb{S}^1, \mathbb{R}^2)$ be an area-constrained elastica. Then, there exist $C, \sigma > 0$ and $\theta \in (0,\frac{1}{2}]$ such that for all $\gamma = \bar{\gamma}+u \in \mathcal{X}$ with $u \in V^\perp$ and $\|u\|_{W^{4,2}} \leq \sigma$ we have
    \begin{equation*}
        \vert \mathcal{E}(\gamma)-\mathcal{E}(\bar{\gamma})\vert^{1-\theta} \leq C \|-\nabla \mathcal{E}(\gamma)+\lambda(\gamma)\nabla \mathcal{A}(\gamma)\|_{L^2(ds_\gamma)}.
    \end{equation*}
\end{theorem}

\begin{proof}
    We need to verify the conditions of \Cref{cor:lojasiewicz_simon_manifolds} for the energy $\mathcal{E}=E_\varepsilon$ and the constraint $\mathcal{G}(u)=\mathcal{A}(\bar{\gamma}+u)-a$
    on the spaces $V=V^\perp$, $H=H^\perp$.
    \begin{enumerate}[label=(\arabic*)]
        \item We have $V^\perp \hookrightarrow H^\perp$ densely.
        \item By \Cref{prop:energy_properties}.
        \item By \Cref{prop:energy_properties}.
        \item Follows from \Cref{prop:area_properties}.
        \item Also follows from \Cref{prop:area_properties}.
        \item The gradient is $P^\perp(\nabla_{L^2(dx)} \mathcal{A}(\bar{\gamma})) = P^\perp(-\nu_{\bar{\gamma}} \abs{\partial_x \bar{\gamma}}) = -\nu_{\bar{\gamma}} \abs{\partial_x \bar{\gamma}}$. This does not vanish.
    \end{enumerate}
    Since we assumed that $\bar{\gamma}$ is an area-constrained elastica, we know that $0$ is a critical point of $E_\varepsilon$ on $\mathcal{M}=\mathcal{G}^{-1}(\{0\})$.
    Then, by \Cref{cor:lojasiewicz_simon_manifolds} there exist $C, \sigma > 0$ and $\theta \in (0,\frac{1}{2}]$ such that for any $u\in \mathcal{M}$ with $\|u\|_{W^{4,2}}\leq \sigma$ we have
    \begin{equation*}
        \vert E_\varepsilon(u)-E_\varepsilon(0)\vert^{1-\theta} \leq C \|P(u) \nabla_{L^2} E_\varepsilon(u)\|_{L^2(dx)},
    \end{equation*}
    where $P(u):H^\perp \to H^\perp$ denotes the orthogonal projection on $\overline{T_u\mathcal{M}}=\{y \in H^\perp\mid \langle \nabla_{L^2(dx)}A_\varepsilon(u),y\rangle_{L^2(dx)}=0\}$. Thus, for $\lambda = \lambda(\gamma)$ from \cref{eq:a_lagrange_multiplier} with $\gamma = \bar{\gamma}+u$ we have
    \begin{equation*}
    \begin{split}
        &\quad\|P(u)\nabla_{L^2(dx)}E_\varepsilon(u)\|_{L^2(dx)}\\
        &= \|P(u)(-\nabla_{L^2(dx)}E_\varepsilon(u)+\lambda \nabla_{L^2(dx)} A_\varepsilon(u))\|_{L^2(dx)}\\
        &\leq \|-\nabla_{L^2(dx)}E_\varepsilon(u)+\lambda\nabla_{L^2(dx)}A_\varepsilon(u)\|_{L^2(dx)}\\
        &= \|-\nabla_{L^2(ds_\gamma)}E_\varepsilon(u)\vert \partial_x \gamma \vert+\lambda\nabla_{L^2(ds_{\gamma})}A_\varepsilon(u)\vert \partial_x \gamma \vert\|_{L^2(dx)}\\
        &\leq \| \partial_x \gamma \|_{L^\infty}^\frac{1}{2}\|-\nabla_{L^2(ds_\gamma)}E_\varepsilon(u)+\lambda\nabla_{L^2(ds_{\gamma})}A_\varepsilon(u)\|_{L^2(ds_\gamma)}.
    \end{split}
    \end{equation*}
    By the Sobolev embedding we may assume $\|\partial_x \gamma \|_{L^\infty}$ is bounded independent of $u$.
\end{proof}

Using \Cref{lem:normal_reparametrization}, we demonstrate that a \L{}ojasiewicz--Simon inequality also holds for nonnormal deviations which are sufficiently small.

\begin{theorem}
    Let $\bar{\gamma} \in \mathcal{X} \cap C^\infty(\mathbb{S}^1,\mathbb{R}^2)$ be an area-constrained elastica. Then, there exist $C, \tilde{\sigma} > 0$ and $\theta \in (0,\frac{1}{2}]$ such that
    \begin{equation*}
        \abs{\mathcal{E}(\gamma)-\mathcal{E}(\bar{\gamma})}^{1-\theta} \leq C \|-\nabla\mathcal{E}(\gamma)+\lambda \nabla\mathcal{A}(\gamma)\|_{L^2(ds_\gamma)}
    \end{equation*}
    for $\gamma \in \mathcal{X}$ with $\|\gamma-\bar{\gamma}\|_{W^{4,2}} \leq \tilde{\sigma}$.
\end{theorem}

\begin{proof}
    Let $C, \sigma > 0$, $\theta \in (0,\frac{1}{2}]$ be as in \Cref{thm:a_lojasiewicz_simon_normal}. Since $\bar{\gamma} \in C^\infty(\mathbb{S}^1,\mathbb{R}^2)$, we can use \Cref{lem:normal_reparametrization}. This gives us $\tilde{\sigma}>0$ such that for all $\gamma \in \mathcal{X} \subset W^{4,2}(\mathbb{S}^1,\mathbb{R}^2)$ with $\|\gamma-\bar{\gamma}\|_{W^{4,2}} \leq \tilde{\sigma}$ there exist a reparametrization $\Phi$ and a normal vector field $\phi \in V^\perp$ such that
    \begin{equation*}
        \gamma \circ \Phi = \bar{\gamma} + \phi
    \end{equation*}
    and with $\| \phi \|_{W^{4,2}} \leq \sigma$. Since the area $\mathcal{A}$ is invariant under orientation-preserving diffeomorphisms, we have $\mathcal{A}(\bar{\gamma}+\phi)=\mathcal{A}(\gamma \circ \Phi) = \mathcal{A}(\gamma)=a$ and thus $\bar{\gamma}+\phi \in \mathcal{X}$. By observing that the elastic energy is invariant under reparametrization and using \Cref{thm:a_lojasiewicz_simon_normal}, we have
    \begin{equation*}
    \begin{split}
        \vert \mathcal{E}(\gamma)-\mathcal{E}(\bar{\gamma})\vert^{1-\theta} &= \vert \mathcal{E}(\bar{\gamma}+\phi)-\mathcal{E}(\bar{\gamma})\vert^{1-\theta}\\
        &\leq C \norm{-\nabla \mathcal{E}(\bar{\gamma}+\phi)+\lambda(\bar{\gamma}+\phi)\nabla \mathcal{A}(\bar{\gamma}+\phi)}_{L^2(ds_{\bar{\gamma}+\phi})}.
    \end{split}
    \end{equation*}
    Observe that the Lagrange multiplier from \cref{eq:a_lagrange_multiplier}, the gradients as well as the $L^2(ds)$-norm are geometric, i.e. are invariant under reparametrization. We obtain
    \begin{equation*}
        \vert \mathcal{E}(\gamma) -\mathcal{E}(\bar{\gamma})\vert^{1-\theta} \leq C\norm{-\nabla \mathcal{E}(\gamma)+\lambda(\gamma)\nabla \mathcal{A}(\gamma)}_{L^2(ds_{\gamma})}.
    \end{equation*}
\end{proof}

This yields the result below since $\gamma_\infty$ is smooth and an area-constrained elastica.

\begin{corollary}
\label{cor:lojasiewicz_simon}
    Let $\gamma_0 \in C^\infty(\mathbb{S}^1, \mathbb{R}^2)$ be simple with $\mathcal{A}(\gamma_0)=a$ and $\mathcal{E}(\gamma_0) < 2 \mathcal{E}(\gamma^*_{\frac{a}{2}})$. For the reparametrized and translated version $\tilde{\gamma}$ of the global smooth and simple solution $\gamma$ to \cref{eq:evolution_gradient_flow} for initial data $\gamma_0$ from \Cref{thm:subconvergence} there exist $\theta \in (0,\frac{1}{2}]$, $C, \tilde{\sigma} > 0$ such that
    \begin{equation*}
        \abs{\mathcal{E}(\tilde{\gamma})-\mathcal{E}(\gamma_\infty)}^{1-\theta} \leq C\|\partial_t \gamma \|_{L^2(ds_\gamma)}
    \end{equation*}
    holds for all $t$ such that $\|\tilde{\gamma}_t-\gamma_\infty\|_{W^{4,2}} \leq \tilde{\sigma}$, where $\gamma_\infty$ is the limit from \Cref{thm:subconvergence}.
\end{corollary}

\subsection{Full convergence}

\Cref{cor:lojasiewicz_simon} is essentially enough to prove convergence of $\tilde{\gamma}$. However, in the proof of \Cref{thm:convergence}, we need to bound the full derivative $\|\partial_t \tilde{\gamma} \|_{L^2(dx)}$ by its normal component $\|\partial_t \gamma\|_{L^2(ds)}$, for which we use \Cref{lem:const_speed_reparametrization_bounds}.

\begin{proof}[Proof of \Cref{thm:convergence}]
    Let $x_0=0\in\mathbb S^1$ and for each $t\ge0$, let $\tilde \gamma_t$ be the constant speed parametrization of $\gamma(t,\cdot)-\gamma(t,x_0)$. Let $t_i \to \infty$ be the times such that $\tilde\gamma_{t_i}\to\gamma_\infty$ as $i\to\infty$ in the $C^\infty$-topology, where $\gamma_\infty$ is the area-constrained elastica from \Cref{thm:subconvergence}.

    Define the function $H(t)=(\mathcal{E}(\tilde{\gamma}_t)-\mathcal{E}(\gamma_\infty))^\theta$ with $\theta$ as in \Cref{cor:lojasiewicz_simon}. Notice $H$ is 
    decreasing. 
    
    So if $H(t)=0$ for some $t_0 \in [0,\infty)$ it stays zero for all time $t \geq t_0$. Therefore, by the energy identity we have for $t \geq t_0$
    \begin{equation*}
        0 = \partial_t \mathcal{E}(\tilde{\gamma}_t) = \partial_t \mathcal{E}(\gamma_t) =  -\|\partial_t \gamma_t\|_{L^2(ds_{\gamma_t})}^2.
    \end{equation*}
    It follows from subconvergence that $\tilde{\gamma}_t = \gamma_\infty$ for $t \geq t_0$.
    
    From now on assume $H(t) > 0$ for all $t \in [0,\infty)$. Since $H$ is decreasing and $H(t_i) \to 0$ as $i \to \infty$, we have $H(t) \to 0$ as $t \to \infty$.
    Compute for all $t$ such that $\|\tilde{\gamma}_t - \gamma_\infty\|_{W^{4,2}(dx)} \leq \delta$ for $\delta \leq \tilde{\sigma}$ with $\tilde{\sigma}$ as in \Cref{cor:lojasiewicz_simon}
    \begin{equation}
    \label{eq:L2_derivative_gamma}
    \begin{split}
        -C\partial_t H(t) &= -C\theta(H(t))^{\frac{\theta-1}{\theta}}\partial_t \mathcal{E}(\tilde{\gamma}_t)\\
        &= C\theta(H(t))^{\frac{\theta-1}{\theta}} \norm{\partial_t \gamma}_{L^2(ds_{\gamma})}^2\\
        &\geq C\theta(H(t))^{\frac{\theta-1}{\theta}} \|\partial_t \gamma\|_{L^2(ds_{\gamma})} (H(t))^{\frac{1-\theta}{\theta}}\\
        &= C\theta\|\partial_t \gamma \|_{L^2(ds_{\gamma})}\\
        &\geq \theta \|\partial_t \tilde{\gamma}_t \|_{L^2(dx)},
    \end{split}
    \end{equation}
    where in the above chain of inequalities, the constant $C>0$ absorbs the constants from the \L{}ojasiewicz--Simon inequality (\Cref{cor:lojasiewicz_simon}) and the reparametrization bound (\Cref{lem:const_speed_reparametrization_bounds}), and may change its value from line to line.

    Our goal is to show that $\|\tilde{\gamma}_t - \gamma_\infty\|_{W^{4,2}} \leq \tilde{\sigma}$ for all $t \in [t_0, \infty)$ for some $t_0 \in [0,\infty)$. Define $s_i \coloneqq \sup \{s \geq t_i \mid \forall t \in [t_i,s]\quad \|\tilde{\gamma}_t - \gamma_\infty\|_{W^{4,2}} \leq \tilde{\sigma} \}$. Then, \cref{eq:L2_derivative_gamma} holds for $t \in [t_i, s_i)$, and if $i$ is sufficiently large, the set $[t_i, s_i)$ is nonempty. From Minkowski's inequality and \cref{eq:L2_derivative_gamma} it follows
    \begin{equation}
    \label{eq:L2_difference_gamma}
        \|\tilde{\gamma}_t-\tilde{\gamma}_{t_i}\|_{L^2(dx)} \leq \int_{t_i}^t \|\partial_t \tilde{\gamma}_r\|_{L^2(dx)}dr \leq C(H(t_i)-H(t)) \to 0
    \end{equation}
    as $i \to \infty$. Assume that all $s_i$ are finite. By continuity, we get \cref{eq:L2_difference_gamma} also for $t=s_i$. From a similar argument as in the subconvergence result \Cref{thm:subconvergence}, we obtain $\tilde{\gamma}_{s_i} \to \psi$ in $C^\infty$ up to a subsequence. Due to continuity and the definition of $s_i$ we have $\|\psi - \gamma_\infty \|_{W^{4,2}} = \tilde{\sigma}$. However, \cref{eq:L2_difference_gamma} yields
    \begin{equation*}
        \|\psi - \gamma_\infty \|_{L^2(dx)} = \lim_{i \to \infty} \|\tilde{\gamma}_{s_i}-\tilde{\gamma}_{t_i}\|_{L^2(dx)} = 0.
    \end{equation*}
    Thus, $\psi = \gamma_\infty$ in $L^2$ and as they are both smooth we have $\psi = \gamma_\infty$ in $W^{4,2}$, a contradiction. So there is $i \in \mathbb{N}$ such that $s_i=\infty$. Hence, for all $t \geq t_i$ it holds $\|\tilde{\gamma}_t-\gamma_\infty \|_{W^{4,2}} \leq \tilde{\sigma}$.

    As before in \cref{eq:L2_difference_gamma} we deduce
    \begin{equation*}
        \|\tilde{\gamma}_{t} - \tilde{\gamma}_{\tilde{t}}\|_{L^2(dx)} \leq \int_{\tilde{t}}^t \|\partial_t \tilde{\gamma}_r\|_{L^2(dx)} dr \leq C(H(\tilde{t})-H(t)) \to 0
    \end{equation*}
    as $t,\tilde{t} \to \infty$. So every subsequence $\tilde{\gamma}_{t_i}$ with $t_i \to \infty$ of $(\tilde{\gamma}(t,\cdot))_{t \in [0,\infty)}$ is a Cauchy sequence in $L^2(\mathbb{S}^1, \mathbb{R}^2)$ and converges to $\gamma_\infty$. Then it already has to converge to $\gamma_\infty$ in $C^\infty$ by the compactness argument from the proof of \Cref{thm:subconvergence}.
\end{proof}

\section{Energy profile curve}
\label{sec:energy_profile}
In order to find a length bound for the area-constrained elastic flow, one can also analyze the function that maps a given length to the minimal energy of a curve with said length and fixed area. Here, we fix the area to $\pi$. In this section, we will only consider curves with winding number $n_\gamma=1$.

\subsection{Continuity}

\begin{proposition}
\label{prop:weak_lower_semicontinuity}
    The function $\mathcal{E}: W^{2,2}_{\text{Imm}}(\mathbb{S}^1, \mathbb{R}^2) \to [0,\infty)$ is weakly lower semicontinuous.
\end{proposition}

\begin{proof}
    Let $(\gamma_n)_{n \in \mathbb{N}} \subset W^{2,2}_{\text{Imm}}(\mathbb{S}^1, \mathbb{R}^2)$ be a sequence such that $\gamma_n \rightharpoonup \gamma$ weakly for some $\gamma \in W^{2,2}_{\text{Imm}}(\mathbb{S}^1, \mathbb{R}^2)$. Without loss of generality, we may pass to a subsequence, which will be still denoted by $\gamma_n$, such that $\lim_{n \to \infty}\mathcal{E}(\gamma_n) = \liminf_{n \to \infty} \mathcal{E}(\gamma_n)$. We know from the Sobolev embedding that $W^{2,2}_{\text{Imm}}(\mathbb{S}^1, \mathbb{R}^2)$ continuously embeds in $C^{1,\frac{1}{2}}(\mathbb{S}^1, \mathbb{R}^2)$. The weak convergence implies that $(\gamma_n)$ is bounded in $C^{1,\frac{1}{2}}(\mathbb{S}^1, \mathbb{R}^2)$. By Arzel\`a--Ascoli, after passing to a subsequence, $\gamma_n \to \gamma$ in $C^1(\mathbb{S}^1, \mathbb{R}^2)$. Hence, $\abs{\partial_x \gamma_n} \geq c > 0$ uniformly, because the limit $\gamma$ is an immersion. Moreover, the weak convergence in $W^{2,2}(\mathbb{S}^1, \mathbb{R}^2)$ implies $\partial_x^2 \gamma_n$ weakly converges to $\partial_x^2 \gamma$ in $L^2(\mathbb{S}^1, \mathbb{R}^2)$. Let us write the energy in a more general formulation with \Cref{prop:a_explicit_formulation}
    \begin{equation*}
        \mathcal{E}(w) = I(\partial_x w, \partial_x^2 w) = \int_{\mathbb{S}^1} f(\partial_x w, \partial_x^2 w) dx
    \end{equation*}
    for $f(u,v) = \big\vert\frac{v}{\abs{u}^2}-\frac{(v\cdot u)}{\abs{u}^4}u\big\vert^2 \abs{u}$ and $I(g,h) = \int_{\mathbb{S}^1} f(g,h) dx$. If we can show lower semicontinuity with respect to $L^\infty(\mathbb{S}^1, \mathbb{R}^2) \times L^2(\mathbb{S}^1, \mathbb{R}^2)$ where $L^\infty$ is equipped with the standard topology and $L^2$ with the weak topology, we are done. For this, we want to apply \cite[Theorem~1]{ioffe}. First of all, $f$ is continuous away from $(0,v)$, which is sufficient for us, since our sequence $\gamma_n$ has a first derivative bounded away from zero. Moreover, for fixed $u$ the map $v \mapsto \frac{v}{\abs{u}^2}-\frac{(v\cdot u)}{\abs{u}^4}u$ is linear. Consequently, $v \mapsto \big\vert\frac{v}{\abs{u}^2}-\frac{(v\cdot u)}{\abs{u}^4}u\big\vert^2 \abs{u}$ is convex. Because $f \geq 0$, it trivially satisfies the lower compactness property. We are left with verifying the following conditions:
    \begin{description}
        \item[(H1)] Let $g_k \to 0$ in $L^\infty(\mathbb{S}^1, \mathbb{R}^2)$ and $h_k \rightharpoonup 0$ weakly in $L^2(\mathbb{S}^1, \mathbb{R}^2)$. Moreover, let $T_k \subset \mathbb{S}^1$ such that $\int_{T_k}dx \to 0$. It trivially holds $\mathbbm{1}_{T_k} g_k \to 0$ in $L^\infty(\mathbb{S}^1, \mathbb{R}^2)$. For $\phi \in (L^2(\mathbb{S}^1, \mathbb{R}^2))^* = L^2(\mathbb{S}^1, \mathbb{R}^2)$ it holds
        \begin{equation*}
            \Bigg(\int_{\mathbb{S}^1} \phi \cdot (\mathbbm{1}_{T_k}h_k) dx\Bigg)^2 \leq \int_{T_k} \abs{\phi}^2 dx \int_{\mathbb{S}^1} \abs{h_k}^2 dx \to 0
        \end{equation*}
        by dominated convergence and the fact that weakly converging sequences are bounded.
        \item[(H2)] The topology of $L^\infty(\mathbb{S}^1, \mathbb{R}^2)$ is not weaker than the topology of convergence in measure. Since
        \begin{equation*}
            (L^1(\mathbb{S}^1, \mathbb{R}^2))^* = L^\infty(\mathbb{S}^1, \mathbb{R}^2) \subset L^2(\mathbb{S}^1, \mathbb{R}^2) = (L^2(\mathbb{S}^1, \mathbb{R}^2))^*,
        \end{equation*}
        the weak topology of $L^2(\mathbb{S}^1, \mathbb{R}^2)$ is not weaker than the weak topology of $L^1(\mathbb{S}^1, \mathbb{R}^2)$.
    \end{description}
    \cite[Theorem~1]{ioffe} yields
    \begin{equation*}
        \liminf_{n \to \infty} I(\partial_x \gamma_n, \partial_x^2 \gamma_n) \geq I(\partial_x \gamma, \partial_x^2 \gamma).
    \end{equation*}
\end{proof}

\begin{remark}
\label{rem:existence_minimizer}
    This proposition shows that the infimum in \cref{eq:energy_profile_curve} is actually a minimum. Indeed, choose a minimizing sequence $\gamma_n \in W^{2,2}_{\text{Imm}}(\mathbb{S}^1, \mathbb{R}^2)$ with $\mathcal{A}(\gamma_n)=\pi$ and $\mathcal{L}(\gamma_n) = \ell$. Without loss of generality we may assume that $\abs{\partial_x \gamma_n} = \frac{\ell}{2\pi}$. Moreover, since length, area, and energy are translation invariant, fixing $\gamma_n(0)=0$ by translating the curves ensures $\norm{\gamma_n}_{L^\infty}$ is uniformly bounded. It holds $\norm{\partial_x^2 \gamma_n}_{L^2(dx)}^2 = \frac{\ell}{2\pi} \mathcal{E}(\gamma_n)$. Thus, $\gamma_n$ is bounded in $W^{2,2}_{\text{Imm}}(\mathbb{S}^1, \mathbb{R}^2)$ and consequently $\gamma_n \rightharpoonup \gamma$ weakly up to a subsequence for some $\gamma \in W_{\text{Imm}}^{2,2}(\mathbb{S}^1, \mathbb{R}^2)$. As in the proof of \Cref{prop:weak_lower_semicontinuity}, we have $\gamma_n \to \gamma$ in $C^1(\mathbb{S}^1, \mathbb{R}^2)$ by potentially reducing to a further subsequence. As the length and area are continuous with respect to the $C^1$-topology, we conclude that $\mathcal{L}(\gamma) = \ell$ and $\mathcal{A}(\gamma) = \pi$. By the same reason, we have $\abs{\partial_x \gamma} = \frac{\ell}{2 \pi} > 0$. \Cref{prop:weak_lower_semicontinuity} yields the claim.
\end{remark}

\begin{theorem}
    For given length $\ell$ and area $a$, the constant speed minimizer $\gamma$ of $\mathcal{E}: W^{2,2}_{\text{Imm}}(\mathbb{S}^1, \mathbb{R}^2) \to [0, \infty)$ subject to fixed area $a$ and length $\ell$ is in $C^\infty(\mathbb{S}^1, \mathbb{R}^2)$.
\end{theorem}

\begin{proof}
    In this proof we will use $\partial_s$ instead of $\partial_x$ for convenience, which is justified because they only differ by a multiplicative constant. Since $\gamma$ is a minimizer, it formally satisfies the Euler-Lagrange equation
    \begin{equation*}
        0 = 2\partial_s^2 \kappa + \kappa^3 + \lambda_1 + \lambda_2 \kappa
    \end{equation*}
    with respect to $ds$ for some $\lambda_1, \lambda_2 \in \mathbb{R}$. Rewriting this in terms of $\gamma$ gives
    \begin{equation*}
        0 = 2 \partial_s^4 \gamma + 3 \partial_s (\abs{\partial_s^2 \gamma}^2 \partial_s \gamma) + \lambda_1 Q(\partial_s \gamma) + \lambda_2 \partial_s^2 \gamma
    \end{equation*}
    for the rotation $Q = \begin{pmatrix}
        0 & -1\\
        1 & 0
    \end{pmatrix}$. So for arbitrary $\varphi \in C^\infty(\mathbb{S}^1, \mathbb{R}^2)$ we have
    \begin{equation}
    \label{eq:bootstrapping}
        \int_{\mathbb{S}^1} \partial_s^2 \gamma \cdot (\partial_s^2 \varphi) ds = \frac{1}{2} \int_{\mathbb{S}^1} \big(3 \abs{\partial_s^2 \gamma}^2 \partial_s \gamma + \lambda_1 Q\gamma + \lambda_2 \partial_s \gamma\big) \cdot (\partial_s \varphi) ds.
    \end{equation}
    Define $f \coloneqq 3 \abs{\partial_s^2 \gamma}^2 \partial_s \gamma + \lambda_1 Q\gamma + \lambda_2 \partial_s \gamma$. By the Sobolev embedding we have the continuous embedding $W^{2,2}(\mathbb{S}^1, \mathbb{R}^2) \hookrightarrow C^1(\mathbb{S}^1, \mathbb{R}^2)$. Hence, $f \in L^1(\mathbb{S}^1, \mathbb{R}^2)$. \Cref{eq:bootstrapping} shows that $\partial_s^2 \gamma \in W^{1,1}(\mathbb{S}^1, \mathbb{R}^2)$ with $\partial_s^3 \gamma = -\frac{1}{2}f + c$ where $c \in \mathbb{R}^2$ is constant. It arises because we only test against functions ($\partial_s \varphi$) with zero mean in \cref{eq:bootstrapping}. We have the embedding $W^{1,1}(\mathbb{S}^1, \mathbb{R}^2) \hookrightarrow C^0(\mathbb{S}^1, \mathbb{R}^2)$ and hence $\partial_s^2 \gamma \in C^0(\mathbb{S}^1, \mathbb{R}^2)$. By definition, $f \in C^0(\mathbb{S}^1, \mathbb{R}^2)$ and from $\partial_s^3 \gamma = -\frac{1}{2}f+c$ it follows $\gamma \in C^3(\mathbb{S}^1, \mathbb{R}^2)$. Iterating this argument allows us to conclude $\gamma \in C^\infty(\mathbb{S}^1, \mathbb{R}^2)$.
\end{proof}

Let us show continuity of $E: [2\pi, \infty) \to [0,\infty)$ defined in \cref{eq:energy_profile_curve}.

\begin{lemma}
\label{lem:lower_semicontinuity}
    The function $E: [2\pi, \infty) \to [0,\infty)$ is lower semicontinuous.
\end{lemma}

\begin{proof}
    Suppose $(\ell_n)_{n \in \mathbb{N}} \subset [2\pi, \infty)$ with $\ell_n \to \ell \neq \infty$. Without loss of generality, we may assume that $\liminf_{n \to \infty} E(\ell_n) < \infty$ (otherwise the claim holds trivially). Let us pass to a subsequence, still denoted by $\ell_n$, such that $\lim_{n \to \infty} E(\ell_n) = \liminf_{n \to \infty} E(\ell_n)$. Let $\gamma_n \in W^{2,2}_{\text{Imm}}(\mathbb{S}^1, \mathbb{R}^2)$ be curves realizing these energies with length $\ell_n$ and area $\pi$, i.e. $E(\ell_n) = \mathcal{E}(\gamma_n)$. Since the energies and lengths are bounded, we find by the same method as in the proof of \Cref{rem:existence_minimizer} that $\gamma_n \rightharpoonup \tilde{\gamma} \in W^{2,2}_{\text{Imm}}(\mathbb{S}^1, \mathbb{R}^2)$ weakly up to a subsequence (and translation and reparametrization) with $\mathcal{A}(\tilde{\gamma})=\pi$, $\mathcal{L}(\tilde{\gamma})=\ell$. \Cref{prop:weak_lower_semicontinuity} gives
    \begin{equation*}
        E(\ell) \leq \mathcal E(\tilde{\gamma}) \leq \liminf_{n \to \infty} \mathcal{E}(\gamma_n) = \liminf_{n \to \infty} E(\ell_n).
    \end{equation*}
\end{proof}

\begin{lemma}
\label{lem:length_increasing_variation}
    Let $\gamma \in C^\infty(\mathbb{S}^1, \mathbb{R}^2)$ be an immersed curve whose image is not a circle. Define $\gamma_t \coloneqq \gamma + t(\kappa- \frac{2 \pi n_\gamma}{\mathcal{L}(\gamma)})\nu$. It holds
    \begin{equation*}
        \left. \frac{d}{dt} \right|_{t=0} \mathcal{L}(\gamma_t) \neq 0, \quad \left. \frac{d}{dt} \right|_{t=0} \mathcal{A}(\gamma_t) = 0,
    \end{equation*}
    and $t \mapsto \mathcal{E}(\gamma_t)$ is differentiable at $t=0$.
\end{lemma}

\begin{proof}
    Compute
    \begin{equation*}
        \left. \frac{d}{dt} \right|_{t=0} \mathcal{L}(\gamma_t) = -\int_{\mathbb{S}^1} \kappa \Big(\kappa- \frac{2 \pi n_\gamma}{\mathcal{L}(\gamma)}\Big)ds = - \mathcal{E}(\gamma) + \frac{4 \pi ^2 n_\gamma^2}{\mathcal{L}(\gamma)} < 0.
    \end{equation*}
    The last inequality is due to Cauchy-Schwarz and the nonconstant curvature by assumption
    \begin{equation}
    \label{eq:energy_length_inequality}
        4 \pi^2 n_\gamma^2 = \Big(\int_{\mathbb{S}^1} \kappa ds\Big)^2 < \int_{\mathbb{S}^1} \kappa^2 ds \int_{\mathbb{S}^1} ds.
    \end{equation}
    The derivative of the area is given by
    \begin{equation*}
        \left. \frac{d}{dt} \right|_{t=0} \mathcal{A}(\gamma_t) = -\int_{\mathbb{S}^1} \kappa- \frac{2 \pi n_\gamma}{\mathcal{L}(\gamma)} ds = 0.
    \end{equation*}
\end{proof}

\begin{lemma}
\label{lem:upper_semicontinuity}
    The function $E$ is upper semicontinuous on $(2\pi, \infty)$.
\end{lemma}

\begin{proof}
    Let $\ell \in (2\pi, \infty)$ and $\gamma \in C^\infty(\mathbb{S}^1, \mathbb{R}^2)$ be a minimizer of the energy $\mathcal{E}: W^{2,2}_{\text{Imm}}(\mathbb{S}^1, \mathbb{R}^2) \to [0,\infty)$ subject to fixed length $\ell$ and area $\pi$. Since $\gamma$ cannot be a circle by the isoperimetric inequality and the winding number is fixed $n_\gamma=1$, the curvature is nonconstant. Consider the variation $\gamma_t$ from \Cref{lem:length_increasing_variation} and define $\tilde{\gamma}_t \coloneqq \sqrt{\frac{\pi}{\mathcal{A}(\gamma_t)}} \gamma_t$. We have
    \begin{equation}
    \label{eq:length_increasing_variation}
    \begin{split}
        \left. \frac{d}{dt} \right|_{t=0} \mathcal{L}(\tilde{\gamma}_t) &= \left. \frac{d}{dt} \right|_{t=0} \Big(\sqrt{\frac{\pi}{\mathcal{A}(\gamma_t)}} \mathcal{L}(\gamma_t)\Big) = \left. \frac{d}{dt} \right|_{t=0} \mathcal{L}(\gamma_t) \neq 0,\\
        \left. \frac{d}{dt} \right|_{t=0} \mathcal{E}(\tilde{\gamma}_t) &= \left. \frac{d}{dt} \right|_{t=0} \Big(\sqrt{\frac{\mathcal{A}(\gamma_t)}{\pi}} \mathcal{E}(\gamma_t)\Big) = \left. \frac{d}{dt} \right|_{t=0} \mathcal{E}(\gamma_t),
    \end{split}
    \end{equation}
    and $\tilde{\gamma}_t$ has constant area $\pi$. Since $\left. \frac{d}{dt} \right|_{t=0}\mathcal{L}(\tilde{\gamma}_t) \neq 0$, by the Inverse Function Theorem, there is a continuously differentiable inverse $\psi: (\ell-\varepsilon, \ell+\varepsilon) \to \mathbb{R}$ of $t \mapsto \mathcal{L}(\tilde{\gamma}_t)$ for $\varepsilon > 0$ onto a neighborhood of $0$. By continuity of $t \mapsto \mathcal{E}(\tilde{\gamma}_t)$ it holds
    \begin{equation*}
        \limsup_{l \to \ell}E(l) \leq \limsup_{l \to \ell} \mathcal{E}(\tilde{\gamma}_{\psi(l)}) = \limsup_{t \to 0} \mathcal{E}(\tilde{\gamma}_t) = E(\ell).
    \end{equation*}
\end{proof}

\begin{lemma}
\label{lem:upper_semicontinuity_2pi}
    The function $E$ is upper semicontinuous at $2\pi$.
\end{lemma}

\begin{proof}
    The idea is to construct an explicit sequence of curves with greater length than the circle which converge to the circle and whose energy also converges to the energy of the circle. Specifically, let $\gamma_a(x) = (a \cos(x), \frac{1}{a} \sin(x))$ for $a \geq 1$. 
    For $a \to 1$ this function converges smoothly to the circle. Observe
    \begin{equation*}
        \mathcal{A}(\gamma_a) = \int_{\mathbb{S}^1} \cos^2(x) dx = \pi.
    \end{equation*}
    The isoperimetric inequality proves that $\gamma_a$ has length greater than the circle for $a > 1$. Moreover, the map $a \mapsto \mathcal{L}(\gamma_a)$ is continuous with $\mathcal{L}(\gamma_1) = 2\pi$. Thus, for any length $\ell$ sufficiently close to $2\pi$, there is $a(\ell) \geq 1$ such that $\mathcal{L}(\gamma_{a(\ell)})=\ell$. Because $\gamma_{a(\ell)}$ has area $\pi$ and length $\ell$, it is a valid test curve for the infimum $E(\ell)$, meaning $E(\ell) \leq E(\gamma_{a(\ell)})$. Taking the $\limsup$ yields
    \begin{equation*}
        \limsup_{\ell \to 2\pi} E(\ell) \leq \lim_{a \to 1} \mathcal{E}(\gamma_a) = 2\pi = E(2\pi).
    \end{equation*}
\end{proof}

\begin{theorem}
\label{thm:energy_continuity}
    The function $E$ is continuous on $[2\pi, \infty)$ and for $\ell \to \infty$ we have $E(\ell) \to 0$.
\end{theorem}

\begin{proof}
    Continuity follows from \Cref{lem:lower_semicontinuity}, \Cref{lem:upper_semicontinuity} and \Cref{lem:upper_semicontinuity_2pi}. For the second part, we construct an explicit sequence of curves with area $\pi$, length that goes to $\infty$ and energy that goes to zero. These curves are of the form shown in \Cref{fig:triple_eight_energy_to_zero}.

    The left 'belly' is formed by three-quarters of a circle, connected to the central 'belly' via straight line segments. By matching the tangents, the resulting curve is smooth except for the 'glueing points', where it is $C^1$. Hence, it is a valid test curve in $W^{2,2}(\mathbb{S}^1, \mathbb{R}^2)$. The upper and lower parts of the central 'belly' each consist of a quarter circle. The right half of the curve is symmetric to the left half.
    By scaling the radii of the circles appropriately, we can ensure that the enclosed area of the curve remains constant, since the central 'belly' contributes negatively to the total area. Consequently, we can increase the overall size of the curve such that its length tends to infinity while the enclosed area remains fixed. The elastic energy of the curve can then be bounded from above by the energy of three circles, whose radii tend to infinity. This implies that the energy of the curve converges to zero.
\end{proof}

\Cref{thm:energy_continuity} ensures that $E: [2\pi, \infty) \to [0,\infty)$ actually attains its supremum. In order to find a noteworthy regime of convergence later on, we demonstrate that the maximum is not attained by the circle.

\begin{theorem}
\label{thm:energy_maximum}
    Denote by $\ell^*$ a global maximum of the function $E: [2\pi, \infty) \to [0, \infty)$. Then $E(\ell^*)>2\pi$, proving that the maximum is not achieved by the circle.
\end{theorem}

\begin{proof}
    From \cite[Lemma~7]{Watanabe2002} it follows that a minimizer $\gamma$ of all curves in $W^{2,2}_{\text{Imm}}(\mathbb{S}^1, \mathbb{R}^2)$ with fixed area $\pi$ and length $\ell$ with $2\pi < \ell \leq \frac{5}{2} \pi$ is simple. \cite[Theorem~1.1]{bucur2014newisoperimetricinequalityelasticae} yields $\mathcal{E}(\gamma) > 2\pi$. Since $\ell^*$ is a global maximizer, it follows that $E(\ell^*) \geq E(\ell) = \mathcal{E}(\gamma) > 2 \pi$.
\end{proof}

\begin{proof}[Proof of \Cref{cor:energy_profile_convergence}]
    If there is $t \in (0, \infty)$ with $\mathcal{L}(\gamma_t) \geq \ell^*$, by continuity there exists $t_0 \leq t$ with $\mathcal{L}(\gamma_{t_0}) = \ell^*$. This implies $\mathcal{E}(\gamma_{t_0}) \geq \mathcal{E}(\ell^*) > \mathcal{E}(\gamma_0)$, which is a contradiction since the energy is non-increasing along the flow.

    The convergence result can be proven in the same way as \Cref{thm:convergence}.
\end{proof}

Using the same barrier argument, one can also establish a lower bound for the length.

\begin{corollary}
    Let $\gamma_0 \in C^\infty(\mathbb{S}^1, \mathbb{R}^2)$ with $\mathcal{A}(\gamma_0) = \pi$, $\mathcal{L}(\gamma_0) > \ell^*$ and $\mathcal{E}(\gamma_0) < E(\ell^*)$ for the maximum as in \Cref{thm:energy_maximum}. Then, the solution $\gamma$ to \cref{eq:evolution_gradient_flow} has length strictly bounded from below by $\ell^*$.
\end{corollary}

\subsection{Critical points}

Numerical simulations reveal curves whose lengths diverge along the gradient flow \cref{eq:evolution_gradient_flow}, as seen in the triple-eight configuration (\Cref{fig:triple-eight}). Our goal is to derive explicit conditions for this behavior. However, the flow could also converge to a local minimum of $E: [2\pi, \infty) \to [0,\infty)$ or an area-constrained elastica. We show that local minima are area-constrained elasticae, so we can focus on the latter to understand halting mechanisms.

\begin{lemma}
\label{lem:diff_minimum}
    Let $\ell \in (2\pi, \infty)$ be a local minimum of $E$. Then, $E$ is differentiable at $\ell$ with $\left. \frac{d}{dl} \right|_{l=\ell}E(l) = 0$.
\end{lemma}

\begin{proof}
    Let $\gamma \in C^\infty(\mathbb{S}^1, \mathbb{R}^2)$ be a minimizer of $\mathcal{E}: W^{2,2}_{\text{Imm}}(\mathbb{S}^1, \mathbb{R}^2) \to [0,\infty)$ subject to fixed length $\ell$ and area $\pi$. Since $\gamma$ cannot be a circle by the isoperimetric inequality and the winding number is fixed $n_\gamma=1$, the curvature is nonconstant. 
    Consider the variation $\gamma_t$ from \Cref{lem:length_increasing_variation} and define $\tilde{\gamma}_t \coloneqq \sqrt{\frac{\pi}{\mathcal{A}(\gamma_t)}} \gamma_t$. We have \cref{eq:length_increasing_variation}, hence $\left. \frac{d}{dt} \right|_{t=0} \mathcal{L}(\tilde{\gamma}_t) \neq 0$,
    and $\tilde{\gamma}_t$ has constant area $\pi$. By the Inverse Function Theorem, there is a continuously differentiable inverse $\psi: (\ell-\varepsilon, \ell+\varepsilon) \to \mathbb{R}$ of $t \mapsto \mathcal{L}(\tilde{\gamma}_t)$ for $\varepsilon > 0$ onto a neighborhood of $0$. Since $\ell$ is a local minimum of $E$, there is $0 < \delta \leq \varepsilon$ with $E(l) \geq E(\ell)$ for $l \in (\ell-\delta, \ell+\delta)$. Hence, it holds
    \begin{equation*}
        \mathcal{E}(\tilde{\gamma}_{\psi(l)}) \geq E(l) \geq E(\ell) = \mathcal{E}(\tilde{\gamma}_{\psi(\ell)})
    \end{equation*}
    for $l \in (\ell-\delta, \ell+\delta)$. To show differentiability, we consider the one-sided limits of the difference quotient. For $l > \ell$, dividing by $l-\ell > 0$ preserves the inequality
    \begin{equation*}
        0 \leq \frac{E(l)-E(\ell)}{l-\ell} \leq \frac{\mathcal{E}(\tilde{\gamma}_{\psi(l)})-\mathcal{E}(\tilde{\gamma}_{\psi(\ell)})}{l-\ell}.
    \end{equation*}
    Taking the limit as $l \searrow \ell$, the right-hand side goes to $\left. \frac{d}{dl} \right|_{l=\ell} \mathcal{E}(\tilde{\gamma}_{\psi(l)}) = 0$, since $l \mapsto \mathcal{E}(\tilde{\gamma}_{\psi(l)})$ is differentiable and attains a local minimum at $\ell$. Thus, the right derivative of $E$ is $0$.

    For $l < \ell$, dividing by $l-\ell < 0$ reverses the inequalities
    \begin{equation*}
        0 \geq \frac{E(l)-E(\ell)}{l-\ell} \geq \frac{\mathcal{E}(\tilde{\gamma}_{\psi(l)})-\mathcal{E}(\tilde{\gamma}_{\psi(\ell)})}{l-\ell}.
    \end{equation*}
    Taking the limit as $l \nearrow \ell$, the right-hand side again goes to $0$, meaning the left derivative of $E$ is also $0$. 

    Since both one-sided limits exist and equal $0$, $E$ is differentiable at $\ell$ with $\left. \frac{d}{dl} \right|_{l=\ell} E(l) = 0$.
\end{proof}

\begin{theorem}
\label{thm:critical_point_area_constrained_elastica}
    If $E$ is differentiable at some $\ell \in (2\pi, \infty)$ with $\left. \frac{d}{dl} \right|_{l=\ell} E(l) = 0$, then every minimizer $\gamma$ of $\mathcal{E}$ among all curves with length $\ell$ and area $\pi$ is an area-constrained elastica.
\end{theorem}

\begin{proof}
    Since $\gamma$ is a minimizer among all curves with area $a$ and length $\ell$, there are $\lambda_1, \lambda_2 \in \mathbb{R}$ with
    \begin{equation*}
        \nabla \mathcal{E}(\gamma) + \lambda_1 \nabla \mathcal{A}(\gamma) + \lambda_2 \nabla \mathcal{L}(\gamma) = 0.
    \end{equation*}
    Testing the Euler-Lagrange equation with the normal variations $\nu, \kappa \nu$ gives
    \begin{equation*}
    \begin{split}
        \int_{\mathbb{S}^1} 2 \partial_s^2 \kappa + \kappa^3ds &= \mathcal{L}(\gamma) \lambda_1 + 2 \pi n_\gamma \lambda_2,\\
        \int_{\mathbb{S}^1} (2 \partial_s^2 \kappa + \kappa^3) \kappa ds &= 2 \pi n_\gamma \lambda_1 + \mathcal{E}(\gamma) \lambda_2.
    \end{split}
    \end{equation*}
    Hence, we look for a solution of the linear system of equations
    \begin{equation}
    \label{eq:linear_system}
        \begin{pmatrix}
            \mathcal{L}(\gamma) & 2 \pi n_\gamma\\
            2 \pi n_\gamma & \mathcal{E}(\gamma)
        \end{pmatrix}
        \begin{pmatrix}
            \lambda_1\\
            \lambda_2
        \end{pmatrix}
        = \begin{pmatrix}
            \int_{\mathbb{S}^1} 2 \partial_s^2 \kappa + \kappa^3 ds\\
            \int_{\mathbb{S}^1} (2 \partial_s^2 \kappa + \kappa^3) \kappa ds
        \end{pmatrix}.
    \end{equation}
    We only consider curves with $n_\gamma=1$. Hence, from $\ell \neq 2\pi$ it follows that the curvature is nonconstant.
    Compute
    \begin{equation*}
        \det \begin{pmatrix}
            \mathcal{L}(\gamma) & 2 \pi n_\gamma\\
            2 \pi n_\gamma & \mathcal{E}(\gamma)
        \end{pmatrix}
        = \mathcal{L}(\gamma) \mathcal{E}(\gamma) - 4 \pi^2 n_\gamma^2 > 0
    \end{equation*}
    by \cref{eq:energy_length_inequality}. So there is a unique solution, which we compute now.
    Consider the variation $\gamma_t$ from \Cref{lem:length_increasing_variation} and define $\tilde{\gamma}_t \coloneqq \sqrt{\frac{\pi}{\mathcal{A}(\gamma_t)}} \gamma_t$. It holds \cref{eq:length_increasing_variation} and $\tilde{\gamma}_t$ has constant area $\pi$. The Inverse Function Theorem yields a continuously differentiable inverse of $t \mapsto \mathcal{L}(\tilde{\gamma}_t)$ named $\psi: (\ell -\varepsilon, \ell + \varepsilon) \to \mathbb{R}$ for some $\varepsilon > 0$ onto a neighborhood of $0$. Compute
    \begin{equation*}
        \left. \frac{d}{dl} \right|_{l=\ell} \mathcal{E}(\tilde{\gamma}_{\psi(l)}) = \psi'(\ell) \left. \frac{d}{dt} \right|_{t=0} \mathcal{E}(\tilde{\gamma}_t) = \frac{1}{\left. \frac{d}{dt} \right|_{t=0} \mathcal{L}(\tilde{\gamma}_t)} \left. \frac{d}{dt} \right|_{t=0} \mathcal{E}(\tilde{\gamma}_t).
    \end{equation*}
    Observe for $l \in (\ell, \ell+\varepsilon)$
    \begin{equation*}
        \frac{\mathcal{E}(l)-\mathcal{E}(\ell)}{l-\ell} \leq \frac{\mathcal{E}(\tilde{\gamma}_{\psi(l)})-\mathcal{E}(\tilde{\gamma}_{\psi(\ell)})}{l-\ell}.
    \end{equation*}
    Taking the limit $l \searrow \ell$ shows $0 \leq \lim_{l \searrow \ell} \frac{\mathcal{E}(\tilde{\gamma}_{\psi(l)})-\mathcal{E}(\tilde{\gamma}_{\psi(\ell)})}{l-\ell}$. Similarly, the inequality reverses for $l \nearrow \ell$ due to the division by a negative number $0 \geq \lim_{l \nearrow \ell} \frac{\mathcal{E}(\tilde{\gamma}_{\psi(l)})-\mathcal{E}(\tilde{\gamma}_{\psi(\ell)})}{l-\ell}$. Therefore, it holds $\left. \frac{d}{dt} \right|_{t=0} \mathcal{E}(\tilde{\gamma}_t) = \left. \frac{d}{dl} \right|_{l=\ell} \mathcal{E}(\tilde{\gamma}_{\psi(l)}) = 0$. We obtain together with \cref{eq:length_increasing_variation} and the definition of the variation from \Cref{lem:length_increasing_variation}
    \begin{equation*}
    \begin{split}
        0 = \left. \frac{d}{dt} \right|_{t=0} \mathcal{E}(\tilde{\gamma}_t) = \left. \frac{d}{dt} \right|_{t=0} \mathcal{E}(\gamma_t) &= \int_{\mathbb{S}^1} (2 \partial_s^2 \kappa + \kappa^3) \Big(\kappa- \frac{2 \pi n_\gamma}{\mathcal{L}(\gamma)}\Big)ds,\\
        \frac{2 \pi n_\gamma}{\mathcal{L}(\gamma)}\int_{\mathbb{S}^1} 2 \partial_s^2 \kappa + \kappa^3 ds &= \int_{\mathbb{S}^1} (2 \partial_s^2 \kappa + \kappa^3) \kappa ds.
    \end{split}
    \end{equation*}
    Consequently, the unique solution to \cref{eq:linear_system} is given by $\lambda_1 = \frac{1}{\mathcal{L}(\gamma)} \int_{\mathbb{S}^1} 2 \partial_s^2 \kappa + \kappa^3 ds, \lambda_2 = 0$.
\end{proof}

Let us analyze area-constrained elastica. Suppose $\gamma \in C^\infty(\mathbb{S}^1, \mathbb{R}^2)$ is an area-constrained elastica with constant speed and area $\pi$, i.e.
\begin{equation}
\label{eq:area_constrained_elastica}
    2 \partial_s^2 \kappa + \kappa^3 + \lambda = 0
\end{equation}
for $\lambda \in \mathbb{R}$. Rescaling $\gamma$ by $\big(-\frac{\lambda}{2}\big)^{\frac{1}{3}}$ and considering the parametrization $[0,\mathcal{L}(\gamma)]$ of $\mathbb{S}^1$, we obtain
\begin{equation}
\label{eq:area_constrained_elastica_normalized}
    -\kappa'' - \frac{1}{2}\kappa^3 + 1 = 0.
\end{equation}
Since solutions to \cref{eq:area_constrained_elastica} are rescalings of solutions to \cref{eq:area_constrained_elastica_normalized}, it suffices to analyze the latter. This has already been done in \cite{bucur2014newisoperimetricinequalityelasticae} for the analysis of the optimal drop. \cite{bucur2014newisoperimetricinequalityelasticae} derived the following optimality conditions, where (3) and (4) are obtained using that a solution $\kappa$ to the equation \cref{eq:area_constrained_elastica_normalized} has to correspond to a closed curve $\gamma$.

\begin{theorem}[{\cite[Theorem~2.5]{bucur2014newisoperimetricinequalityelasticae}}]
\label{thm:optimality_conditions}
    Let $\ell>0$ and $\gamma \in W^{2,2}_{\text{Imm}}((0,\ell), \mathbb{R}^2)$ be a weak solution of \cref{eq:area_constrained_elastica_normalized} parametrized by arc-length. Then, $\gamma \in C^\infty((0,\ell), \mathbb{R}^2)$ and 
    \begin{enumerate}[label=(\arabic*)]
        \item $\kappa'' = -\frac{1}{2} \kappa^3 + 1$,
        \item $\kappa'^2 = -\frac{1}{4} \kappa^4 + 2 \kappa + 2C$ for some constant $C \in \mathbb{R}$,
        \item there is $p \in \mathbb{R}^2$ such that for all $x \in (0, \ell)$ it holds $\abs{\gamma(x)-p}^2 = 2 \kappa(x) + 2C$ for some constant $C$,
        \item there is $p \in \mathbb{R}^2$ such that for all $x \in (0, \ell)$ it holds $(p-\gamma(x))\cdot \nu(x) = \frac{1}{2} \kappa^2(x)$.
    \end{enumerate}
\end{theorem}

\begin{remark}[{\cite[Remark~2.6]{bucur2014newisoperimetricinequalityelasticae}}]
    The constant $C$ in (2) and (3) is identical, as is the point $p$ in (3) and (4).
\end{remark}

\begin{lemma}
\label{lem:bound_C}
    Let $\gamma \in W^{2,2}_{\text{Imm}}(\mathbb{S}^1, \mathbb{R}^2)$ be a weak solution of \cref{eq:area_constrained_elastica_normalized} parametrized by arc-length which is not a circle. Then $C \geq 0$ for the constant in (2) and (3).
\end{lemma}

\begin{proof}
    From \cite[Appendix]{bucur2014newisoperimetricinequalityelasticae} we know that a solution to $\kappa'^2 = -\frac{1}{4} \kappa^4 + 2 \kappa + 2C$ on $\mathbb{R}$ is periodic. On each period there is only one local minimum $\kappa_m$, which is actually a global minimum, and similarly the only local maximum $\kappa_M$ is a global maximum. They are the two real roots of the polynomial $P_C(x) = -\frac{1}{4} x^4 + 2x + 2C$. The mapping $C \mapsto \kappa_m(C)$ is decreasing and it holds $\kappa_m(0) = 0$. So if $C < 0$, then $\kappa_m(C) > 0$ and the associated curve $\gamma$ is convex. \cite[Theorem~1.5]{andrews2002} showed that the only convex solution to $-\gamma \cdot \nu = \frac{1}{2} \kappa^2$ are circles. This condition is precisely, up to translation, condition (4) of \Cref{thm:optimality_conditions}. Hence, $C \geq 0$.
\end{proof}

\begin{lemma}
\label{lem:area_constrained_elastica_properties}
    Let $\gamma$ be as in \Cref{lem:bound_C}. Then, $\kappa$ is periodic with number of periods $m \geq 2$. Conversely, for each $m \geq 2$ there is exactly one solution $\kappa$ to \cref{eq:area_constrained_elastica_normalized} with $m$ periods, which corresponds to a closed curve $\gamma$. Furthermore, the following estimates hold
    \begin{equation*}
    \begin{split}
        \mathcal{E}(\gamma) &\geq \sqrt{\frac{11}{6}}m\pi,\\
        \mathcal{A}(\gamma) &= \frac{1}{4} \mathcal{E}(\gamma) \geq \sqrt{\frac{11}{96}}m\pi.
        % \mathcal{L}(\gamma) &\geq \frac{2m(\kappa_M-\kappa_m)}{\sqrt{2^{-\frac{2}{3}}(4m^2-1)^{\frac{2}{3}}-2^{\frac{4}{3}} (4m^2-1)^{-\frac{4}{3}}+2^{-\frac{2}{3}}3}}.
    \end{split}
    \end{equation*}
\end{lemma}

\begin{proof}
    By \Cref{lem:bound_C} we know that $\gamma$ satisfies condition (2) from \Cref{thm:optimality_conditions} with $C \geq 0$. Since $C \mapsto \kappa_m(C)$ is decreasing and $\kappa_m(0) = 0$, it holds $\kappa_m(C) \leq 0$. We know from \cite[Appendix]{bucur2014newisoperimetricinequalityelasticae} that $\kappa$ is periodic. Assume that $\gamma$ is constant speed $1$ and let $T > 0$ be the length of one period. Define
    \begin{equation*}
        I(C) \coloneqq \int_0^T \kappa(x) dx = 2\int_{\kappa_m}^{\kappa_M} \frac{k}{\sqrt{-\frac{1}{4}k^4 + 2k + 2C}}dk,
    \end{equation*}
    where the second equality comes from condition (2) of \Cref{thm:optimality_conditions} and the fact that $\kappa$ is symmetric with respect to its minimizers and maximizers, and monotone between them (cf. \cite[Appendix]{bucur2014newisoperimetricinequalityelasticae}). Moreover, the only values where the denominator is zero are $\kappa_m, \kappa_M$. This is precisely the integral that has been analyzed in \cite[proof of theorem~3.5]{bucur2014newisoperimetricinequalityelasticae}. If we split the integral in two parts from $\kappa_m$ to $0$ and from $0$ to $\kappa_M$, we can compute the derivative explicitly and show that it is strictly negative. A direct computation yields $I(0) = \frac{4\pi}{3}$. Since $\gamma$ is a closed curve with winding number $1$, it has to hold $\frac{4 \pi}{3} \geq I(C) = \frac{2\pi}{m}$ for the number of periods $m$. This can only hold for $m \geq 2$.

    Let us prove that the limit of $I(C)$ as $C \to \infty$ is smaller than or equal to $0$. We want to estimate $I(C)$ from above.
    Let $S = \kappa_M + \kappa_m$, $P=\kappa_M \kappa_m$ and $P_C(x)=2C+2x-\frac{1}{4}x^4$. Since $\kappa_M, \kappa_m$ are the two real roots of $P_C$ (cf. \cite[Appendix]{bucur2014newisoperimetricinequalityelasticae}), we may write $-4 P_C(k) = (k-w)(k-z)(k-\kappa_M)(k-\kappa_m)$. $P_C$ does not have a cubic term, so it holds $w+z+\kappa_M+\kappa_m = 0$.
    This gives $-4P_C(k)=(k^2+Sk+wz)(k^2-Sk+P)$.
    From the quadratic term we obtain $P-S^2+wz=0$ and it follows $-4P_C(k)=\underbrace{(k^2+Sk+S^2-P)}_{Q(k)}(k^2-Sk+P)$. The linear and constant term yield the equations
    \begin{equation*}
    \begin{split}
        SP+(P-S^2)S &= -8,\\
        (S^2-P)P &= -8C.
    \end{split}
    \end{equation*}
    From this it follows 
    \begin{equation}
    \begin{split}
    \label{eq:relation_C_S}
        P &= -\frac{4}{S} + \frac{S^2}{2},\\
        0 &= 8C+\frac{S^4}{4}-\frac{16}{S^2}.
    \end{split}
    \end{equation}
    We conclude $Q(k)=k^2+Sk+\frac{S^2}{2}+\frac{4}{S}$.

    We define $D\coloneqq\frac{\kappa_M-\kappa_m}{2}$ and obtain by a substitution
    \begin{equation*}
    \begin{split}
        I(C) &= 4 \int_{\kappa_m}^{\kappa_M} \frac{k}{\sqrt{(\kappa_M-k)(k-\kappa_m)Q(k)}} dk\\
        &= 4 \int_{-D}^{D} \frac{k+\frac{S}{2}}{\sqrt{(D^2-k^2)Q(k+\frac{S}{2})}}dk\\
        &= 4 \int_{0}^D \frac{k}{\sqrt{D^2-k^2}}\Bigg( \frac{1}{\sqrt{Q(k+\frac{S}{2})}} - \frac{1}{\sqrt{Q(-k+\frac{S}{2})}} \Bigg)dk\\
        &\quad+ 2\int_{-D}^D \frac{S}{\sqrt{(D^2-k^2) Q(k+\frac{S}{2})}} dk.
    \end{split}
    \end{equation*}
    We have $Q(k+\frac{S}{2}) \geq Q(-k+\frac{S}{2})$ for $k \geq 0$, since $S \geq 0$ (due to the coeffiecient of the linear term of $P_C$, which is positive). We estimate
    \begin{equation*}
    \begin{split}
        I(C) &\leq 2S \int_{-D}^D \frac{1}{\sqrt{(D^2-k^2) Q(k+\frac{S}{2})}} dk\\
        &\leq \frac{2S}{\min_{k \in [-D,D]}\sqrt{Q(k+\frac{S}{2})}} \int_{-D}^D \frac{1}{\sqrt{D^2-k^2}} dk.
    \end{split}
    \end{equation*}
    Write $Q(k) = (k+\frac{S}{2})^2+ \frac{S^2}{4} + \frac{4}{S}$. Thus, $\min_{k \in [-D,D]}\sqrt{Q(k+\frac{S}{2})} \geq \sqrt{\frac{S^2}{4} + \frac{4}{S}}$. Together with a substitution this shows
    \begin{equation*}
    \begin{split}
        I(C) \leq \frac{2S}{\sqrt{\frac{S^2}{4}+\frac{4}{S}}} \int_{-1}^1 \frac{1}{\sqrt{1-k^2}} dk = \frac{2S \pi}{\sqrt{\frac{S^2}{4}+\frac{4}{S}}} = \frac{4 \pi S^{\frac{3}{2}}}{\sqrt{S^3+16}}.
    \end{split}
    \end{equation*}
    \Cref{eq:relation_C_S} reveals that $S \to 0$ as $C \to \infty$, proving the claim. One can use this analysis to find an upper bound on $C$ also.

    Conversely, let $m \geq 2$. For $C \in \mathbb{R}$, which will be determined later on, consider a solution to the ordinary differential equation
    \begin{equation}
    \label{eq:ode}
    \begin{split}
        \kappa'' &= -\frac{1}{2} \kappa^3+1,\\
        \kappa(0) &= \kappa_m(C),\\
        \kappa'(0) &= 0,
    \end{split}
    \end{equation}
    which has a unique solution. For any other curve $\tilde{\gamma}$ that solves \cref{eq:area_constrained_elastica_normalized} and has $m$ periods, a shifted version of its curvature $\tilde{\kappa}(\cdot+x_0)$ solves \cref{eq:ode}, where $x_0$ is a point that attains the minimum of the curvature $\tilde{\kappa}$. Therefore, proving uniqueness of $C$ demonstrates that there is at most one area-constrained elastica with $m$ periods. The solution $\kappa$ is periodic with period $T$. We claim that, if we find a value for $C$ such that $I(C) = \frac{2\pi}{m}$, we have found a closed solution to \cref{eq:area_constrained_elastica_normalized}. We have to check the closedness condition. Let $\gamma$ be the associated curve to $\kappa$. Without loss of generality, we assume $\gamma(0) = 0$ and $\theta(0) = 0$. The solution $\kappa$ of \cref{eq:ode} is symmetric with respect to its maximizers and minimizers (cf. \cite[Appendix]{bucur2014newisoperimetricinequalityelasticae}). Let $x_m = \frac{mT}{2}$. Then, $x_m$ is a minimizer or a maximizer, since it is either at the beginning of a period or exactly in the middle. By symmetry, it has to hold $\theta(x_m)=\pi$ and we have $\kappa(x_m+x) = \kappa(x_m-x)$. Integrating gives $\theta(x_m+x) = 2\pi - \theta(x_m-x)$. We deduce for $\gamma = (\gamma_1, \gamma_2)$
    \begin{equation*}
    \begin{split}
        \gamma_2(mT) &= \int_0^{x_m} \sin\theta(x) dx + \int_{x_m}^{mT} \sin\theta(x) dx\\
        &= \int_0^{x_m} \sin\theta(x)+\sin(2\pi - \theta(x))dx = 0.
    \end{split}
    \end{equation*}
    
    Since the curvature is minimal at $0$, by (3) from \Cref{thm:optimality_conditions} we have $0 = \frac{d}{dx}|_{x=0} \abs{\gamma(x)-p}^2 = 2 \tau(0) \cdot (\gamma(0) - p)$. Therefore, $p$ lies on the $y$-axis. Since we know that $x_m$ is a minimizer or maximizer of $\kappa$ and $\theta(x_m) = \pi$, we have by the same argument that $\gamma(x_m)$ lies on the $y$-axis. We obtain
    \begin{equation*}
    \begin{split}
        \gamma_1(mT) &= \int_0^{mT} \cos(\theta(x)) dx = \int_0^{x_m} \cos\theta(x)+\cos(2\pi - \theta(x))dx\\
        &= 2 \gamma_1(x_m) = 0.
    \end{split}
    \end{equation*}
    
    As mentioned before, $I$ is strictly decreasing with $I(0) = \frac{4 \pi}{3}$ and the limit of $I(C)$ is smaller than or equal to $0$ as $C \to \infty$. Consequently, there actually is exactly one solution to $I(C) = \frac{2 \pi}{m}$ with $C \geq 0$. Moreover, by \Cref{lem:bound_C} we deduce that it is the only solution.

    For the energy observe $\mathcal{E}(\gamma) = 2m\int_{0}^{x_M} \kappa^2(x) dx$, where $x_M \in [0,T]$ is the first maximizer of $\kappa$. Since $\kappa$ is increasing on $[0,x_M]$ and $\kappa'(x) = \sqrt{2C + 2\kappa(x) - \frac{1}{4}\kappa(x)^4} = 0$ only for $0$ and $x_M$, it holds
    \begin{equation*}
        \mathcal{E}(\gamma) = 2m \int_{\kappa_m}^{\kappa_M} \frac{k^2}{\sqrt{2C + 2k - \frac{1}{4}k^4}} dk.
    \end{equation*}
    \cite{bucur2014newisoperimetricinequalityelasticae} estimated in the proof of Theorem 3.5
    \begin{equation*}
        \int_{\kappa_m}^{\kappa_M} \frac{k^2}{\sqrt{2C + 2k - \frac{1}{4}k^4}} dk \geq \frac{\pi}{4} \sqrt{\frac{22}{3}}.
    \end{equation*}
    As the area is invariant under translation, condition (4) of \Cref{thm:optimality_conditions} gives
    \begin{equation*}
        \mathcal{A}(\gamma) = \frac{1}{2}\int_{0}^{\mathcal{L}(\gamma)} (p-\gamma(x))\cdot \nu(x) dx = \frac{1}{4} \int_0^{\mathcal{L}(\gamma)} \kappa^2(x)dx = \frac{1}{4} \mathcal{E}(\gamma).
    \end{equation*}

    % Lastly, we need to estimate the length. Since $\kappa$ is periodic and symmetric with respect to its maximizers and minimizers, it is sufficient to consider only one half period. Between a minimizer and maximizer $\kappa$ is strictly increasing with $\kappa' = 0$ only at minimizer and maximizer. We estimate
    % \begin{equation}
    % \begin{split}
    % \label{eq:length_bound_elastica}
    %     \mathcal{L}(\gamma) &= \int_0^{\mathcal{L}(\gamma)} dx = 2m \int_0^T dx = 2m \int_{\kappa_m}^{\kappa_M} \frac{1}{\kappa'(\kappa^{-1}(k))} dk\\
    %     &= 2m \int_{\kappa_m}^{\kappa_M} \frac{1}{\sqrt{2C+2k-\frac{1}{4}k^4}} dk \geq \frac{2(\kappa_M-\kappa_m)}{\max_{k \in [\kappa_m, \kappa_M]}\sqrt{2C+2k-\frac{1}{4}k^4}}\\
    %     &= \frac{2m(\kappa_M-\kappa_m)}{\sqrt{2C+2^{-\frac{2}{3}}3}}.
    % \end{split}
    % \end{equation}
\end{proof}

\begin{lemma}
\label{lem:example_branch1}
    There is a family of curves in $W^{2,2}_{\text{Imm}}(\mathbb{S}^1, \mathbb{R}^2)$ with area $\pi$ and length varying between the values $2\pi$ and $\frac{8}{3} \pi$, whose energy is bounded from above by $\frac{16}{3} \pi$.
\end{lemma}

\begin{proof}
    Let $\gamma \in W^{2,2}_{\text{Imm}}(\mathbb{S}^1, \mathbb{R}^2)$ with $\mathcal{A}(\gamma) = \pi$ of the form illustrated in \Cref{fig:example_branch1}. It consists of a half circle on the left which is connected via straight lines to another half circle.
    Let $r>0$ be the radius of the circles and $b \geq 0$ the length of the straight line segments. Then, it holds
    \begin{equation*}
    \begin{split}
        \mathcal{A}(\gamma) &= \pi r^2 + 2br,\\
        \mathcal{L}(\gamma) &= 2\pi r + 2b,\\
        \mathcal{E}(\gamma) &= \frac{2\pi}{r}.
    \end{split}
    \end{equation*}
    We want $\mathcal{A}(\gamma) = \pi$, which gives $b = \frac{\pi}{2r} - \frac{\pi}{2} r$. Hence, for the length we have $\mathcal{L}(\gamma) = \pi r + \frac{\pi}{r}$. For $r \in (0,1]$ this is strictly decreasing in $r$. To cover all lengths $\ell \in [2\pi, \frac{8}{3}\pi]$, we restrict our family to radii $r \in [\frac{3}{8}, 1]$. At the lower endpoint $r=\frac{3}{8}$, the length is $\mathcal{L}(\gamma) = \frac{73}{24}\pi$, which is strictly greater than $\frac{8}{3}\pi$. By continuity, this family of curves covers all lengths between $2\pi$ and $\frac{8}{3}\pi$. The energy is also decreasing in $r$. Consequently, over this restricted family, the maximum energy is attained at $r=\frac{3}{8}$, and is given by $\mathcal{E}(\gamma) = \frac{16}{3} \pi$.
\end{proof}

\begin{figure}
    \begin{subfigure}[b]{0.48\textwidth}
        \centering
        \begin{tikzpicture}
            \draw (0,0) circle (1);
            \draw[-] (0,0) -- node[left] {$r$} (0,1);
        \end{tikzpicture}
    \end{subfigure}
    \hfill
    \begin{subfigure}[b]{0.48\textwidth}
        \centering
        \begin{tikzpicture}
            \draw (0,1) arc (90:270:1);
            \draw (2,-1) arc (-90:90:1);
            \draw (0,1) -- node[above] {$b$} (2,1);
            \draw (0,-1) -- (2,-1);
            \draw[-] (0,0) -- node[left] {$r$} (0,1);
        \end{tikzpicture}
    \end{subfigure}
    \caption{Family of curves with small length.}
\label{fig:example_branch1}
\end{figure}
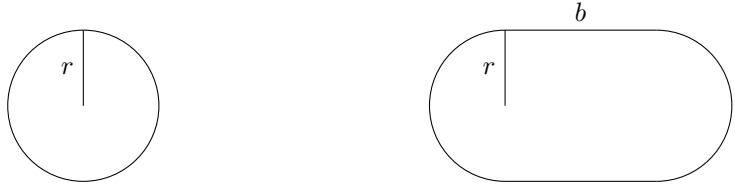

\begin{lemma}
\label{lem:example_branch2}
    There is a family of curves in $W^{2,2}_{\text{Imm}}(\mathbb{S}^1, \mathbb{R}^2)$ with area $\pi$ and length varying between the values $4 \pi \sqrt{\frac{\pi}{\pi+4}} \approx 2.653 \pi$ and $\infty$, whose energy is bounded from above by $4\pi\sqrt{\frac{\pi+4}{\pi}} \approx 6.031 \pi$.
\end{lemma}

\begin{proof}
    Let $\gamma \in W^{2,2}_{\text{Imm}}(\mathbb{S}^1, \mathbb{R}^2)$ with $\mathcal{A}(\gamma) = \pi$ of the form illustrated in \Cref{fig:example_branch2}. It consists of two three-quarter circles (solid), that are connected via straight line segments (dotted) to the missing quarters of the circles (dashed). We may vary its length by altering the length of the straight line segments up until the point where the two middle quarters touch each other tangentially. From here on, we can increase the length of the curve indefinitely by splitting the curve at the touching point and inserting two perfectly overlapping straight line segments. Because these segments perfectly overlap, they enclose zero area, and because they are straight, they contribute zero to the elastic energy. Hence, the length grows to infinity while the area and energy remain constant.
    Let $r \in (0,1]$ be the radius of the circles. The energy of the curve is given by $\mathcal{E}(\gamma) = \frac{4\pi}{r}$. Let $\sqrt{2}b$ be the length of the straight line segments for $b\geq0$ and compute the enclosed area by $\mathcal{A}(\gamma) = r^2(\pi+4) - 2b^2$. We want $\mathcal{A}(\gamma) = \pi$, which leads to $r = \sqrt{\frac{\pi+2b^2}{\pi+4}}$. So, $r$ is increasing in $b$. As the energy is decreasing in $r$, its maximal value is obtained for $b=0$. Inserting shows $\mathcal{E}(\gamma) \leq 4\pi\sqrt{\frac{\pi+4}{\pi}}$. The length is given by $\mathcal{L}(\gamma) = 4 \pi r + 4\sqrt{2}b$. The minimum is achieved for $b=0$, yielding the result.
\end{proof}

\begin{figure}
    \begin{subfigure}[b]{0.48\textwidth}
        \centering
        \begin{tikzpicture}
            \draw (0,0) arc (45:315:1);
            \draw[dashed] (0,0) arc (225:315:1);
            \draw ({sqrt(2)},0) arc (135:-135:1);
            \draw[dashed] ({sqrt(2)},{-sqrt(2)}) arc (45:135:1);
            \draw ({1/sqrt(2)},{1/sqrt(2)}) -- node[left] {$r$} ({1/sqrt(2)}, {1/sqrt(2)-1});
            \draw[-] ({-1/sqrt(2)},{-1/sqrt(2)}) -- node[below right] {$r$} (0,0);
        \end{tikzpicture}
    \end{subfigure}
    \hfill
    \begin{subfigure}[b]{0.48\textwidth}
        \centering
        \begin{tikzpicture}
            \draw (0,0) arc (45:315:1);
            \draw[dashed] ({sqrt(2)-1},{1-sqrt(2)}) arc (225:315:1);
            \draw ({3*sqrt(2)-2},0) arc (135:-135:1);
            \draw[dashed] ({sqrt(2)-1},-1) arc (135:45:1);
            \draw[dotted] (0,0) -- node[above right] {$\sqrt{2}b$} ({sqrt(2)-1},{1-sqrt(2)});
            \draw[dotted] ({2*sqrt(2)-1},{1-sqrt(2)}) -- ({3*sqrt(2)-2},0);
            \draw[dotted] (0,{-sqrt(2)}) -- ({sqrt(2)-1},-1);
            \draw[dotted] ({2*sqrt(2)-1},-1) -- ({3*sqrt(2)-2},{-sqrt(2)});
        \end{tikzpicture}
    \end{subfigure}
    \caption{Family of curves with increasing length.}
\label{fig:example_branch2}
\end{figure}
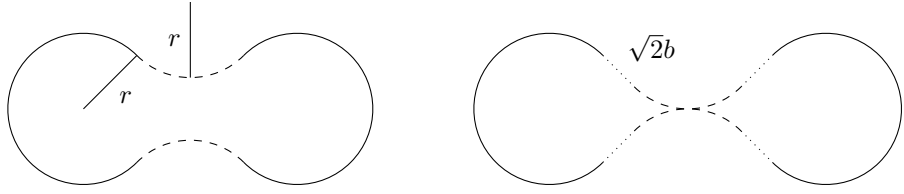

\begin{theorem}
\label{thm:number_critical_points}
    Let $\gamma \in C^\infty(\mathbb{S}^1, \mathbb{R}^2)$ be a solution to \cref{eq:area_constrained_elastica_normalized} parametrized by arc-length with the number of periods $m \geq 4$. Then, for the rescaled version $\tilde{\gamma}$ with $\mathcal{A}(\tilde{\gamma}) = \pi$, we have $\mathcal{E}(\tilde{\gamma}) > E(\ell^*)$. Hence, $E: [2\pi, \infty) \to [0,\infty)$ has at most $3$ critical points, one of which is the circle.
\end{theorem}

\begin{proof}
    From \Cref{lem:area_constrained_elastica_properties} we obtain
    \begin{equation*}
    \begin{split}
        \mathcal{E}(\gamma) &\geq \sqrt{\frac{11}{6}}m\pi,\\
        \mathcal{A}(\gamma) &= \frac{1}{4} \mathcal{E}(\gamma) \geq \sqrt{\frac{11}{96}}m\pi.
    \end{split}
    \end{equation*}
    Hence, for the rescaled version $\tilde{\gamma}$ with area $\pi$ it holds
    \begin{equation*}
        \mathcal{E}(\tilde{\gamma}) = \sqrt{\frac{\mathcal{A}(\gamma)}{\pi}} \mathcal{E}(\gamma) \geq \Big(\frac{11}{96}\Big)^\frac{1}{4}\Big(\frac{11}{6}\Big)^{\frac{1}{2}}m^{\frac{3}{2}}\pi.
    \end{equation*}
    The expression is increasing in $m$ and for $m=4$ it evaluates to $\mathcal{E}(\tilde{\gamma}) \geq 6.302 \pi$. By \Cref{lem:example_branch1} and \Cref{lem:example_branch2}, the maximum of $E: [2\pi, \infty) \to [0, \infty)$ is strictly below $6.302 \pi$. Therefore, $\mathcal{E}(\tilde{\gamma}) > E(\ell^*)$, meaning $\tilde{\gamma}$ does not lie on the profile curve.

    Every critical point of the energy profile curve is attained by an area-constrained elastica by \Cref{thm:critical_point_area_constrained_elastica}. By \Cref{lem:area_constrained_elastica_properties} every number $m \in \mathbb{N}$ corresponds to exactly one area-constrained elastica (the circle corresponding to $m=1$) and vice versa. By the previous calculation, the only possible critical points can be attained by the area-constrained elasticae with $m=2,3$ or by the circle.
\end{proof}

This result allows us to prove that the length diverges under specific conditions on the length and energy.

\begin{proof}[Proof of \Cref{cor:length_divergence}]
    Assume there is an increasing sequence $0 \leq t_i \to \infty$ as $i \to \infty$ such that $\mathcal{L}(\gamma_{t_i}) \leq c$ for some constant $c > 0$. We apply the same method of proof as the one of \Cref{thm:subconvergence} and define the translated and arc-length parametrized version $\tilde{\gamma}$ of $\gamma$. However, this time, we obtain compactness only for the sequence $(\tilde{\gamma}_{t_i})_{i \in \mathbb{N}}$, since only for these times the length is bounded. Then, up to a subsequence, $\tilde{\gamma}_{t_i} \to \gamma_\infty \in C^\infty(\mathbb{S}^1, \mathbb{R}^2)$ smoothly.

    We claim that $\gamma_\infty$ is an area-constrained elastica. As in the proof of \Cref{thm:subconvergence}, define $u(t) \coloneqq \norm{\partial_t \gamma_t}_{L^2(ds)}^2$. By the energy identity $u(t) = - \partial_t \mathcal{E}(\gamma)$ we conclude $u \in L^1((0,\infty))$. This is not sufficient to conclude $u(t_i) \to 0$ as $i \to \infty$ up to a subsequence. $L^1((0,\infty))$ only gives us a sequence $\tilde{t}_i$ such that $u(\tilde{t}_i) \to 0$. However, $\tilde{t}_i$ could be a completely different sequence of times than $t_i$. Consequently, we require bounds on the derivative of $u$. Define $\xi \coloneqq \partial_t \gamma \cdot \nu$. Compute with \Cref{lem:differential_properties}
    \begin{equation*}
    \begin{split}
        \partial_t u(t) &= 2\int_{\mathbb{S}^1} \partial_t^2 \gamma \cdot \partial_t \gamma ds - \int_{\mathbb{S}^1} \abs{\partial_t \gamma}^2 \xi \kappa ds,\\
        \partial_t^2 \gamma &= -2 \partial_t(\partial_s^2 \kappa) \nu - 3 \partial_t \kappa \kappa^2 \nu - \partial_t \lambda \nu + \xi \partial_t \nu,\\
        \partial_t \lambda &= - \frac{\int_{\mathbb{S}^1} 3 \partial_t \kappa \kappa^2 ds}{\mathcal{L}(\gamma)} + \frac{\int_{\mathbb{S}^1} \kappa^4 \xi ds}{\mathcal{L}(\gamma)} + \frac{\partial_t \mathcal{L}(\gamma) \int_{\mathbb{S}^1} \kappa^3 ds}{\mathcal{L}^2(\gamma)},\\
        \partial_t \mathcal{L}(\gamma) &= -\int_{\mathbb{S}^1} \kappa \xi ds,\\
        \partial_t \kappa &= \partial_s^2 \xi + \kappa^2 \xi,\\
        \partial_t \nu &= -\partial_s \xi \tau.
    \end{split}
    \end{equation*}
    First, consider the integral $\int_{\mathbb{S}^1} \partial_t^2 \gamma \cdot \partial_t \gamma ds$. Let us compute the first summand $-2\partial_t(\partial_s^2\kappa)\nu$ in $\partial_t^2 \gamma_t$ with \Cref{lem:differential_properties}
    \begin{equation}
    \label{eq:second_time_derivative_gamma1}
    \begin{split}
        \partial_t (\partial_s^2 \kappa) &= \partial_s \partial_t \partial_s \kappa + \kappa \xi \partial_s^2 \kappa\\
        &= \partial_s^2 \partial_t \kappa + \partial_s(\kappa \xi \partial_s \kappa) + \kappa \xi \partial_s^2 \kappa\\
        &= \partial_s^2(\partial_s^2 \xi + \kappa^2 \xi)  + \xi (\partial_s \kappa)^2+\kappa \partial_s \xi \partial_s \kappa + 2\kappa \xi \partial_s^2 \kappa.
    \end{split}
    \end{equation}
    With the notation from \Cref{sec:global_existence} and $\xi = -2\partial_s^2 \kappa - \kappa^3 - \lambda$ we write
    \begin{equation*}
        -2 \partial_t(\partial_s^2\kappa) = P_1^6(\kappa)+P_3^4(\kappa)+P_5^2(\kappa) + \lambda P_2^2(\kappa).
    \end{equation*}
    For the second summand $-3\partial_t \kappa \kappa^2 \nu$ we write
    \begin{equation*}
    \begin{split}
        -3 \partial_t \kappa \kappa^2 &= -3(\partial_s^2 \xi + \kappa^2 \xi) \kappa^2\\
        &= P_3^4(\kappa) + P_5^2(\kappa) + P_7^0(\kappa) + \lambda P_4^0(\kappa).
    \end{split}
    \end{equation*}
    We leave the third summand $-\partial_t \lambda \nu$ as it is for now. The last term $\xi \partial_t \nu$ in $\partial_t^2 \gamma_t$ can be neglected, since it has only a tangential component and $\partial_t \gamma$ has only a normal component.
    Similarly, we write
    \begin{equation*}
        \partial_t \gamma = (P_1^2(\kappa)+P_3^0(\kappa)-\lambda)\nu.
    \end{equation*}
    The product $\partial_t^2 \gamma \cdot \partial_t \gamma$ is a sum of terms of the form $\lambda^k (\partial_t \lambda)^l P_\beta^\alpha(\kappa)$ with $k,l,\beta,\alpha \in \mathbb{N}$. We either have $k+l > 0$ or $\beta \geq 2$. If $\beta \geq 2$, we estimate with \Cref{lem:infinity_bound_kappa}
    \begin{equation}
    \begin{split}
    \label{eq:estimates_lambda_dt_lambda1}
        \abs{\lambda} &\leq \frac{\int_{\mathbb{S}^1} \abs{\kappa}^3 ds}{\mathcal{L}(\gamma)} \leq c(\gamma_0) \frac{\int_{\mathbb{S}^1}ds}{\mathcal{L}(\gamma)} = c(\gamma_0),\\
        \abs{\partial_t \lambda} &\leq 3\frac{\int_{\mathbb{S}^1} \abs{\partial_t \kappa \kappa^2} ds}{\mathcal{L}(\gamma)} + \frac{\int_{\mathbb{S}^1} \abs{\kappa^4 \xi} ds}{\mathcal{L}(\gamma)} + \frac{\int_{\mathbb{S}^1} \abs{\kappa \xi} ds \int_{\mathbb{S}^1}\abs{\kappa^3}ds}{\mathcal{L}^2(\gamma)}\\
        &\leq c(\gamma_0)\Big(\frac{ \int_{\mathbb{S}^1}ds}{\mathcal{L}(\gamma)} + \frac{\int_{\mathbb{S}^1}ds \int_{\mathbb{S}^1}ds}{\mathcal{L}^2(\gamma)}\Big) = c(\gamma_0).
    \end{split}
    \end{equation}
    Then \Cref{prop:inequality_P} and \cref{eq:2_norm_kappa} show
    \begin{equation*}
        \int_{\mathbb{S}^1} \abs{\lambda^k (\partial_t \lambda)^l P_\beta^\alpha(\kappa)} ds \leq c(\gamma_0).
    \end{equation*}
    If $k+l > 0$, we estimate with \Cref{lem:infinity_bound_kappa}
    \begin{equation*}
        \abs{P_\beta^\alpha(\kappa)} \leq c(\gamma_0).
    \end{equation*}
    Moreover, we can bound with \Cref{lem:infinity_bound_kappa}
    \begin{equation*}
    \begin{split}
    \label{eq:estimates_lambda_dt_lambda2}
        \abs{\lambda} &\leq \frac{\int_{\mathbb{S}^1} \abs{\kappa}^3 ds}{\mathcal{L}(\gamma)} \leq c(\gamma_0) \frac{\int_{\mathbb{S}^1}\kappa^2 ds}{\mathcal{L}(\gamma)} = \frac{c(\gamma_0)}{\mathcal{L}(\gamma)},\\
        \abs{\partial_t \lambda} &\leq 3\frac{\int_{\mathbb{S}^1} \abs{\partial_t \kappa \kappa^2} ds}{\mathcal{L}(\gamma)} + \frac{\int_{\mathbb{S}^1} \abs{\kappa^4 \xi} ds}{\mathcal{L}(\gamma)} + \frac{\int_{\mathbb{S}^1} \abs{\kappa \xi} ds \int_{\mathbb{S}^1}\abs{\kappa^3}ds}{\mathcal{L}^2(\gamma)}\\
        &\leq c(\gamma_0)\Big(\frac{\int_{\mathbb{S}^1}\kappa^2 ds}{\mathcal{L}(\gamma)} + \frac{\int_{\mathbb{S}^1}ds \int_{\mathbb{S}^1}\kappa^2ds}{\mathcal{L}^2(\gamma)}\Big) = \frac{c(\gamma_0)}{\mathcal{L}(\gamma)}.
    \end{split}
    \end{equation*}
    This shows together with \Cref{eq:estimates_lambda_dt_lambda1}
    \begin{equation*}
        \int_{\mathbb{S}^1} \abs{\lambda^k (\partial_t \lambda)^l P_\beta^\alpha(\kappa)} ds \leq \frac{c(\gamma_0) \int_{\mathbb{S}^1}ds}{\mathcal{L}(\gamma)} = c(\gamma_0).
    \end{equation*}
    In total, we have found a bound $\abs{\int_{\mathbb{S}^1} \partial_t^2 \gamma \cdot \partial_t \gamma ds} \leq c(\gamma_0)$. Using the same strategy we can also bound $\abs{\int_{\mathbb{S}^1} \abs{\partial_t \gamma}^2 \xi \kappa ds} \leq c(\gamma_0)$. Consequently, $\partial_t u(t)$ is bounded. Since it is in $L^1((0,\infty))$, $u(t_i) \to 0$ for $i \to \infty$. Now, we can proceed as in the proof of \Cref{thm:subconvergence} and obtain that $\gamma_\infty$ is an area-constrained elastica.

    However, this area-constrained elastica cannot have more than $3$ periods, since the energy of such an elastica is larger than $E(\ell^*)$ by \Cref{thm:number_critical_points}. Also, it cannot be an area-constrained elastica with $2$ or $3$ periods because $\mathcal{E}(\gamma_2), \mathcal{E}(\gamma_3) > \mathcal{E}(\gamma_0)$. If it were the circle, there is a time $t^*> 0$ such that $\mathcal{L}(\gamma_{t^*}) = \ell^*$. Then, $\mathcal{E}(\gamma_{t^*}) \geq E(\ell^*) > \mathcal{E}(\gamma_0)$, which is a contradiction. Therefore, $\mathcal{L}(\gamma) \to \infty$ as $t \to \infty$.
\end{proof}

\section{Numerical simulations}
\label{sec:numerical_simulations}
In this chapter, we complement the analysis of the area-constrained elastic flow with numerical simulations, adapting the discretization schemes proposed by \cite{Dziuk2002}. Our goal is to investigate the behavior of the gradient flow under various initial configurations, offering insights into potential regimes of convergence.

First, recall that by the Frenet equations, we have
\begin{equation*}
\begin{split}
    \partial_s(2\partial_s(\kappa \nu)+3\kappa^2 \tau) &= \partial_s(2\partial_s \kappa \nu + \kappa^2\tau)\\
    &= 2 \partial_s^2\kappa \nu - 2\partial_s \kappa \kappa \tau + 2 \partial_s \kappa \kappa \tau + \kappa^3\nu\\
    &= (2 \partial_s^2 \kappa + \kappa^3)\nu.
\end{split}
\end{equation*}
For that reason, we write \cref{eq:evolution_gradient_flow} as
\begin{equation*}
    \partial_t \gamma + 2 \partial_s^2 \vec{\kappa} + 3 \partial_s(\kappa^2 \partial_s \gamma) + \lambda \nu = 0,
\end{equation*}
where we denote $\vec{\kappa} \coloneqq \kappa \nu$.
By definition, the following hold true
\begin{equation*}
\begin{split}
    0 &= \nu - Q \partial_s \gamma,\\
    0 &= \vec{\kappa}-\partial_s^2 \gamma,
\end{split}
\end{equation*}
where the rotation $Q$ is given by $Q = \begin{pmatrix}
    0 & -1\\
    1 & 0
\end{pmatrix}$.
The variational formulation of \cref{eq:evolution_gradient_flow} is
\begin{equation}
\begin{split}
\label{eq:a_numeric_variational_formulation}
    \int_{\mathbb{S}^1} \partial_t \gamma \cdot \varphi \abs{\partial_x \gamma} dx - 2 \int_{\mathbb{S}^1} \frac{\partial_x \vec{\kappa}}{\abs{\partial_x \gamma}} \cdot \partial_x \varphi dx - 3 \int_{\mathbb{S}^1} \abs{\vec{\kappa}}^2 \frac{\partial_x \gamma}{\abs{\partial_x \gamma}} \cdot \partial_x \varphi dx\\
    + \lambda \int_{\mathbb{S}^1} \nu \cdot \varphi \abs{\partial_x \gamma}dx &=0,\\
    \int_{\mathbb{S}^1} \nu \cdot \theta \abs{\partial_x \gamma} dx - \int_{\mathbb{S}^1} Q \partial_x \gamma \cdot \theta dx &= 0,\\
    \int_{\mathbb{S}^1} \vec{\kappa} \cdot \psi \abs{\partial_x \gamma} dx + \int_{\mathbb{S}^1} \frac{\partial_x \gamma}{\abs{\partial_x \gamma}} \cdot \partial_x \psi dx &= 0,
\end{split}
\end{equation}
for all test functions $\varphi, \theta, \psi \in W^{1,2}(\mathbb{S}^1, \mathbb{R}^2)$.

For the purpose of the simulations, we want to use a finite element discretization in space. Consider the parametrization $I = [0,2 \pi]$ of $\mathbb{S}^1$. For a decomposition $I = \bigcup_{j=1}^N I_j$ given by nodes $0 = x_0 < x_1 < \ldots < x_N = 2 \pi$ with $I_j = [x_{j-1}, x_j]$ define $h_j = \abs{I_j}$ and the diameter of the maximal grid element $h = \max_{j \in [1,\ldots, N]} h_j$. We use the discretization $Y_h \coloneqq (X_h)^2 \subset W^{1,2}(I, \mathbb{R}^2)$ as the space of test functions, where
\begin{equation*}
    X_h \coloneqq \{w \in C^0(I,\mathbb{R}) \mid \forall j \in [1,\ldots, N]\quad w|_{I_j} \in \mathbb{P}_1(I_j), w(x_0)=w(x_N)\}.
\end{equation*}
We denote by $\mathbb{P}_1$ polynomials of degree $1$. Hence, $X_h$ is the space of scalar continuous piecewise affine functions. The scalar nodal basis functions $\varphi_j \in X_h$, which are uniquely defined by $\varphi_j(x_i) = \delta_{ij}$, form a basis of $X_h$. Therefore, it is sufficient to test with $\varphi_j e_1$ and $\varphi_j e_2$ for all $j \in [1,\ldots,N]$, where $\{e_1,e_2\}$ denotes the standard basis of $\mathbb{R}^2$.

For a function $w \in C^0(I, \mathbb{R})$ with $w(0)=w(2\pi)$ consider the pointwise interpolation $I_h(w) \in X_h$, which is uniquely defined by $I_h(w)(x_j) = w(x_j)$. We may write
\begin{equation*}
    I_h(w)(x) = \sum_{j=1}^N w(x_j) \varphi_j(x).
\end{equation*}
A numerical solution to \cref{eq:a_numeric_variational_formulation} will be a triplet $\gamma_h: [0,T] \to Y_h, \nu_h:[0,T] \to Y_h, \vec{\kappa}_h:[0,T] \to Y_h$ given by
\begin{equation*}
\begin{split}
    \gamma_h(t,x) &= \sum_{j=1}^N \gamma_j(t) \varphi_j(x),\\
    \nu_h(t,x) &= \sum_{j=1}^N \nu_j(t)\varphi_j(x),\\
    \vec{\kappa}_h(t,x) &= \sum_{j=1}^N \vec{\kappa}_j(t)\varphi_j(x).
\end{split}
\end{equation*}
Observe that the normal as well as the curvature cannot be computed at the nodes of functions in $Y_h$ as they are essentially a rotation of the first derivative and the second derivative of the curve $\gamma_h$. For that reason we need the variational formulation of the normal and curvature equations.

The problem is the following: search $\gamma_h(t,\cdot), \nu_h(t,\cdot), \vec{\kappa}_h(t,\cdot) \in Y_h$ such that $\gamma_h(0,\cdot)=I_h(\gamma_0)$ and for all discrete test functions $\varphi_h, \theta_h, \psi_h \in Y_h$ the equations
\begin{equation}
\begin{split}
\label{eq:a_numeric_variational_formulation_discrete}
    \int_{I} I_h(\partial_t \gamma_h \cdot \varphi_h) \abs{\partial_x \gamma_h} dx - 2 \int_{I} \frac{\partial_x \vec{\kappa}_h}{\abs{\partial_x \gamma_h}} \cdot \partial_x \varphi_h dx\\
    - 3 \int_{I} \abs{\vec{\kappa}_h}^2 \frac{\partial_x \gamma_h}{\abs{\partial_x \gamma_h}}\cdot \partial_x \varphi_h dx + \lambda \int_{I} I_h(\nu_h \cdot \varphi_h) \abs{\partial_x \gamma_h}dx &=0,\\
    \int_{I} I_h(\nu_h \cdot \theta_h) \abs{\partial_x \gamma_h} dx - \int_{I} Q \partial_x \gamma_h \cdot \theta_h dx &= 0,\\
    \int_{I} I_h(\vec{\kappa}_h \cdot \psi_h) \abs{\partial_x \gamma_h} dx + \int_{I} \frac{\partial_x \gamma_h}{\abs{\partial_x \gamma_h}} \cdot \partial_x \psi_h dx &= 0
\end{split}
\end{equation}
are satisfied. We apply the interpolation $I_h$ in the equations, which is known as mass lumping, to accelerate and stabilize the computations. As mentioned before, testing against $\varphi_j$ in both coordinates is enough. Inserting and integrating gives a system of $2N$ ODEs
\begin{equation*}
\begin{split}
    \frac{1}{2}(\partial_t \gamma_j+\lambda \nu_j)(\abs{\gamma_j-\gamma_{j-1}}+\abs{\gamma_{j+1}-\gamma_j})\\
    -2\Big(\frac{\vec{\kappa}_j-\vec{\kappa}_{j-1}}{\abs{\gamma_j-\gamma_{j-1}}}-\frac{\vec{\kappa}_{j+1}-\vec{\kappa}_j}{\abs{\gamma_{j+1}-\gamma_j}}\Big)
    -(\abs{\vec{\kappa}_{j-1}}^2+\vec{\kappa}_{j-1}\cdot\vec{\kappa}_j+\abs{\vec{\kappa}_j}^2)\frac{\gamma_j-\gamma_{j-1}}{\abs{\gamma_j-\gamma_{j-1}}}\\
    +(\abs{\vec{\kappa}_j}^2+\vec{\kappa}_j\cdot\vec{\kappa}_{j+1}+\abs{\vec{\kappa}_{j+1}}^2) \frac{\gamma_{j+1}-\gamma_j}{\abs{\gamma_{j+1}-\gamma_j}} =& 0,\\
    \frac{1}{2}\nu_j(\abs{\gamma_j-\gamma_{j-1}}+\abs{\gamma_{j+1}-\gamma_j})-\frac{1}{2}Q(\gamma_{j+1}-\gamma_{j-1}) =& 0,\\
    \frac{1}{2}\vec{\kappa}_j(\abs{\gamma_j-\gamma_{j-1}}+\abs{\gamma_{j+1}-\gamma_j})+\frac{\gamma_j-\gamma_{j-1}}{\abs{\gamma_j-\gamma_{j-1}}}-\frac{\gamma_{j+1}-\gamma_j}{\abs{\gamma_{j+1}-\gamma_j}} =& 0,
\end{split}
\end{equation*}
where $\gamma_0=\gamma_N, \gamma_{N+1}=\gamma_1, \nu_0=\nu_N, \nu_{N+1}=\nu_1, \vec{\kappa}_0=\vec{\kappa}_N, \vec{\kappa}_{N+1}=\vec{\kappa}_1$. The initial values are given by $\gamma_j(0)=\gamma_0(x_j)$.

Moreover, in order to compute $\lambda$, we test with the normal $\nu_h$. The area constraint leads to
\begin{equation*}
    \int_{I}I_h(\partial_t \gamma_h \cdot \nu_h)\abs{\partial_x \gamma_h}dx = 0.
\end{equation*}

Then \cref{eq:a_numeric_variational_formulation_discrete} shows
\begin{equation*}
\begin{split}
    - 2 \int_{I} \frac{\partial_x \vec{\kappa}_h}{\abs{\partial_x \gamma_h}} \cdot \partial_x \nu_h dx - 3 \int_{I} \abs{\vec{\kappa}_h}^2 \frac{\partial_x \gamma_h}{\abs{\partial_x \gamma_h}} \cdot \partial_x \nu_h dx\\
    + \lambda \int_{I} I_h(\abs{\nu_h}^2) \abs{\partial_x \gamma_h}dx =0.
\end{split}
\end{equation*}
Rearranging and computing the integrals gives
\begin{equation*}
\begin{split}
    \lambda &= \Big(\sum_{j=1}^N \abs{\gamma_j-\gamma_{j-1}}(\abs{\nu_j}^2+\abs{\nu_{j-1}}^2)\Big)^{-1}\\
    &\quad\Big(4\sum_{j=1}^N \frac{\vec{\kappa}_j-\vec{\kappa}_{j-1}}{\abs{\gamma_j-\gamma_{j-1}}}\cdot(\nu_j-\nu_{j-1})\\
    &\quad+2\sum_{j=1}^N\frac{\gamma_j-\gamma_{j-1}}{\abs{\gamma_j-\gamma_{j-1}}}\cdot(\nu_j-\nu_{j-1})(\abs{\vec{\kappa}_{j-1}}^2+\vec{\kappa}_{j-1}\cdot\vec{\kappa}_j+\abs{\vec{\kappa}_j}^2)\Big).
\end{split}
\end{equation*}
 For the algorithm we will use the notation $w^m=w(m \Delta t, \cdot)$ where $\Delta t$ is the time step and $T=M \Delta t$ is the total simulation time.

\subsection{Algorithm}
For $j=1,\ldots,N$:
\begin{equation*}
\begin{split}
    \gamma_j^0 &= \gamma_0(x_j)\\
    h_j^0 &= \abs{\gamma_j^0-\gamma_{j-1}^0}\\
    \nu_j^0 &= \frac{Q(\gamma_{j+1}^{0}-\gamma_{j-1}^{0})}{h_j^0+h_{j+1}^0}\\
    \vec{\kappa}_j^0 &= \frac{2}{h_{j+1}^0(h_j^0+h_{j+1}^0)} \gamma_{j+1}^0 - \frac{2}{h_j^0h_{j+1}^0} \gamma_j^0 + \frac{2}{h_j^0(h_j^0+h_{j+1}^0)}\gamma_{j-1}^0.
\end{split}
\end{equation*}
For $m=0,\ldots,M-1$:
\begin{equation*}
\begin{split}
    h_j^m &= \abs{\gamma_j^m-\gamma_{j-1}^m}\\
    \beta_j^m &= \abs{\vec{\kappa}_{j-1}^m}^2+\vec{\kappa}_{j-1}^m\cdot\vec{\kappa}_j^m+ \abs{\vec{\kappa}_j^m}^2\\
    \lambda^m &= \Big(\sum_{j=1}^N h_j^m(\abs{\nu^m_j}^2+\abs{\nu^m_{j-1}}^2)\Big)^{-1}\\
    &\quad\Big(4\sum_{j=1}^N\frac{\vec{\kappa}^m_j-\vec{\kappa}^m_{j-1}}{h_j^m}\cdot(\nu^m_j-\nu^m_{j-1})\\
    &\quad+2\sum_{j=1}^N\frac{\gamma^m_j-\gamma^m_{j-1}}{h_j^m}\cdot(\nu^m_j-\nu^m_{j-1})\beta_j^m\Big).
\end{split}
\end{equation*}
Determine $\gamma_j^{m+1}, \nu_j^{m+1}, \vec{\kappa}_j^{m+1} \in \mathbb{R}^2$ from the linear system
\begin{equation*}
\begin{split}
    &\quad\frac{\beta_j^m}{h_j^m} \gamma_{j-1}^{m+1} + \Big(\frac{h_j^m+h_{j+1}^m}{2 \Delta t}- \frac{\beta_j^m}{h_j^m}- \frac{\beta_{j+1}^m}{h_{j+1}^m}\Big)\gamma_j^{m+1} +\frac{\beta_{j+1}^m}{h_{j+1}^m}\gamma_{j+1}^{m+1}\\
    &+ \frac{2}{h_j^m}\vec{\kappa}_{j-1}^{m+1} - \Big(\frac{2}{h_j^m}+\frac{2}{h_{j+1}^m}\Big)\vec{\kappa}_j^{m+1}+\frac{2}{h_{j+1}^{m}} \vec{\kappa}_{j+1}^{m+1}\\
    &+ \frac{\lambda^m}{2} (h_j^m+h_{j+1}^m) \nu_j^{m+1} = \frac{h_j^m+h_{j+1}^m}{2\Delta t} \gamma_j^m,\\
    &-Q (\gamma_{j+1}^{m+1}-\gamma_{j-1}^{m+1}) + (h_j^m+h_{j+1}^m)\nu_j^{m+1} = 0,\\
    &-\frac{1}{h_j^m} \gamma_{j-1}^{m+1} + (\frac{1}{h_j^m}+\frac{1}{h_{j+1}^m})\gamma_j^{m+1} - \frac{1}{h_{j+1}^m} \gamma_{j+1}^{m+1}\\
    &+\frac{1}{2}(h_j^m+h_{j+1}^m)\vec{\kappa}_j^{m+1} = 0.
\end{split}
\end{equation*}

For the sake of numeric stability, we rearranged the points $\gamma_j$ every $k \in \mathbb{N}$ steps such that they are evenly distributed along the curve.

\clearpage

\subsection{Simulation results}
In the following calculations, we used $N=100$ for the first two simulations and $N=300$ for the last two for stability. We parametrized $\mathbb{S}^1$ with $I=[0,2 \pi]$ and distributed nodes $x_j = 2\pi\frac{j-1}{N}$ evenly, omitting the redundant closure point at $2\pi$. As in \cite{Dziuk2002}, we choose $\Delta t=0.5 h^2$ for the maximal grid element $h$. As mentioned before, we rearrange the points $\gamma_j$ every $k\coloneqq100$ steps.\\
We are particularly interested in the four following cases.

\textbf{Convergence to the circle}\\
The circle is an area-constrained elastica, satisfying the equation
\begin{equation*}
    -2\partial_s^2\kappa-\kappa^3 - \lambda  = 0
\end{equation*}
for some $\lambda \in \mathbb{R}$. We expect that if the gradient flow starts from a curve $\gamma_0$ sufficiently close to the circle, it will converge to the circle. For example, consider the initial curve $\gamma_0(x) = (\cos(x), \frac{1}{2}\sin(x))$, an oval shape. The gradient flow starting from this curve is illustrated in \Cref{fig:oval}.

\begin{figure}[!htbp]
    \centering
    \begin{subfigure}[b]{0.3\textwidth}
        \centering
        \includegraphics[width=\textwidth]{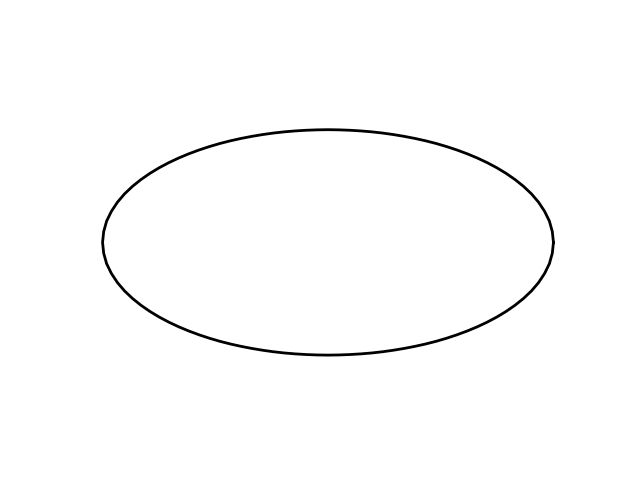}
        \caption{$t=0$} % Iteration 0
    \label{fig:oval1}
    \end{subfigure}
    \hfill
    \begin{subfigure}[b]{0.3\textwidth}
        \centering
        \includegraphics[width=\textwidth]{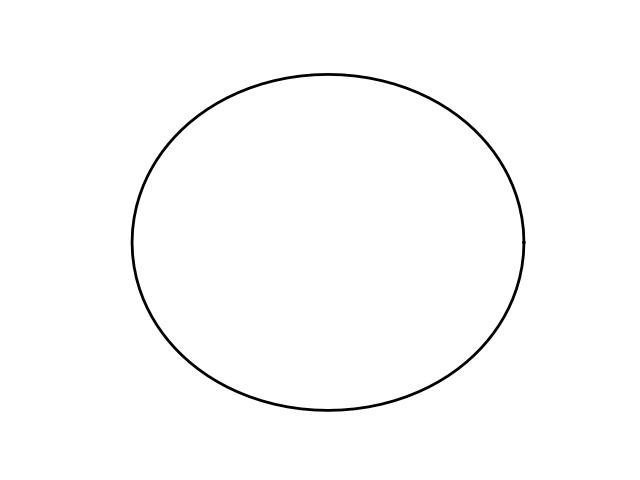}
        \caption{$t=0.04$} % Iteration 20
    \label{fig:oval2}
    \end{subfigure}
    \hfill
    \begin{subfigure}[b]{0.3\textwidth}
        \centering
        \includegraphics[width=\textwidth]{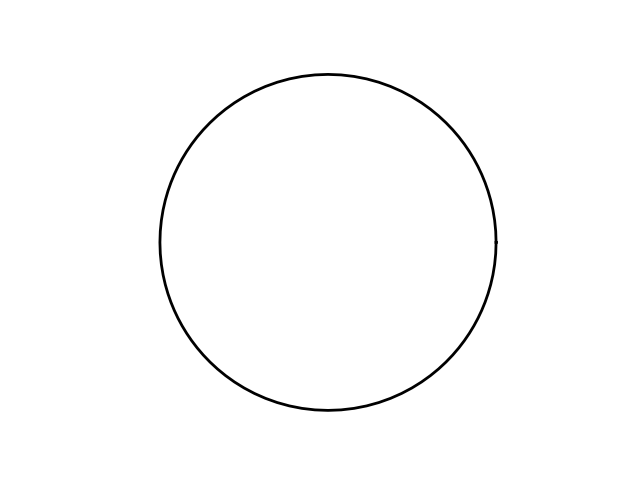}
        \caption{$t=0.16$} % Iteration 80
    \label{fig:oval3}
    \end{subfigure}
    \caption{Oval shape $\gamma_0(x) = (\cos(x), \frac{1}{2}\sin(x))$.}
\label{fig:oval}
\end{figure}

\textbf{Unbounded growth due to scaling}\\
The elastic energy $\mathcal{E}$ exhibits scaling properties: $\mathcal{E}(r\gamma) = \frac{1}{r}\mathcal{E}(\gamma)$. This implies scaling a curve by $r>1$ reduces its elastic energy. For the unconstrained elastic flow, we therefore expect the curve to grow indefinitely. However, under an area constraint, this growth is limited---unless the algebraic area $\mathcal{A}$ remains unchanged. Consider the lemniscate of Bernoulli
\begin{equation*}
    \gamma_0(x) = \frac{1}{1+\sin^2 (x)}(\cos(x), \frac{1}{2}\sin(2x)).
\end{equation*}
As mentioned in \cite{Miura_2025}, the lemniscate expands self-similarly for the free elastic flow with length in $\mathcal{O}(t^{\frac{1}{4}})$ as $t \to \infty$. It holds $\lambda = - \frac{\int_{\mathbb{S}^1} \kappa^3 ds}{\mathcal{L}(\gamma_0)} = 0$ for the lemniscate by its symmetry properties. Therefore, the area-constrained elastic flow behaves in the same way. Since $\mathcal{A}(\gamma_0) = 0$, scaling $\gamma_0$ does not alter its algebraic area. The gradient flow for this initial curve is shown in \Cref{fig:lemniscate}, with curve lengths provided in \Cref{tab:lengths_lemniscate} to highlight the growth behavior. Moreover, in \Cref{tab:growth_lemniscate} we observe the same scaling behavior $t^{\frac{1}{4}}$ as for the free elastic flow.

\clearpage

\begin{figure}[!t]
    \centering
    \begin{subfigure}[b]{0.48\textwidth}
        \centering
        \includegraphics[width=\textwidth]{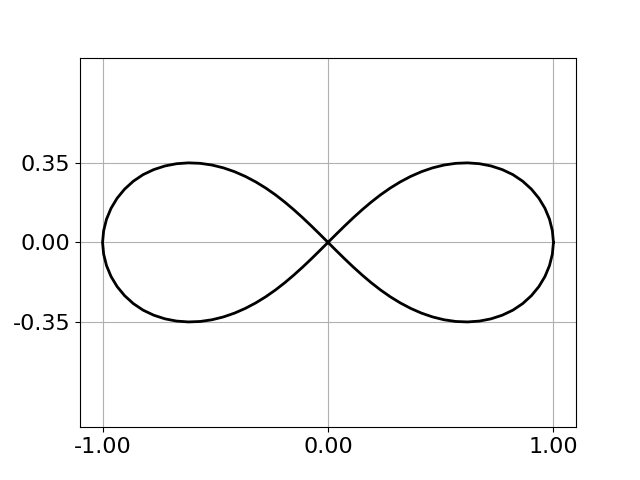}
        \caption{$t=0$} % Iteration $0$
    \label{fig:lemniscate1}
    \end{subfigure}
    \hfill
    \begin{subfigure}[b]{0.48\textwidth}
        \centering
        \includegraphics[width=\textwidth]{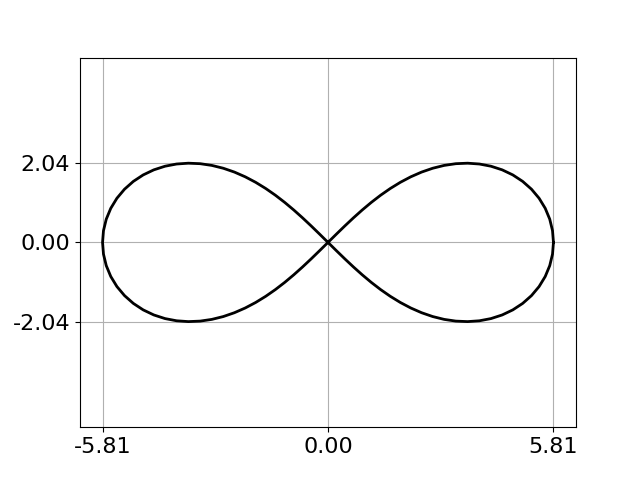}
        \caption{$t=19.74$} % Iteration $10000$
        \label{fig:lemniscate2}
    \end{subfigure}
    \caption{Lemniscate of Bernoulli $\gamma_0(x) = \frac{1}{1+\sin^2 (x)}(\cos(x), \frac{1}{2}\sin(2x))$.}
\label{fig:lemniscate}
\end{figure}

\begin{table}[!t]
    \centering
    \pgfplotstableread[col sep=comma]{tables/lengths_lemniscate_rounded.csv}\mydata
    \pgfplotstabletranspose[colnames from=Time]\transposeddata{\mydata}
    \pgfplotstabletypeset[
        display columns/0/.style={
            column name={$t$},
            column type=|l|,
            string type,
            string replace={Length}{$\mathcal{L}(\gamma_t)$}
        },
        every column/.style={
            column type=l| 
        },
        every head row/.style={
            before row=\hline,
            after row=\hline
        },
        every last row/.style={after row=\hline},
        every row/.style={after row=\hline},
    ]{\transposeddata}
    \caption{Lengths of lemniscate of Bernoulli.}
    \label{tab:lengths_lemniscate}
    \vspace{1em}
    \begin{filecontents*}{tables/growth_lemniscate_rounded.csv}
Time 1, Time 2, Length 1, Length 2, Ratio
7.9,1.97,24.27,17.41,3.78
15.79,3.95,28.78,20.52,3.87
19.74,4.93,30.41,21.65,3.89
\end{filecontents*}
\pgfplotstabletypeset[
        col sep=comma,
        columns/Time 1/.style={
            column name={$t_1$},
            column type=|l| % Vertical line and left-aligned
        },
        columns/Time 2/.style={
            column name={$t_2$},
            column type=l| % Vertical line and left-aligned
        },
        columns/Length 1/.style={
            column name={$\mathcal{L}(\gamma_{t_1})$},
            column type=l| % Vertical line and left-aligned
        },
        columns/Length 2/.style={
            column name={$\mathcal{L}(\gamma_{t_2})$},
            column type=l| % Vertical line and left-aligned
        },
        columns/Ratio/.style={
            column name={$\Big(\frac{\mathcal{L}(\gamma_{t_1})}{\mathcal{L}(\gamma_{t_2})}\Big)^4$},
            column type=l| % Vertical line and left-aligned
        },
        every head row/.style={
            before row=\hline,
            after row=\hline
        },
        every last row/.style={after row=\hline},
        every row/.style={after row=\hline},
    ]{tables/growth_lemniscate_rounded.csv}
    \caption{Lengths of the flow for different times $t_1,t_2$. The ratio $\Big(\frac{\mathcal{L}(\gamma_{t_1})}{\mathcal{L}(\gamma_{t_2})}\Big)^4$ highlights the expected length growth of $t^{\frac{1}{4}}$ and approximates the ratio of times $\frac{t_1}{t_2}=4$.}
\label{tab:growth_lemniscate}
\end{table}

\begin{figure}[!htbp]
    \centering
    \begin{subfigure}[b]{0.48\textwidth}
        \centering
        \includegraphics[width=\textwidth]{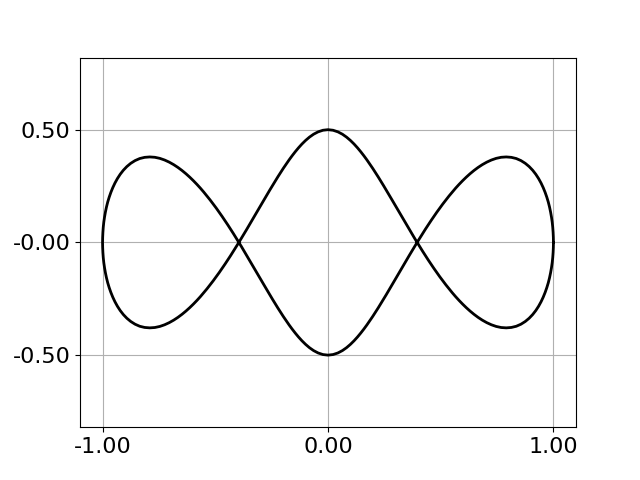}
        \caption{$t=0$} % Iteration 0
        \label{fig:triple-eight1}
    \end{subfigure}
    \hfill
    \begin{subfigure}[b]{0.48\textwidth}
        \centering
        \includegraphics[width=\textwidth]{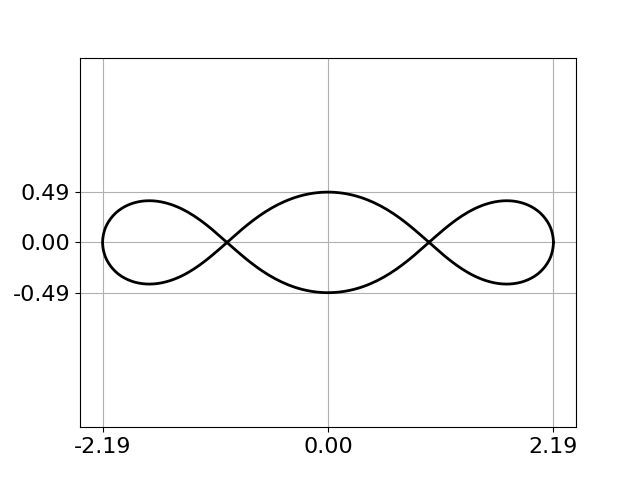}
        \caption{$t=0.11$} % Iteration 500
        \label{fig:triple-eight2}
    \end{subfigure}
    \vspace{1em}
    \begin{subfigure}[b]{0.48\textwidth}
        \centering
        \includegraphics[width=\textwidth]{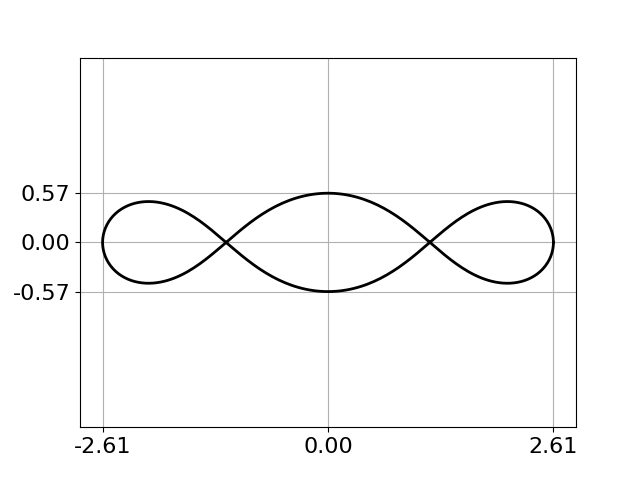}
        \caption{$t=0.22$} % Iteration 1000
        \label{fig:triple-eight3}
    \end{subfigure}
    \hfill
    \begin{subfigure}[b]{0.48\textwidth}
        \centering
        \includegraphics[width=\textwidth]{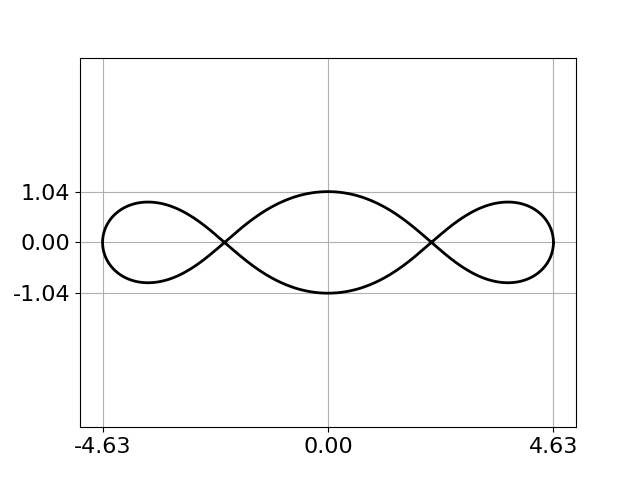}
        \caption{$t=2.19$} % Iteration 10000
        \label{fig:triple-eight4}
    \end{subfigure}
    \caption{Triple-eight shape $\gamma_0(x) = \Big(\frac{\cos(x)}{\sin^2(x) + 1}, \frac{\cos(x) \sin(2x)}{\sin^2(x) + 1}-\frac{\sin(x)}{2}\Big)$.}
    \label{fig:triple-eight}
\end{figure}

\textbf{Unbounded length with fixed algebraic area}\\
Finally, we examine a curve $\gamma_0$ with winding number $n_{\gamma_0}=1$, such as the 'triple-eight' $\gamma_0(x) = \Big(\frac{\cos(x)}{\sin^2(x) + 1}, \frac{\cos(x) \sin(2x)}{\sin^2(x) + 1}-\frac{\sin(x)}{2}\Big)$. Here, the algebraic area $\mathcal{A}$ accounts for the signed contributions of each loop. The curve can expand indefinitely while preserving $\mathcal{A}$, if the central 'belly' compensates for the gain in area of the outer loops. This behavior is observed in \Cref{fig:triple-eight} and the evolution of its lengths can be seen in \Cref{tab:lengths_triple-eight}, while the scaling behavior can be observed in \Cref{tab:growth_triple_eight}.

\begin{table}[!htbp]
    \centering
    \pgfplotstableread[col sep=comma]{tables/lengths_triple_eight_rounded.csv}\mydata
    \pgfplotstabletranspose[colnames from=Time]\transposeddata{\mydata}
    \pgfplotstabletypeset[
        display columns/0/.style={
            column name={$t$},
            column type=|l|,
            string type,
            string replace={Length}{$\mathcal{L}(\gamma_t)$}
        },
        every column/.style={
            column type=l| 
        },
        every head row/.style={
            before row=\hline,
            after row=\hline
        },
        every last row/.style={after row=\hline},
        every row/.style={after row=\hline},
    ]{\transposeddata}
    \caption{Lengths of 'triple-eight' shaped curves.}
    \label{tab:lengths_triple-eight}
    \vspace{1em}
    \begin{filecontents*}{tables/growth_triple_eight_rounded.csv}
Time 1, Time 2, Length 1, Length 2, Ratio
0.88,0.22,17.9,12.67,3.98
1.75,0.44,21.32,15.05,4.03
2.19,0.55,22.56,15.91,4.04
\end{filecontents*}
\pgfplotstabletypeset[
        col sep=comma,
        columns/Time 1/.style={
            column name={$t_1$},
            column type=|l| % Vertical line and left-aligned
        },
        columns/Time 2/.style={
            column name={$t_2$},
            column type=l| % Vertical line and left-aligned
        },
        columns/Length 1/.style={
            column name={$\mathcal{L}(\gamma_{t_1})$},
            column type=l| % Vertical line and left-aligned
        },
        columns/Length 2/.style={
            column name={$\mathcal{L}(\gamma_{t_2})$},
            column type=l| % Vertical line and left-aligned
        },
        columns/Ratio/.style={
            column name={$\Big(\frac{\mathcal{L}(\gamma_{t_1})}{\mathcal{L}(\gamma_{t_2})}\Big)^4$},
            column type=l| % Vertical line and left-aligned
        },
        every head row/.style={
            before row=\hline,
            after row=\hline
        },
        every last row/.style={after row=\hline},
        every row/.style={after row=\hline},
    ]{tables/growth_triple_eight_rounded.csv}
    \caption{Lengths of the flow for different times $t_1,t_2$. The ratio $\Big(\frac{\mathcal{L}(\gamma_{t_1})}{\mathcal{L}(\gamma_{t_2})}\Big)^4$ highlights the expected length growth of $t^{\frac{1}{4}}$ and approximates the ratio of times $\frac{t_1}{t_2}=4$.}
\label{tab:growth_triple_eight}
\end{table}

\clearpage

\textbf{Convergence to the circle for nonsimple curve}\\
To highlight the importance of the convergence result \Cref{cor:energy_profile_convergence}, we present a flow starting from a nonsimple initial curve \Cref{fig:triple-eight20}, which converges to a circle. Consequently, there may exist a broader regime of guaranteed convergence than that specified in \Cref{thm:convergence}, which also encompasses nonsimple curves. This suggests that the conditions stated in \Cref{thm:convergence} are not sharp.

\begin{figure}[!htbp]
    \centering
    \begin{subfigure}[b]{0.24\textwidth}
        \centering
        \includegraphics[width=\textwidth]{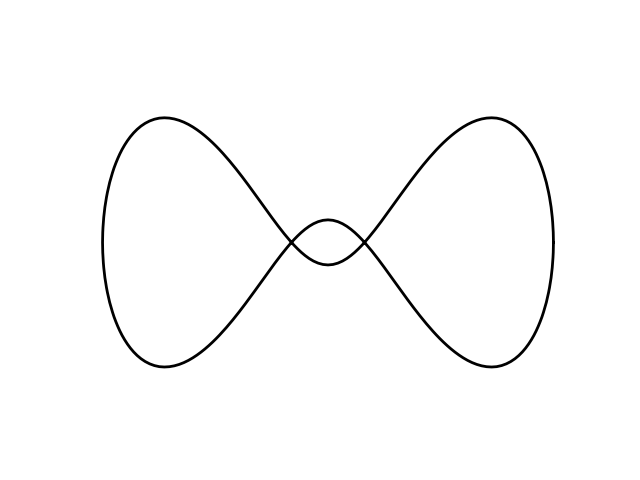}
        \caption{$t=0$} % Iteration 0
    \label{fig:triple-eight21}
    \end{subfigure}
    \hfill
    \centering
    \begin{subfigure}[b]{0.24\textwidth}
        \centering
        \includegraphics[width=\textwidth]{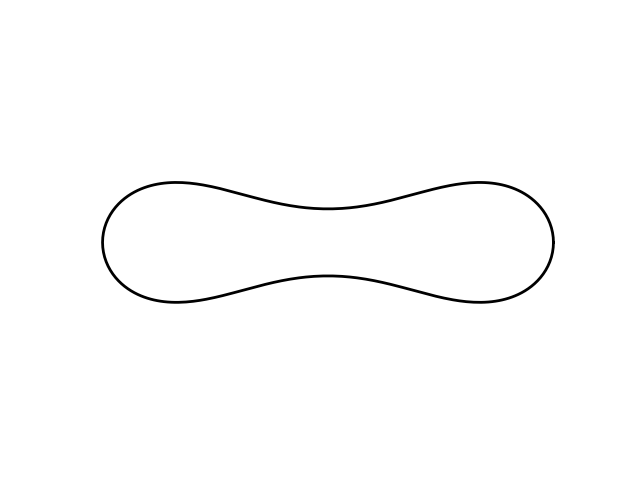}
        \caption{$t=0.01$} % Iteration 50
    \label{fig:triple-eight22}
    \end{subfigure}
    \hfill
    \begin{subfigure}[b]{0.24\textwidth}
        \centering
        \includegraphics[width=\textwidth]{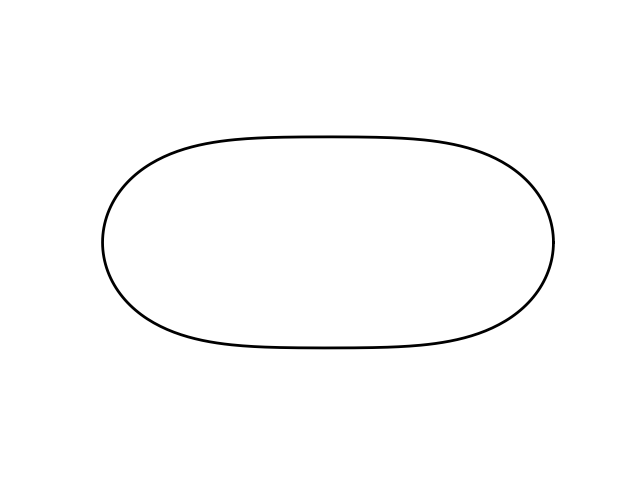}
        \caption{$t=0.04$} % Iteration 200
    \label{fig:triple-eight23}
    \end{subfigure}
    \hfill
    \begin{subfigure}[b]{0.24\textwidth}
        \centering
        \includegraphics[width=\textwidth]{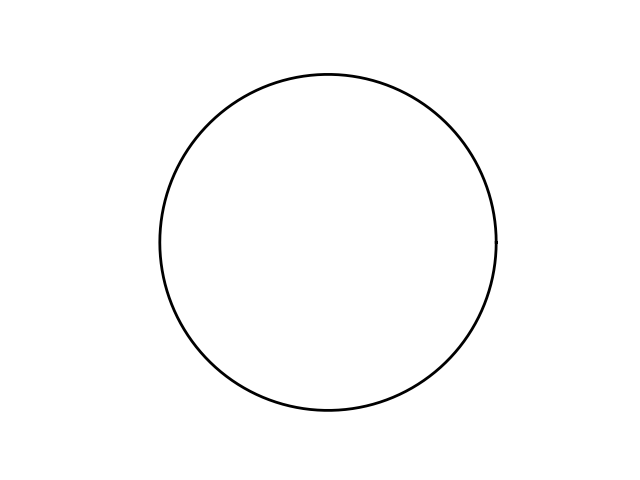}
        \caption{$t=0.11$} % Iteration 500
    \label{fig:triple-eight24}
    \end{subfigure}
    \caption{Nonsimple curve $\gamma_0(x) = \Big(\frac{\cos(x)}{\sin^2(x) + 1}, \frac{\cos(x) \sin(2x)}{\sin^2(x) + 1}-\frac{\sin(x)}{10}\Big)$.}
\label{fig:triple-eight20}
\end{figure}

\section*{Acknowledgments}
This article was submitted as the first author's master thesis at the University of Bonn. The authors want to thank Stefan Müller for his interest in the work and
for reviewing the thesis.

\begin{appendix}
\section{Derivatives of operators}
\label{sec:app_derivative_F}

In the following, the derivative of the operator $\mathcal{F}$ from \cref{eq:definition_F} is presented. Recall that it was given by
\begin{equation*}
\begin{split}
    \mathcal{F}(\gamma) &= -2\frac{\partial_x^{4}\gamma}{\abs{\partial_x \gamma}^{4}} + 12 \frac{\partial_{x}^{2} \gamma \cdot \partial_x \gamma}{\abs{\partial_x \gamma}^{6}}\partial_x^{3}\gamma + 8 \frac{\partial_x^{3}\gamma \cdot \partial_x \gamma}{\abs{\partial_x \gamma}^{6}}\partial_x^{2}\gamma\\
    &\quad+ 5 \frac{\abs{\partial_x^{2} \gamma}^{2}}{\abs{\partial_x \gamma}^{6}} \partial_x^{2} \gamma - 35 \frac{(\partial_x^{2}\gamma \cdot \partial_x \gamma)^{2}}{\abs{\partial_x \gamma}^{8}} \partial_x^{2}\gamma.
\end{split}
\end{equation*}
We compute the derivative of the summands separately:

\begin{equation}
\label{eq:computation_derivative_F}
\begin{aligned}
    &\quad\frac{d}{dl}\bigg|_{l=0} \frac{\partial_x^4(\gamma + lv)}{|\partial_x (\gamma + lv)|^4} = \frac{\partial_x^4 v}{|\partial_x \gamma|^4} - 4 \frac{(\partial_x \gamma \cdot \partial_x v)}{|\partial_x \gamma|^6} \partial_x^4 \gamma;\\[1ex]
    &\quad\frac{d}{dl}\bigg|_{l=0} \frac{(\partial_x^2(\gamma + lv) \cdot \partial_x(\gamma + lv))}{|\partial_x (\gamma + lv)|^6} \partial_x^3(\gamma + lv)\\
    &= \frac{(\partial_x^2 v \cdot \partial_x \gamma) + (\partial_x^2 \gamma \cdot \partial_x v)}{|\partial_x \gamma|^6} \partial_x^3 \gamma - 6 \frac{(\partial_x^2 \gamma \cdot \partial_x \gamma)(\partial_x \gamma \cdot \partial_x v)}{|\partial_x \gamma|^8} \partial_x^3 \gamma\\
    &\quad+ \frac{(\partial_x^2 \gamma \cdot \partial_x \gamma)}{|\partial_x \gamma|^6} \partial_x^3 v;\\[1ex]
    &\quad\frac{d}{dl}\bigg|_{l=0} \frac{(\partial_x^3(\gamma + lv) \cdot \partial_x(\gamma + lv))}{|\partial_x (\gamma + lv)|^6} \partial_x^2(\gamma + lv)\\
    &= \frac{(\partial_x^3 v \cdot \partial_x \gamma) + (\partial_x^3 \gamma \cdot \partial_x v)}{|\partial_x \gamma|^6} \partial_x^2 \gamma - 6 \frac{(\partial_x^3 \gamma \cdot \partial_x \gamma)(\partial_x \gamma \cdot \partial_x v)}{|\partial_x \gamma|^8} \partial_x^2 \gamma \\
    &\quad + \frac{(\partial_x^3 \gamma \cdot \partial_x \gamma)}{|\partial_x \gamma|^6} \partial_x^2 v;\\[1ex]
    &\quad\frac{d}{dl}\bigg|_{l=0} \frac{|\partial_x^2 (\gamma + lv)|^2}{|\partial_x (\gamma + lv)|^6} \partial_x^2(\gamma + lv) = 2 \frac{(\partial_x^2 \gamma \cdot \partial_x^2 v)}{|\partial_x \gamma|^6} \partial_x^2 \gamma - 6 \frac{|\partial_x^2 \gamma|^2 (\partial_x \gamma \cdot \partial_x v)}{|\partial_x \gamma|^8} \partial_x^2 \gamma\\
    &\quad+ \frac{|\partial_x^2 \gamma|^2}{|\partial_x \gamma|^6} \partial_x^2 v;\\[1ex]
    &\quad\frac{d}{dl}\bigg|_{l=0} \frac{(\partial_x^2(\gamma + lv) \cdot \partial_x(\gamma + lv))^2}{|\partial_x (\gamma + lv)|^8} \partial_x^2(\gamma + lv)\\
    &= 2 \frac{(\partial_x^2 \gamma \cdot \partial_x \gamma)((\partial_x^2 v \cdot \partial_x \gamma) + (\partial_x^2 \gamma \cdot \partial_x v))}{|\partial_x \gamma|^8} \partial_x^2 \gamma - 8 \frac{(\partial_x^2 \gamma \cdot \partial_x \gamma)^2 (\partial_x \gamma \cdot \partial_x v)}{|\partial_x \gamma|^{10}} \partial_x^2 \gamma \\
    &\quad + \frac{(\partial_x^2 \gamma \cdot \partial_x \gamma)^2}{|\partial_x \gamma|^8} \partial_x^2 v.
\end{aligned}
\end{equation}

These are all linear expressions in $v$ and hence we can write the derivative in the claimed form.

\label{sec:app_derivative_G}
Let us compute the derivative of $\mathcal{G}$ from \cref{eq:definition_G}. Recall that it is given by
\begin{equation*}
    \mathcal{G}(\gamma) = \frac{\int_{\mathbb{S}^1} \kappa^3 ds}{\int_{\mathbb{S}^1}ds} \nu.
\end{equation*}

Recall $\nu = \frac{1}{\abs{\partial_x \gamma}}\begin{pmatrix}
    -\partial_x \gamma_2\\
    \partial_x \gamma_1
\end{pmatrix}$. Moreover, we write $\kappa = (\kappa \nu)\cdot \nu$ and with \Cref{prop:a_explicit_formulation}, we can compute the derivative of $\int_{\mathbb{S}^1} \kappa^3 ds$ by

\begin{equation*}
\begin{split}
    &\frac{d}{dl}\bigg|_{l=0} \int_{\mathbb{S}^1} \Bigg(\bigg(\frac{\partial_x^2 (\gamma+lv)}{\abs{\partial_x (\gamma+lv)}^2} - \frac{\partial_x^2 (\gamma+lv) \cdot \partial_x (\gamma+lv)}{\abs{\partial_x (\gamma+lv)}^4} \partial_x (\gamma+lv)\bigg)\\
    &\quad\cdot\frac{1}{\abs{\partial_x (\gamma+lv)}}\begin{pmatrix}
        -\partial_x (\gamma+lv)_2\\
        \partial_x (\gamma+lv)_1
    \end{pmatrix}\Bigg)^3 \abs{\partial_x(\gamma+lv)} dx\\
    &= \int_{\mathbb{S}^1} 3\Bigg(\bigg(\frac{\partial_x^2v}{\abs{\partial_x \gamma}^2} - 2\frac{(\partial_x \gamma \cdot \partial_x v)}{\abs{\partial_x \gamma}^4}\partial_x^2 \gamma-\frac{(\partial_x^2\gamma \cdot \partial_x v) + (\partial_x^2 v \cdot \partial_x \gamma)}{\abs{\partial_x \gamma}^4}\partial_x \gamma\\
    &\quad + 4\frac{(\partial_x^2 \gamma \cdot \partial_x \gamma)(\partial_x \gamma \cdot \partial_x v)}{\abs{\partial_x \gamma}^6} \partial_x \gamma - \frac{(\partial_x^2 \gamma \cdot \partial_x \gamma)}{\abs{\partial_x \gamma}^4} \partial_x v\bigg)\cdot \frac{1}{\abs{\partial_x \gamma}} \begin{pmatrix}
        -\partial_x \gamma_2\\
        \partial_x \gamma_1
    \end{pmatrix}\\
    &\quad+ \bigg(\frac{\partial_x^2 \gamma}{\abs{\partial_x \gamma}^2} - \frac{\partial_x^2 \gamma \cdot \partial_x \gamma}{\abs{\partial_x \gamma}^4} \partial_x \gamma\bigg) \cdot \bigg(-\frac{(\partial_x \gamma \cdot \partial_x v)}{\abs{\partial_x \gamma}^3}\begin{pmatrix}
        -\partial_x \gamma_2\\
        \partial_x \gamma_1
    \end{pmatrix}\\
    &\quad+ \frac{1}{\abs{\partial_x \gamma}} \begin{pmatrix}
        -\partial_x v_2\\
        \partial_x v_1
    \end{pmatrix}\bigg) \Bigg)\\
    &\quad\Bigg(\big(\frac{\partial_x^2 \gamma}{\abs{\partial_x \gamma}^2} - \frac{\partial_x^2 \gamma \cdot \partial_x \gamma}{\abs{\partial_x \gamma}^4} \partial_x \gamma\big)\cdot\frac{1}{\abs{\partial_x \gamma}}\begin{pmatrix}
        -\partial_x \gamma_2\\
        \partial_x \gamma_1
    \end{pmatrix}\Bigg)^2 \abs{\partial_x \gamma} dx\\
    &\quad+ \int_{\mathbb{S}^1} \bigg(\big(\frac{\partial_x^2 \gamma}{\abs{\partial_x \gamma}^2} - \frac{\partial_x^2 \gamma \cdot \partial_x \gamma}{\abs{\partial_x \gamma}^4} \partial_x \gamma\big)\cdot\frac{1}{\abs{\partial_x \gamma}}\begin{pmatrix}
        -\partial_x \gamma_2\\
        \partial_x \gamma_1
    \end{pmatrix}\bigg)^3\frac{(\partial_x \gamma \cdot \partial_x v)}{\abs{\partial_x \gamma}} dx.
\end{split}
\end{equation*}

The derivative of $\frac{1}{\int_{\mathbb{S}^1} ds}$ is given by
\begin{equation*}
    \frac{d}{dl}\bigg|_{l=0} \frac{1}{\int_{\mathbb{S}^1} \abs{\partial_x (\gamma+lv)}dx} = -\frac{\int_{\mathbb{S}^1} \frac{(\partial_x \gamma \cdot \partial_x v)}{\abs{\partial_x \gamma}} dx}{\big(\int_{\mathbb{S}^1} \abs{\partial_x \gamma} dx\big)^2}.
\end{equation*}

Lastly, the derivative of $\nu$ is given by

\begin{equation*}
\begin{split}
    \frac{d}{dl}\bigg|_{l=0} \frac{1}{\abs{\partial_x(\gamma+lv)}}\begin{pmatrix}
        -\partial_x (\gamma+lv)_2\\
        \partial_x (\gamma+lv)_1
    \end{pmatrix} &= -\frac{(\partial_x \gamma \cdot \partial_x v)}{\abs{\partial_x \gamma}^3}\begin{pmatrix}
        -\partial_x \gamma_2\\
        \partial_x \gamma_1
    \end{pmatrix}\\
    &\quad+ \frac{1}{\abs{\partial_x \gamma}} \begin{pmatrix}
        -\partial_x v_2\\
        \partial_x v_1
    \end{pmatrix}.
\end{split}
\end{equation*}

Since Hölder spaces are algebras and we can pull the Hölder-norm inside the integral, while $\mathbb{S}^1$ has finite measure, the derivative can be estimated by
\begin{equation}
\label{eq:estimate_norm_derivative_G}
    \norm{D\mathcal{G}[\gamma]v}_{H^{\frac{\alpha}{4}, \alpha}} \leq c(\abs{\partial_x \gamma}^{-1}, \norm{\gamma}_{H^{\frac{4+\alpha}{4}, 4+\alpha}})(\norm{\partial_x v}_{H^{\frac{\alpha}{4},\alpha}}+\norm{\partial_x^2 v}_{H^{\frac{\alpha}{4}, \alpha}})
\end{equation}
as claimed.

\end{appendix}

\bibliographystyle{abbrv}
\bibliography{references}

\begin{thebibliography}{10}

\bibitem{MR1071170}
H.~Amann.
\newblock {\em Ordinary differential equations}, volume~13 of {\em De Gruyter
  Studies in Mathematics}.
\newblock Walter de Gruyter \& Co., Berlin, 1990.
\newblock An introduction to nonlinear analysis, Translated from the German by
  Gerhard Metzen.

\bibitem{andrews2002}
B.~Andrews.
\newblock Classification of limiting shapes for isotropic curve flows.
\newblock {\em J. Amer. Math. Soc.}, 16(2):443--459, 2003.

\bibitem{brezis2011functional}
H.~Brezis.
\newblock {\em Functional analysis, {S}obolev spaces and partial differential
  equations}.
\newblock Universitext. Springer, New York, 2011.

\bibitem{bucur2014newisoperimetricinequalityelasticae}
D.~Bucur and A.~Henrot.
\newblock A new isoperimetric inequality for elasticae.
\newblock {\em J. Eur. Math. Soc. (JEMS)}, 19(11):3355--3376, 2017.

\bibitem{Dall_Acqua_2016}
A.~Dall'Acqua, P.~Pozzi, and A.~Spener.
\newblock The \l ojasiewicz-{S}imon gradient inequality for open elastic
  curves.
\newblock {\em J. Differential Equations}, 261(3):2168--2209, 2016.

\bibitem{dallacqua2017elasticflowcurveshyperbolic}
A.~Dall'Acqua and A.~Spener.
\newblock The elastic flow of curves in the hyperbolic plane, 2017.

\bibitem{Doria2021}
C.~M. Doria.
\newblock {\em Differentiability in {B}anach spaces, differential forms and
  applications}.
\newblock Springer, Cham, [2021] \copyright 2021.

\bibitem{Dziuk2002}
G.~Dziuk, E.~Kuwert, and R.~Sch\"atzle.
\newblock Evolution of elastic curves in {$\mathbb{R}^n$}: existence and
  computation.
\newblock {\em SIAM J. Math. Anal.}, 33(5):1228--1245, 2002.

\bibitem{Eidelman1998}
S.~D. Eidelman and N.~V. Zhitarashu.
\newblock {\em Parabolic boundary value problems}, volume 101 of {\em Operator
  Theory: Advances and Applications}.
\newblock Birkh\"auser Verlag, Basel, 1998.
\newblock Translated from the Russian original by Gennady Pasechnik and Andrei
  Iacob.

\bibitem{ferone2015elasticaproblemareaconstraint}
V.~Ferone, B.~Kawohl, and C.~Nitsch.
\newblock The elastica problem under area constraint.
\newblock {\em Math. Ann.}, 365(3-4):987--1015, 2016.

\bibitem{gerhardt2006curvature}
C.~Gerhardt.
\newblock {\em Curvature problems}, volume~39 of {\em Series in Geometry and
  Topology}.
\newblock International Press, Somerville, MA, 2006.

\bibitem{ioffe}
A.~D. Ioffe.
\newblock On lower semicontinuity of integral functionals. {I}.
\newblock {\em SIAM J. Control Optim.}, 15(4):521--538, 1977.

\bibitem{MR1882174}
W.~K\"uhnel.
\newblock {\em Differential geometry}, volume~16 of {\em Student Mathematical
  Library}.
\newblock American Mathematical Society, Providence, RI, 2002.
\newblock Curves---surfaces---manifolds, Translated from the 1999 German
  original by Bruce Hunt.

\bibitem{LANGER198575}
J.~Langer and D.~A. Singer.
\newblock Curve straightening and a minimax argument for closed elastic curves.
\newblock {\em Topology}, 24(1):75--88, 1985.

\bibitem{lee2003smooth}
J.~M. Lee.
\newblock {\em Introduction to smooth manifolds}, volume 218 of {\em Graduate
  Texts in Mathematics}.
\newblock Springer, New York, second edition, 2013.

\bibitem{Miura_2025}
T.~Miura and G.~Wheeler.
\newblock The free elastic flow for closed planar curves.
\newblock {\em J. Funct. Anal.}, 289(7):Paper No. 111030, 22, 2025.

\bibitem{Mueller_2021}
M.~M\"uller and F.~Rupp.
\newblock A {L}i-{Y}au inequality for the 1-dimensional {W}illmore energy.
\newblock {\em Adv. Calc. Var.}, 16(2):337--362, 2023.

\bibitem{MinoruMurai2013ConferencePublications}
M.~Murai, W.~Matsumoto, and S.~Yotsutani.
\newblock Representation formula for the plane closed elastic curves.
\newblock {\em Discrete Contin. Dyn. Syst.}, pages 565--585, 2013.

\bibitem{okabe2007}
S.~Okabe.
\newblock The motion of elastic planar closed curves under the area-preserving
  condition.
\newblock {\em Indiana Univ. Math. J.}, 56(4):1871--1912, 2007.

\bibitem{osserman1978isoperimetric}
R.~Osserman.
\newblock The isoperimetric inequality.
\newblock {\em Bull. Amer. Math. Soc.}, 84(6):1182--1238, 1978.

\bibitem{POZZETTA2022112581}
M.~Pozzetta.
\newblock Convergence of elastic flows of curves into manifolds.
\newblock {\em Nonlinear Anal.}, 214:Paper No. 112581, 53, 2022.

\bibitem{RUPP2020108708}
F.~Rupp.
\newblock On the \l ojasiewicz-{S}imon gradient inequality on submanifolds.
\newblock {\em J. Funct. Anal.}, 279(8):108708, 33, 2020.

\bibitem{Rupp_Phd}
F.~Rupp.
\newblock {\em Constrained gradient flows for Willmore-type functionals}.
\newblock PhD thesis, Universität Ulm, 2022.

\bibitem{Rupp2024}
F.~Rupp and A.~Spener.
\newblock Existence and convergence of the length-preserving elastic flow of
  clamped curves.
\newblock {\em J. Evol. Equ.}, 24(3):Paper No. 59, 41, 2024.

\bibitem{seeley1964extension}
R.~T. Seeley.
\newblock Extension of {$C\sp{\infty }$} functions defined in a half space.
\newblock {\em Proc. Amer. Math. Soc.}, 15:625--626, 1964.

\bibitem{Spener2017}
A.~Spener.
\newblock Short time existence for the elastic flow of clamped curves.
\newblock {\em Math. Nachr.}, 290(13):2052--2077, 2017.

\bibitem{Watanabe2002}
K.~Watanabe.
\newblock Plane domains which are spectrally determined.
\newblock {\em Ann. Global Anal. Geom.}, 18(5):447--475, 2000.

\end{thebibliography}
\end{document}